\documentclass[11pt]{amsart}
\usepackage{amsmath}
\usepackage{amssymb}
\usepackage{amsthm}
\usepackage{latexsym}
\usepackage{graphicx}
\usepackage{hyperref}
\usepackage{enumerate}
\usepackage[all]{xy}

\newtheorem{thm}{Theorem}[section]
\newtheorem{cor}[thm]{Corollary}
\newtheorem{lem}[thm]{Lemma}
\newtheorem{prop}[thm]{Proposition}

\newtheorem{thmintro}{Theorem}

\providecommand{\norm}[1]{\left\| #1 \right\|}
\newcommand{\enuma}[1]{\begin{enumerate}[\textup{(}a\textup{)}] {#1} \end{enumerate}}

\newcommand{\mh}{\mathbb}
\newcommand{\mr}{\mathrm}
\newcommand{\mc}{\mathcal}
\newcommand{\mf}{\mathfrak}
\newcommand{\ds}{\displaystyle}

\newcommand{\ts}{\textstyle}

\newcommand{\N}{\mathbb N}
\newcommand{\Z}{\mathbb Z}

\newcommand{\R}{\mathbb R}
\newcommand{\C}{\mathbb C}
\newcommand{\ep}{\epsilon}

\newcommand{\af}{\mr{aff}}

\newcommand{\wig}{\textstyle \bigwedge}
\newcommand{\inp}[2]{\langle #1 \,,\, #2 \rangle}

\newcommand{\Fs}{\mathcal F} 

\newcommand{\DNO}{DN \Omega} 
\newcommand{\Fxi}{\Phi_\xi^*} 

\begin{document}

\title{Extensions of tempered representations}

\author{Eric Opdam}
\address{Korteweg-de Vries Institute for Mathematics\\
Universiteit van Amsterdam\\
Science Park 904\\
1098 XH Amsterdam\\
The Netherlands}
\email{e.m.opdam@uva.nl}
\author{Maarten Solleveld}
\address{Institute for Mathematics, Astrophysiscs and Particle Physiscs\\
Radboud Universiteit Nijmegen\\
Heyendaalseweg 135\\
6525AJ Nijmegen\\
The Netherlands}
\email{m.solleveld@science.ru.nl}
\date{\today}

\maketitle

\begin{abstract}
Let $\pi, \pi'$ be irreducible tempered representations of an
affine Hecke algebra $\mathcal{H}$ with positive parameters.
We compute the higher extension groups $Ext_\mathcal{H}^n (\pi,\pi')$
explicitly in terms of the representations of analytic R-groups
corresponding to $\pi$ and $\pi'$.
The result has immediate applications to the computation of the
Euler--Poincar\'e pairing $EP (\pi,\pi')$, the alternating sum of
the dimensions of the Ext-groups. The resulting formula for $EP(\pi,\pi')$
is equal to Arthur's formula for the elliptic pairing of tempered
characters in the setting of reductive $p$-adic groups.
Our proof applies equally well to affine Hecke algebras and to
reductive groups over non-archimedean local fields of arbitrary
characteristic. This sheds new light on the formula of Arthur
and gives a new proof of Kazhdan's orthogonality conjecture for
the Euler-Poincar\'e pairing of admissible characters.
\end{abstract}

\tableofcontents

\section*{Introduction}
Let $\mathbb F$ be a non-archimedean local field, and let $L$ be
the group of $\mathbb F$-rational points of a connected reductive
algebraic group defined over $\mathbb F$. In this paper we will
compute the vector space of higher extensions $\textup{Ext}_{L}^n (V,V')$
between smooth irreducible tempered representations $V, V'$
of $L$ in the abelian category of smooth representations of $L$.
In the formulation of the result a predominant role is played by the so called
\emph{analytic R-groups} which, by classical results due to Harish-Chandra,
Knapp, Stein and Silberger, are fundamental in the classification of
smooth irreducible tempered representations \cite{HC,KnSt,Sil}. The result applies
to the computation of the Euler-Poincar\'e pairing for admissible tempered
representations of $L$. In particular, this leads to a new proof for the
``homological version'' of Arthur's formula for the elliptic pairing of tempered characters
of $L$ \cite{Art,Ree} and to a new proof for Kazhdan's othogonality conjecture
\cite{Kaz} for the elliptic pairing of general admissible characters of $L$
(which was previously proved by Bezrukavnikov \cite{Bez} and independently
by Schneider and Stuhler \cite{ScSt}).

The first and crucially important step is a result of Meyer \cite{Mey-Ho}
showing that there exists a natural isomorphism
$\textup{Ext}_{\mc H (L)}^n (V,V')\simeq \textup{Ext}_{\mathcal{S}(L)}^n (V,V')$
where $\mathcal{S}(L)$ denotes the Harish-Chandra Schwartz algebra of $L$. Here 
$\mc H (L)$ and $\mathcal{S}(L)$ are considered as bornological algebras and 
the Ext-groups are defined in categories of bornological modules, see \cite{Mey-Emb}.
Recently the authors found an alternative proof of this result without
the use of bornologies \cite{OpSo3}.

The Schwartz algebra is a direct limit of the Fr\'echet subalgebras
$\mathcal{S}(L,K) = e_K\mathcal{S}(L)e_K$ of $K$-biinvariant functions in $\mathcal{S}(L)$,
where $K$ runs over a system of ``good'' compact open subgroups of $L$.
One can prove by Morita equivalence that $\textup{Ext}_{\mathcal{S}(L)}^i(V,V')\simeq
\textup{Ext}_{\mathcal{S}(L,K)}^i(V^K,{V'}^K)$ if $V$ is generated by its
$K$-invariant vectors (see Section \ref{sec:padic}).
Now we use the structure of $\mathcal{S}(L,K)$ provided by
Harish-Chandra's Fourier isomorphism \cite{Wal} in order to compute
the right hand side.
More precisely, we take the formal completion of the algebra
$\mathcal{S}(L,K)$ at the central character of $V^K$, and show that this is
Morita equivalent to a (twisted) crossed product of a formal power series ring
with the analytic R-group. Finally, the Ext-groups of such algebras are easily
computed using a Koszul resolution.

All ingredients necessary for the above line of arguments have been
developed in detail in the context of abstract affine Hecke algebras as well
\cite{DeOp1,DeOp2,OpSo1}.
The analytic R-groups are defined in terms of the Weyl group and
the Plancherel density, all of which allow for explicit determination \cite{Slo}.
We will use the case of abstract affine Hecke algebras as the point of reference
in this paper. In section \ref{sec:padic} we will carefully
formulate the results for the representation theory of $L$, and discuss the
adaption of the arguments necessary for the proofs of those results.

The main technical difficulties to carry out the above steps arise from
the necessity to control the properties of the formal completion functor in the
context of the Fr\'echet algebras $\mathcal{S}(L,K)$, which are far from Noetherian.
The category of bornological $\mathcal{S}(L,K)$-modules is a Quillen exact category \cite{Qui}
with respect to exact sequences that admit a bounded linear splitting. This implies
(as used implicitly above) that the $\textup{Ext}$-groups over the algebras 
$\mathcal{S}(L,K)$ are well-defined with respect to this exact structure. But the bounded 
linear splittings are not preserved when taking the formal completion at a central
character, and this forces us to make careful choices of the exact categories
with which to work.
\vspace{2mm}

We will now discuss the main results of this paper in more detail. Let us first recall
some basic facts on analytic R-groups. Let $\mc{H}$ denote an abstract affine Hecke algebra
with positive parameters.
By \cite{DeOp1} the category of tempered $\mathcal{H}$-modules of finite length decomposes
into ``blocks'' which are parameterized by the set $\Xi_{un} / \mc W$ of orbits of
tempered standard induction data $\xi\in\Xi_{un}$ for $\mathcal{H}$ under the action
of the Weyl groupoid $\mathcal{W}$ for $\mathcal{H}$. To such a standard tempered
induction datum $\xi\in\Xi_{un}$ one attaches a tempered standard induced module $\pi(\xi)$,
which is a unitary tempered $\mathcal{H}$-module.
The block $\mr{Mod}_{f, \mathcal{W}\xi} (\mc S)$ of tempered modules associated
with $\mathcal{W}\xi$ is generated by $\pi (\xi)$.
The space $\Xi_{un}$ of tempered induction data is a finite disjoint union of compact
tori of various dimensions. In particular, at each $\xi\in\Xi_{un}$ there exists a
well defined tangent space $T_\xi (\Xi_{un})$. With the action
of the Weyl groupoid $\mathcal{W}$ we obtain a smooth orbifold $\Xi_{un}\rtimes\mathcal{W}$.
The isotropy group $\mathcal{W}_\xi \subset \mathcal{W}$ of $\xi$ admits a
canonical decomposition $\mathcal{W}_\xi = W(R_\xi)\rtimes \mathfrak{R}_\xi$.
Here $W(R_\xi)$ is a real reflection group associated to an integral root system $R_\xi$
in $T_\xi (\Xi_{un})$ and $\mathfrak{R}_\xi$ is a group of
diagram automorphisms with respect to a suitable choice of a basis of $R_\xi$.
The subgroup $\mathfrak{R}_\xi$ is called the analytic R-group at $\xi$.
Since $W(R_\xi)$ is a real reflection group, the quotient (in the category of
complex affine varieties) $(T_\xi (\Xi_{un}) \otimes_\R \C ) / W(R_\xi)$ is the
complexification of a real vector space $E_\xi$ which carries a representation
of $\mathfrak{R}_\xi$. The $\mathfrak{R}_\xi$-representation $E_\xi$ is independent
of the choice of $\xi$ in its orbit $\mathcal{W}\xi$, up to equivalence.

The Knapp--Stein linear independence theorem
(Theorem \ref{thm:KnappStein}) for affine Hecke algebras \cite{DeOp2} states that the commutant
of $\pi(\xi)(\mathcal{H})$ in $\operatorname{End}_\mathbb{C}(\pi(\xi))$ is isomorphic
to the twisted group algebra $\mathbb{C}[\mathfrak{R}_\xi,\kappa_\xi]$ of the analytic R-group
$\mathfrak{R}_\xi$, where $\kappa_\xi$ is a certain $2$-cocycle.
This sets up a bijection
\begin{align*}
\mr{Irr}(\mathbb{C}[\mathfrak{R}_\xi,\kappa_\xi])
&\longleftrightarrow \mr{Irr}_{\mathcal{W}\xi}(\mc S)\\
(\rho ,V) & \longleftrightarrow \pi (\xi,\rho)
\end{align*}
between the set of irreducible representations of $\mathbb{C}[\mathfrak{R}_\xi,\kappa_\xi]$
and the set of irreducible objects in $\mr{Mod}_{f, \mathcal{W}\xi} (\mc S)$.
Hence the collection $\pi (\xi,\rho)$ where $\mathcal{W}\xi$ runs
over the set of $\mathcal{W}$-orbits in $\Xi_{un}$ and $\rho$ runs
over the set of irreducible representations of $\mathbb{C}[\mathfrak{R}_\xi,\kappa_\xi]$
is a complete set of representatives for the equivalence
classes of tempered irreducible representations of $\mathcal{H}$.

Let $\mathfrak{R}_\xi^*$ be the Schur extension of $\mathfrak{R}_\xi$
and let $p_\xi \in \mathbb{C}[\mathfrak{R}_\xi^*]$ be the central idempotent 
corresponding to the two-sided ideal
$\mathbb{C}[\mathfrak{R}_\xi,\kappa^{-1}_\xi]\subset\mathbb{C}[\mathfrak{R}_\xi^*]$.
The first main result of this paper is a generalization of the Knapp-Stein Theorem.
It establishes an equivalence between the category of $\mc{S}$-modules with generalized
central character $\mc{W}\xi$ and the module category of the complex algebra
\[
\mathcal{A}_\xi := p_\xi \big( \widehat{S(E_\xi)}\rtimes\mathfrak{R}_\xi^* \big) 
\cong \Big( \widehat{S(E_\xi)}\otimes \mr{End}_\C \big( p_\xi \C [\mathfrak{R}_\xi^*] \big) 
\Big)^{\mf R_\xi^*} .
\]
Here $S (E_\xi)$ denotes the algebra of complex valued polynomial functions on $E^*_\xi$
and $\widehat{S(E_\xi)}$ is its formal completion at $0 \in E^*_\xi$. The group
$\mathfrak{R}_\xi^*$ acts on this space through its natural quotient
$\mathfrak{R}_\xi$. Consider the $\mathcal{A}_\xi$-module
\[
D := \Big( \widehat{S(E_\xi)} \otimes \textup{Hom}_\C
\big( \pi(\xi), p_\xi \C [\mathfrak{R}_\xi^*] \big) \Big)^{\mathfrak{R}_\xi^*} .
\]
The Fourier isomorphism \eqref{eq:FourierIso} from \cite[Theorem 5.3]{DeOp1} implies that
$D$ also admits a natural right $\mathcal{S}$-module structure, turning it into a
$\mathcal{A}_\xi$-$\mathcal{S}$-bimodule. To make the above precise we consider the
categories $\mr{Mod}_{bor}(\mc A_\xi)$ of bornological $\mc{A}_\xi$-modules and
$\mr{Mod}^{\mc W \xi,tor}_{bor}(\mc S)$, which consists of the bornological
$\mc S$-modules that are annihilated by all $\mc W$-invariant smooth functions on $\Xi_{un}$
that are flat in $\xi$. In both cases the bornology is derived from the
Fr\'echet algebra structure. We endow these categories with the exact structure of
module extensions with a bounded linear splitting.

Instead of this bornological machinery, one may also think of Fr\'echet modules and
exact sequences that admit continuous linear splittings. Indeed, the functor
$\mr{Mod}_{Fr}(\mc S) \to \mr{Mod}_{bor}(\mc S)$ that equips a Fr\'echet $\mc S$-module 
with its precompact bornology is a fully exact embedding \cite[Proposition A.2]{OpSo1}.
Hence all our constructions could be viewed purely within the context of 
Frechet $\mc S$-modules and that would yield the same results, restricted  
to this fully exact subcategory of the category of bornological $\mc S$-modules.

\begin{thmintro}\label{thm:0}
(See Theorem \ref{thm:equivCat}.)\\
There is an equivalence of exact categories
\[
\mr{Mod}^{\mc W \xi,tor}_{bor}(\mc S) \to \mr{Mod}_{bor}(\mc A_\xi) ,\;
M \mapsto D\widehat\otimes_{\mc S} M ,
\]
and similarly for the corresponding categories of Fr\'echet modules.
\end{thmintro}

The restriction of this result to irreducible modules essentially recovers the aforementioned
Knapp--Stein Theorem. The proof of Theorem \ref{thm:0} is analytic in nature, some of the
ideas are already present in \cite{Was,LePl}.

Since the algebra $\mc A_\xi$ is a twisted crossed product of a formal power series ring
with $\mathfrak{R}_\xi$, it is not hard to compute the Ext-groups between modules in
$\mr{Mod}_{f, \mathcal{W}\xi} (\mc S)$ using Theorem \ref{thm:0}. But we are more interested in
computing their Ext-groups in the category of $\mc H$-modules or equivalently \cite{Mey-Emb,OpSo1}
in $\mr{Mod}_{bor}(\mc S)$. To go from there to
$\mr{Mod}^{\mc W \xi,tor}_{bor}(\mc S)$ boils down to applying the formal completion
functor at a central character.
In Section \ref{sec:Ext} we show that this completion functor is exact
and preserves Ext-groups in suitable module categories. The underlying reason is that the
Taylor series map from smooth functions to
power series induces an exact functor on finitely generated modules \cite{TouMer}. With
that, Theorem \ref{thm:0} and an explicit Koszul resolution for $\mc A_\xi$ we can calculate
the Ext-groups for irreducible tempered $\mc H$-representations:

\begin{thmintro}\label{thm:1}
(See Theorem \ref{thm:Ext}.) \\
Let $(\rho,V), (\rho',V') \in \textup{Irr}(\mathbb{C}[\mathfrak{R}_\xi,\kappa_\xi])$
and let $n \geq 0$. Then
\begin{equation*}
\mr{Ext}^n_\mathcal{H} \big( \pi (\xi ,\rho),\pi (\xi, \rho') \big) \cong
\big( \mr{Hom}_\mathbb{C}(V',V) \otimes_\R  \wig^n E^*_\xi \big)^{\mathfrak{R}_\xi} .
\end{equation*}
\end{thmintro}

Since $\mc S$-modules with different central characters have only trivial extensions,
Theorem \ref{thm:1} is the essential case for the determination of extensions
between finite dimensional $\mc S$-modules. Our proof also applies in the context of 
tempered representations of reductive $p$-adic groups:

\begin{thmintro}\label{thm:2}
(See Section \ref{sec:padic}.)\\
Let $\mathbf{L}$ be a connected reductive group defined over a
non-archimedean local field $\mathbb F$. Then the analogues of Theorems \ref{thm:0}
and \ref{thm:1} holds for tempered irreducible representations of $L = \mathbf{L}(\mathbb F)$.
\end{thmintro}

These results have obvious consequences for the computation of the
Euler--Poincar\'e pairings.
For any $\C$-linear abelian category $\mathcal{C}$ with
finite homological dimension one may define the \emph{Euler--Poincar\'e pairing}
\cite{ScSt} of two objects of finite length
$\pi,\pi^\prime\in\textup{Obj}(\mathcal{C})$ by the formula
\begin{equation}\label{eq:ep}
EP_{\mc C} (\pi,\pi^\prime) =
\sum_{n \geq 0}(-1)^n \dim_\C \textup{Ext}_{\mathcal{C}}^n (\pi,\pi^\prime)
\end{equation}
provided that the dimensions of all the Ext-groups between irreducible objects are finite.
Let $K_\mathbb{C}(\mathcal{C})$ denote the Grothendieck group
(tensored by $\mathbb{C}$) of finite length objects in $\mathcal{C}$.
The Euler--Poincar\'e pairing extends to a sesquilinear form $EP$ on $K_\C(\mathcal{C})$.

A well known instance of this construction is the Euler--Poincar\'e pairing
on the Grothendieck group $G_\C (L)$ of admissible representations of
$L$ \cite{BeDe,ScSt}. In this case $\mathcal{C}$ is the category of smooth
representations of $L$. The resulting pairing $EP_L$ is Hermitian (\cite{ScSt}; See
Proposition \ref{prop:7.2} for a more direct argument)
and plays a fundamental role in the
local trace formula and in the study of orbital integrals on the regular elliptic
set of $L$ \cite{Art,ScSt,Bez,Ree}.
The definition of $EP$ also applies naturally to the Grothendieck group $G_\C (\mc{H})$
of finite dimensional representations of an abstract affine Hecke algebra $\mc{H}$
with positive parameters. Here one uses that the category of finitely generated
$\mc{H}$-modules has finite cohomological dimension by \cite{OpSo1}.
The form $EP_\mc{H}$ on $G_\C (\mc{H})$ is also Hermitian \cite[Theorem 3.5.a]{OpSo1}.

Theorem \ref{thm:1} implies the following formula for
the Euler--Poincar\'e pairing between irreducible tempered representations
of an affine Hecke algebra $\mathcal{H}$ (see Theorem \ref{thm:ArthurFormula}):
\begin{equation}\label{eq:ArKaScSt}
EP_\mathcal{H}(\pi(\xi,\rho),\pi(\xi,\rho')) =
|\mf R_\xi |^{-1} \sum_{r \in \mf R_\xi} |d(r)| \,
\mr{tr}_\rho (r) \, \overline{\mr{tr}_{\rho'} (r)} .
\end{equation}
where $d(r) = \det_{T_\xi (\Xi_{un})} (1 - r)$.
Similarly Theorem \ref{thm:2} implies the analog of \eqref{eq:ArKaScSt} in the
context of tempered representations of a reductive $p$-adic group $L$.
We formulate this as a theorem since it had not yet been established in this generality:

\begin{thmintro}\label{thm:3}
(See \eqref{eq:7.1}) \\
Let $\mathbf{L}$ be a connected reductive group defined over a
non-archimedean local field $\mathbb F$. Then the analog of
\eqref{eq:ArKaScSt} holds for the Euler-Poincar\'e pairing
of tempered irreducible representations of $L = \mathbf{L}(\mathbb F)$.
\end{thmintro}
In the case of Iwahori-spherical representations of split groups
this result was previously shown by Reeder \cite{Ree} using
different arguments.
If $\textup{char}(\mathbb{F})=0$ then Theorem \ref{thm:3}
can alternatively be derived by combining Arthur's
formula \cite[Corollary 6.3]{Art} for the elliptic pairing of tempered
characters with Kazhdan's
orthogonality conjecture for the Euler--Poincar\'e pairing of
admissible characters of $L$, which was proved independently in
\cite[Theorem III.4.21]{ScSt} and in \cite[Theorem 0.20]{Bez}.
Interestingly, this argument can also be reversed showing that
Theorem \ref{thm:3} gives rise to a new proof of Kazhdan's
orthogonality conjecture. We explain this argument more
precisely now.

Let us recall the
\emph{elliptic pairing} of admissible characters. Suppose that the
ground field $\mathbb F$ of $L$ has characteristic 0 and let
$C^{ell}$ be the set of regular semisimple elliptic conjugacy
classes of $L$. By definition $C^{ell}$ is empty if the center
$Z(L)$ of $L$ is not compact. There exists a canonical elliptic
measure $d\gamma$ (the ``Weyl measure'') on $C^{ell}$ \cite{Kaz}.
Let $\theta_\pi$ denote the locally constant function on $C^{ell}$
determined by the distributional character of an admissible
representation $\pi$. The elliptic pairing of $\pi$ and
$\pi^\prime$ is defined as
\begin{equation}\label{eq:elliptic}
e_L (\pi,\pi^\prime) :=
\int_{C^{ell}}\theta_\pi (c^{-1})\theta_{\pi^\prime}(c) \textup{d} \gamma (c) .
\end{equation}
\emph{Kazhdan's
orthogonality conjecture for the Euler--Poincar\'e pairing of
admissible characters of} $L$ states that, under the assumptions made above,
the elliptic pairing $e_L(\pi,\pi^\prime)$ of two admissible representations $\pi$ and
$\pi^\prime$ of $L$ is equal to their Euler--Poincar\'e pairing:
\begin{equation}\label{eq:Kazhdan}
e_L(\pi,\pi^\prime) = EP_L(\pi,\pi^\prime) .
\end{equation}
Next we indicate how \eqref{eq:Kazhdan} also follows from Theorem \ref{thm:3}.
The connection is provided by results of Arthur on the elliptic pairing
of tempered characters. Let $P\subset L$ be a parabolic subgroup
with Levi component $M$, and let $\sigma$ be a smooth irreducible representation
of $M$, square integrable modulo the center of $M$. The representation
$\mathcal{I}_P^L(\sigma)$, the smooth normalized parabolically induced
representation from $\sigma$, is a tempered admissible unitary representation
of $L$. The decomposition of $\mathcal{I}_P^L(\sigma)$
is governed by the associated analytic R-group $\mf R_\sigma$, a finite
group which acts naturally on the real Lie algebra $\mr{Hom}(X^* (M), \R)$
of the center of $M$. For $r \in \mf R_\sigma$ we denote by $d(r)$
the determinant of the linear transformation $1 - r$ on
$\mr{Hom} (X^* (M), \R)$. We note that $d(r) = 0$ whenever $Z(L)$ is not compact.

Let $\pi$ be an irreducible tempered
representation of $L$ which occurs in $\mathcal{I}_P^L(\sigma)$, which we
denote by $\pi\prec\mathcal{I}_P^L(\sigma)$.
The theory of the analytic R-group asserts that
$\rho=\textup{Hom}_L(\pi,\mathcal{I}_P^L(\sigma))$
is an irreducible projective $\mf R_\sigma$-representation.
It was shown in \cite[Corollary 6.3]{Art} that for all tempered irreducible
representations $\pi,\pi^\prime\prec \mathcal{I}_P^L(\sigma)$ one has
\begin{equation}\label{eq:Arthur}
e_L(\pi,\pi^\prime)=
|\mf R_\sigma|^{-1} \sum_{r \in \mf R_\sigma} |d(r)| \,
\mr{tr}_\rho (r) \, \overline{\mr{tr}_{\rho'} (r)} .
\end{equation}
The proof uses the local trace formula for $L$, which requires that char$(\mathbb F) = 0$.
With our conventions \eqref{eq:Arthur} is trivial when $Z(L)$ is not compact,
Arthur has a more sophisticated equality in that situation.
If $\pi$ and $\pi^\prime$ do not arise (up to equivalence) as components
of the same standard induced representation $\mathcal{I}_P^L(\sigma)$,
then their elliptic pairing is zero.
From \eqref{eq:Arthur} and Theorem \ref{thm:3} it is only a short trip to
a proof of Kazhdan's orthogonality conjecture:

\begin{thmintro}\label{thm:4}
(See Theorem \ref{thm:7.3}.)\\
Let $\mathbf{L}$ be a connected reductive group defined over a
non-archimedean local field $\mathbb{F}$ of characteristic 0.
Then equation \eqref{eq:Kazhdan} holds for all admissible representations
$\pi,\pi^\prime$ of $L = \mathbf{L}(\mathbb F)$.
\end{thmintro}

The elliptic pairing and the set $C^{ell}$ of elliptic conjugacy classes in
\eqref{eq:elliptic} do not seem to have obvious counterparts in the setting of affine
Hecke algebras.
Reeder \cite{Ree} introduced the elliptic pairing for the cross product of a finite
group with a real representation. This construction was extended in
\cite[Theorem 3.3]{OpSo1} to the case of a finite group acting on a lattice.
To relate these notions of elliptic pairing for Weyl groups to the
Euler-Poincar\'e characteristic $EP_{\mc H}$ one compares $EP_{\mc H}$ and $EP_W$,
which is done in \cite[Section 5.6]{Ree} (for affine Hecke algebras with equal parameters)
and in \cite[Chapter 3]{OpSo1}. Recent results
from \cite{Sol-Irr} allow us to conclude that $G_\C (\mc H)$ modulo the radical of $EP_{\mc H}$
equals the vector space Ell$(\mc H)$ of ``elliptic characters", and that this space does not
depend on the Hecke parameters $q$ (Theorem \ref{thm:sigmaEll}).

Arthur's explicit formula \eqref{eq:ArKaScSt} for $EP_\mathcal{H}$
applies only to the Euler-Poincar\'e pairing of tempered characters.
In an abstract sense this no restriction, as it follows from
the Langlands classification of irreducible characters of $\mc{H}$ in terms of standard
induction data in Langlands position, that modulo the radical of the pairing $EP_{\mc H}$
any irreducible character is equivalent to a virtual tempered character. But this is
complicated in practice, and therefore our formula does not qualify as an explicit formula
for $EP_{\mc H}$ for general non-tempered irreducible characters.
It would therefore be desirable to extend the result to general finite dimensional
representations of $\mc{H}$. In Section \ref{sec:Rgroup} we make a first step by
extending the definition of the analytic $R$-group to non-tempered induction data.
We show that its irreducible characters are in natural bijection with the Langlands
quotients associated to the induction datum. However, we have not been able to
generalize the Arthur formula to this case. It seems that one would need an
appropriate version of Kazhdan--Lusztig polynomials to address the problem in
this generality.
\vspace{2mm}

\section{Affine Hecke algebras}

Here we recall the definitions and notations of our most important objects of study.
Several things described in this section can be found in more detail elsewhere in the
literature, see in particular \cite{Lus-Gr,Opd-Sp,OpSo1,Sol-Irr}.

Let $\mf a$ be a finite dimensional real vector space and let $\mf a^*$ be its dual. Let
$Y \subset \mf a$ be a lattice and $X = \mr{Hom}_\Z (Y,\Z) \subset \mf a^*$ the dual lattice. Let
\[
\mc R = (X, R_0, Y ,R_0^\vee ,F_0) .
\]
be a based root datum. Thus $R_0$ is a reduced root system in $X ,\, R^\vee_0 \subset Y$
is the dual root system, $F_0$ is a basis of $R_0$ and the set of positive roots is denoted $R_0^+$.
Furthermore we are given a bijection $R_0 \to R_0^\vee ,\: \alpha \mapsto \alpha^\vee$ such
that $\inp{\alpha}{\alpha^\vee} = 2$ and such that the corresponding reflections
$s_\alpha : X \to X$ (resp. $s^\vee_\alpha : Y \to Y$) stabilize $R_0$ (resp. $R_0^\vee$).
We do not assume that $R_0$ spans $\mf a^*$.

The reflections $s_\alpha$ generate the Weyl group $W_0 = W (R_0)$ of $R_0$, and
$S_0 := \{ s_\alpha : \alpha \in F_0 \}$ is the collection of simple reflections. We have the
affine Weyl group $W^\af = \mh Z R_0 \rtimes W_0$ and the extended (affine) Weyl group
$W = X \rtimes W_0$. Both can be regarded as groups of affine transformations of
$\mf a^*$. We denote the translation corresponding to $x \in X$ by $t_x$.

As is well known, $W^\af$ is a Coxeter group, and the basis of $R_0$ gives rise to a set $S^\af$
of simple (affine) reflections. The length function $\ell$ of the Coxeter system $(W^\af ,S^\af )$
extends naturally to $W$, by counting how many negative roots of the affine root system
$R_0^\vee \times \Z$ are made positive by an element of $W$. We write
\begin{align*}
&X^+ := \{ x \in X : \inp{x}{\alpha^\vee} \geq 0
\; \forall \alpha \in F_0 \} , \\
&X^- := \{ x \in X : \inp{x}{\alpha^\vee} \leq 0
\; \forall \alpha \in F_0 \} = -X^+ .
\end{align*}
It is easily seen that the center of $W$ is the lattice
\[
Z(W) = X^+ \cap X^- .
\]
We say that $\mc R$ is semisimple if $Z (W) = 0$ or equivalently if $R_0$ spans $\mf a^*$.
Thus a root datum is semisimple if and only if the corresponding reductive algebraic group is so.

With $\mc R$ we also associate some other root systems.
There is the non-reduced root system
\[
R_{nr} := R_0 \cup \{ 2 \alpha : \alpha^\vee \in 2 Y \} .
\]
Obviously we put $(2 \alpha )^\vee = \alpha^\vee / 2$. Let $R_1$
be the reduced root system of long roots in $R_{nr}$:
\[
R_1 := \{ \alpha \in R_{nr} : \alpha^\vee \not\in 2 Y \} .
\]
Consider a positive parameter function for $\mc R$, that is,
a function $q : W \to \R_{>0}$ such that
\begin{equation}\label{eq:parameterFunction}
\begin{array}{lll@{\quad}l}
q (\omega ) & = & 1 & \text{if } \ell (\omega ) = 0 , \\
q (w v) & = & q (w) q(v) & \text{if } w,v \in W \quad
\text{and} \quad \ell (wv) = \ell (w) + \ell (v) .
\end{array}
\end{equation}
Alternatively it can be given by $W_0$-invariant map $q : R_{nr}^\vee \to \R_{>0}$,
the relation being
\begin{equation}\label{eq:parameterEquivalence}
\begin{array}{lll}
q_{\alpha^\vee} = q(s_\alpha) = q (t_\alpha s_\alpha) & \text{if} & \alpha \in R_0 \cap R_1, \\
q_{\alpha^\vee} = q(t_\alpha s_\alpha) & \text{if} & \alpha \in R_0 \setminus R_1, \\
q_{\alpha^\vee / 2} = q(s_\alpha) q(t_\alpha s_\alpha)^{-1} & \text{if} &
\alpha \in R_0 \setminus R_1.
\end{array}
\end{equation}
In case $R_0$ is irreducible this means that $q$ is determined by one, two or three 
independent real numbers, where three only occurs for a root datum of type $C_n^{(1)}$. 
The affine Hecke algebra $\mc H = \mc H (\mc R ,q)$ is the unique associative
complex algebra with basis $\{ N_w : w \in W \}$ and multiplication rules
\begin{equation}\label{eq:multrules}
\begin{array}{lll}
N_w \, N_v = N_{w v} & \mr{if} & \ell (w v) = \ell (w) + \ell (v) \,, \\
\big( N_s - q(s)^{1/2} \big) \big( N_s + q(s)^{-1/2} \big) = 0 & \mr{if} & s \in S^\af .
\end{array}
\end{equation}
In the literature one also finds this algebra defined in terms of the elements
$q(s)^{1/2} N_s$, in which case the multiplication can be described without square roots.
The algebra $\mc H$ is endowed with a conjugate-linear involution, defined on basis
elements by $N_w^* := N_{w^{-1}}$.

For $x \in X^+$ we put $\theta_x := N_{t_x}$. The corresponding semigroup morphism
$X^+ \to \mc H (\mc R ,q)^\times$ extends to a group homomorphism
\[
X \to \mc H (\mc R ,q)^\times : x \mapsto \theta_x .
\]
A part of the Bernstein presentation \cite[\S 3]{Lus-Gr} says that the subalgebra
$\mc A := \mr{span} \{ \theta_x : x \in X \}$ is isomorphic to $\C [X]$, and that the
center $Z (\mc H)$ corresponds to $\C [X]^{W_0}$ under this isomorphism.
Let $T$ be the complex algebraic torus
\[
T = \mr{Hom}_{\mh Z} (X, \mh C^\times ) \cong Y \otimes_\Z \C^\times ,
\]
so $\mc A \cong \mc O (T)$ and $Z (\mc H ) = \mc A^{W_0} \cong \mc O (T / W_0 )$.
This torus admits a polar decomposition
\[
T = T_{rs} \times T_{un} = \mr{Hom}_\Z (X, \R_{>0}) \times \mr{Hom}_\Z (X, S^1)
\]
into a real split (or positive) part and a unitary part.

Let $\mr{Mod}_f (\mc H)$ be the category of finite dimensional $\mc H$-modules,
and $\mr{Mod}_{f, W_0 t} (\mc H)$ the full subcategory of modules that admit the
$Z (\mc H)$-character $W_0 t \in T / W_0$. We let $G (\mc H)$ be the Grothendieck
group of $\mr{Mod}_f (\mc H)$ and we write $G_\C (\mc H) = \C \otimes_\Z G (\mc H)$.
Furthermore we denote by Irr$(\mc H)$, respectively $\mr{Irr}_{W_0 t} (\mc H)$,
the set of equivalence classes of irreducible objects in $\mr{Mod}_f (\mc H)$,
respectively $\mr{Mod}_{f, W_0 t} (\mc H)$.
We will use these notations also for other algebras and groups.

\subsection{Parabolic induction}

For a set of simple roots $P \subset F_0$ we introduce the notations
\begin{equation}\label{eq:parabolic}
\begin{array}{l@{\qquad}l}
R_P = \mh Q P \cap R_0 & R_P^\vee = \mh Q P^\vee \cap R_0^\vee , \\
X_P = X \big/ \big( X \cap (P^\vee )^\perp \big) &
X^P = X / (X \cap \mh Q P ) , \\
Y_P = Y \cap \mh Q P^\vee & Y^P = Y \cap P^\perp , \\
\mf a_P = \R P^\vee & \mf a^P = P^\perp , \\
T_P = \mr{Hom}_{\mh Z} (X_P, \mh C^\times ) &
T^P = \mr{Hom}_{\mh Z} (X^P, \mh C^\times ) , \\
\mc R_P = ( X_P ,R_P ,Y_P ,R_P^\vee ,P) & \mc R^P = (X,R_P ,Y,R_P^\vee ,P) .
\end{array}
\end{equation}
Notice that $T_P$ and $T^P$ are algebraic subtori of $T$.
Although $T_{rs} = T_{P,rs} \times T_{rs}^P$, the product
$T_{un} = T_{P,un} T^P_{un}$ is not direct, because the intersection
\[
T_{P,un} \cap T_{un}^P = T_P \cap T^P 
\]
can have more than one element (but only finitely many).

We define parameter functions $q_P$ and $q^P$ on the root
data $\mc R_P$ and $\mc R^P$, as follows. Restrict $q$ to a function on
$(R_P )_{nr}^\vee$ and use \eqref{eq:parameterEquivalence} to extend it to
$W (\mc R_P )$ and $W (\mc R^P )$. Now we can define the parabolic subalgebras
\[
\mc H_P = \mc H (\mc R_P ,q_P ) ,\qquad \mc H^P = \mc H (\mc R^P ,q^P ) .
\]
Despite our terminology $\mc H^P$ and $\mc H_P$ are not subalgebras of $\mc H$,
but they are close. Namely, $\mc H (\mc R^P ,q^P )$ is isomorphic to the subalgebra of
$\mc H (\mc R ,q)$ generated by $\mc A$ and $\mc H (W (R_P) ,q_P)$.
We denote the image of $x \in X$ in $X_P$ by $x_P$ and we let $\mc A_P \subset \mc H_P$
be the commutative subalgebra spanned by $\{ \theta_{x_P} : x_P \in X_P \}$.
There is natural surjective quotient map
\begin{equation}\label{eq:quotientP}
\mc H^P \to \mc H_P : \theta_x N_w \mapsto \theta_{x_P} N_w .
\end{equation}
For all $x \in X$ and $\alpha \in P$ we have
\[
x - s_\alpha (x) = \inp{x}{\alpha^\vee} \alpha \in \Z P,
\]
so $t (s_\alpha (x)) = t(x)$ for all $t \in T^P$. Hence $t (w(x)) = t (x)$ for all $w \in W (R_P)$,
and we can define an algebra automorphism
\begin{equation}
\phi_t : \mc H^P \to \mc H^P, \quad \phi_t (\theta_x N_w) = t (x) \theta_x N_w  \qquad t \in T^P .
\end{equation}
In particular, for $t \in T_P \cap T^P$ this descends to an algebra automorphism
\begin{equation}\label{eq:twistKP}
\psi_t : \mc H_P \to \mc H_P , \quad \theta_{x_P} N_w \mapsto t(x_P) \theta_{x_P} N_w
\qquad t \in T_P \cap T^P .
\end{equation}
Suppose that $g \in W_0$ satisfies $g (P) = Q \subseteq F_0$.
Then there are algebra isomorphisms
\begin{equation}\label{eq:psig}
\begin{array}{llcl}
\psi_g : \mc H_P \to \mc H_Q , &
\theta_{x_P} N_w & \mapsto & \theta_{g (x_P)} N_{g w g^{-1}} , \\
\psi_g : \mc H^P \to \mc H^Q , &
\theta_x N_w & \mapsto & \theta_{g x} N_{g w g^{-1}} .
\end{array}
\end{equation}

We can regard any representation $(\sigma ,V_\sigma)$ of $\mc H (\mc R_P ,q_P )$ as a
representation of $\mc H (\mc R^P ,q^P)$ via the quotient map \eqref{eq:quotientP}.
Thus we can construct the $\mc H$-representation
\[
\pi (P,\sigma ,t) := \mr{Ind}_{\mc H (\mc R^P ,q^P )}^{\mc H (\mc R ,q)} (\sigma \circ \phi_t ) .
\]
Representations of this form are said to be parabolically induced.

We intend to partition Irr$(\mc H )$ into finite packets, each of which is obtained
by inducing a discrete series representation of a parabolic subalgebra of $\mc H$.
The discrete series and tempered representations are defined via the
$\mc A$-weights of a representation, as we recall now.
Given $P \subseteq F_0$, we have the following positive cones in $\mf a$ and in $T_{rs}$:
\begin{equation}
\begin{array}{lll@{\qquad}lll}
\mf a^+ & = & \{ \mu \in \mf a : \inp{\alpha}{\mu} \geq 0 \:
  \forall \alpha \in F_0 \} , & T^+ & = & \exp (\mf a^+) , \\
\mf a_P^+ & = &  \{ \mu \in \mf a_P : \inp{\alpha}{\mu} \geq 0 \: \forall \alpha \in P \} , &
  T_P^+ & = & \exp (\mf a_P^+) , \\
\mf a^{P+} & = & \{ \mu \in \mf a^P : \inp{\alpha}{\mu} \geq 0 \:
  \forall \alpha \in F_0 \setminus P \} , & T^{P+} & = & \exp (\mf a^{P+}) , \\
\mf a^{P++} & = & \{ \mu \in \mf a^P : \inp{\alpha}{\mu} > 0 \:
  \forall \alpha \in F_0 \setminus P \} , & T^{P++} & = & \exp (\mf a^{P++}) .
\end{array}
\end{equation}
The antidual of $\mf a^{*+} :=  \{ x \in \mf a^* :  \inp{x}{\alpha^\vee} \geq 0
\: \forall \alpha \in F_0 \}$ is
\begin{equation}
\mf a^- = \{ \lambda \in \mf a : \inp{x}{\lambda} \leq 0 \: \forall x \in \mf a^{*+} \} =
\big\{ \sum\nolimits_{\alpha \in F_0} \lambda_\alpha \alpha^\vee : \lambda_\alpha \leq 0 \big\} .
\end{equation}
Similarly we define
\begin{equation}
\mf a_P^- = \big\{ \sum\nolimits_{\alpha \in P} \lambda_\alpha \alpha^\vee  \in \mf a_P :
\lambda_\alpha \leq 0 \big\} .
\end{equation}
The interior $\mf a^{--}$ of $\mf a^-$ equals
$\big\{ {\ts \sum_{\alpha \in F_0}} \lambda_\alpha \alpha^\vee : \lambda_\alpha < 0 \big\}$
if $F_0$ spans $\mf a^*$, and is empty otherwise. We write $T^- = \exp (\mf a^-)$ and
$T^{--} = \exp (\mf a^{--})$.

Let $t = |t| \cdot t |t|^{-1} \in T_{rs} \times T_{un}$ be the polar decomposition of $t$.
An $\mc H$-representation is called tempered if $|t| \in T^-$ for all its
$\mc A$-weights $t$, and anti-tempered if $|t|^{-1} \in T^-$ for all such $t$. More
restrictively we say that an irreducible  $\mc H$-representation belongs to the discrete
series (or simply: is discrete series) if $|t| \in T^{--}$, for all its $\mc A$-weights $t$.
In particular the discrete series is empty if $F_0$ does not span $\mf a^*$.

Our induction data are triples $(P,\delta,t)$, where
\begin{itemize}
\item $P \subset F_0$;
\item $(\delta ,V_\delta)$ is a discrete series representation of $\mc H_P$;
\item $t \in T^P$.
\end{itemize}
Let $\Xi$ be the space of such induction data, where we consider $\delta$ only modulo
equivalence of $\mc H_P$-representations. We say that $\xi = (P,\delta,t)$ is unitary
if $t \in T^P_{un}$, and we denote the space of unitary induction data by $\Xi_{un}$.
Similarly we say that $\xi$ is positive if $|t| \in T^{P+}$, which we write as $\xi \in \Xi^+$.
We have three collections of induction data:
\begin{equation}\label{eq:inductionData}
\Xi_{un} \subseteq \Xi^+ \subseteq \Xi .
\end{equation}
By default we endow these spaces with the topology for which $P$ and $\delta$ are
discrete variables and $T^P$ carries its natural analytic topology. With $\xi \in \Xi$
we associate the parabolically induced representation
\[
\pi (\xi) = \pi (P,\delta,t) := \mr{Ind}_{\mc H^P}^{\mc H} (\delta \circ \phi_t ) .
\]
As vector space underlying $\pi (\xi)$ we will always take $\C [W^P] \otimes V_\delta$,
where $W^P$ is the collection of shortest length representatives of $W_0 / W (R_P)$.
This space does not depend on $t$, which will allow us to speak of maps that
are continuous, smooth, polynomial or even rational in the parameter $t \in T^P$.
It is known that $\pi (\xi)$ is unitary and tempered if $\xi \in \Xi_{un}$, and non-tempered
if $\xi \in \Xi \setminus \Xi_{un}$, see \cite[Propositions 4.19 and 4.20]{Opd-Sp} and
\cite[Lemma 3.1.1]{Sol-Irr}.

The relations between such representations are governed by intertwining operators.
Their construction \cite{Opd-Sp} is rather complicated, so we recall only their important
properties. Suppose that $P,Q \subset F_0 , k \in T_P \cap T^P , w \in W_0$ and $w (P) = Q$.
Let $\delta$ and $\sigma$ be discrete series representations of respectively $\mc H_P$ and
$\mc H_Q$, such that $\sigma$ is equivalent with $\delta \circ \psi_k^{-1} \circ \psi_w^{-1}$.

\begin{thm}\label{thm:intOp}
\textup{\cite[Theorem 4.33 and Corollary 4.34]{Opd-Sp}}
\enuma{
\item There exists a family of intertwining operators
\[
\pi (w k,P,\delta,t) : \pi (P,\delta,t) \to \pi (Q,\sigma,w (kt)) .
\]
As a map $\C [W^P] \otimes_\C V_\delta \to \C [W^Q] \otimes_\C V_\sigma$
it is a rational in $t \in T^P$ and constant on $T^{F_0}$-cosets.
\item This map is regular and invertible on an open neighborhood of $T^P_{un}$
in $T^P$ (with respect to the analytic topology).
\item $\pi (w k,P,\delta,t)$ is unitary if $t \in T^P_{un}$.
}
\end{thm}

We can gather all these intertwining operators in a groupoid $\mc W$, which we define now. The
base space of $\mc W$ is the power set of $F_0$ and the collection of arrows from $P$ to $Q$ is
\[
\mc W_{PQ} := \{ w \in W_0 : w (P) = Q \} \times T_P \cap T^P .
\]
Whenever it is defined, the multiplication in $\mc W$ is
\[
(w_1, k_1) \cdot (w_2, k_2) = (w_1 w_2, w_2^{-1}(k_1) k_2) .
\]
Families of intertwining operators $\pi (wk,P,\delta,t)$ satisfying the
properties listed in Theorem \ref{thm:intOp} are unique only up to normalization
by rational functions of $t\in T^P$ which are regular in an open neighborhood of
$T^P_{un}$, have absolute value equal to $1$ on $T^P_{un}$, and are constant on
$T^{F_0}$-cosets. The intertwining operators defined in \cite{Opd-Sp} are normalized
in such a way that composition of the intertwining operators corresponds to
multiplication of the corresponding elements of $\mc{W}$ only up to a scalar.

More precisely, let $g \in \mc W$ be such that $g w$ is defined. Then there exists a
number $\lambda \in \C$ with $|\lambda| = 1$, independent of $t$, such that
\begin{equation}\label{eq:multIntOp}
\pi (g ,Q,\sigma,w (k t)) \circ \pi (w k,P,\delta,t) =
\lambda \, \pi (g w k, P,\delta, t)
\end{equation}
as rational functions of $t$. We fix such a normalization of the intertwining
operators once and for all.

Let $W (R_P) r_\delta \in T_P / W (R_P)$ be the central character of the $\mc H_P$-representation
$\delta$. Then $|r_\delta| \in T_{P,rs} = \exp (\mf a_P)$, so we can define
\begin{equation}\label{eq:ccdelta}
cc_P (\delta) := W (R_P) \log |r_\delta| \in \mf a_P / W (R_P) .
\end{equation}
Since the inner product on $\mf a$ is $W_0$-invariant, the number $\norm{cc_P (\sigma)}$
is well-defined.

\begin{thm}\label{thm:parametrizationOfDual}
\textup{\cite[Theorem 3.3.2]{Sol-Irr}} \\
Let $\rho$ be an irreducible $\mc H$-representation. There exists a unique association class
$\mc W (P,\delta,t) \in \Xi / \mc W$ such that the following equivalent properties hold:
\enuma{
\item $\rho$ is isomorphic to an irreducible quotient of $\pi (\xi^+)$, for some
$\xi^+ = \xi^+ (\rho) \in \Xi^+ \cap \mc W (P,\delta,t)$;
\item $\rho$ is a constituent of $\pi (P,\delta,t)$, and $\norm{cc_P (\delta)}$
is maximal for this property.
}
\end{thm}

For any $\xi \in \Xi$ the packet
\begin{equation} \label{eq:packetReps}
\mr{Irr}_{\mc W \xi} (\mc H) := \{ \rho \in \mr{Irr} (\mc H) : \xi^+ (\rho) \in \mc W \xi \}
\end{equation}
is finite, but it is not so easy to predict how many equivalence classes of representations
it contains. This is one of the purposes of R-groups.

\subsection{The Schwartz algebra}

We recall how to complete an affine Hecke algebra to a topological algebra called the
Schwartz algebra. As a vector space $\mc S$ will consist of rapidly decreasing functions
on $W$, with respect to some length function. For this purpose it is unsatisfactory
that $\ell$ is 0 on the subgroup $Z(W)$, as this can be a large part
of $W$. To overcome this inconvenience, let
$L : X \otimes \mh R \to [0,\infty )$ be a function such that
\begin{itemize}
\item $L (X) \subset \mh Z$,
\item $L(x+y) = L(x) \quad \forall x \in X \otimes \mh R, y \in \R R_0$,
\item $L$ induces a norm on
$X \otimes \mh R / \R R_0 \cong Z(W) \otimes \mh R$.
\end{itemize}
Now we define
\[
\mc N (w) := \ell (w) + L (w(0))  \qquad w \in W .
\]
Since $Z(W) \oplus \mh Z R_0$ is of finite index in $X$, the set
$\{ w \in W : \mc N (w) = 0 \}$ is finite. Moreover, because $W$ is
the semidirect product of a finite group and an abelian group, it is of
polynomial growth and different choices of $L$ lead to equivalent length
functions $\mc N$. For $n \in \mh N$ we define the norm
\begin{equation}\label{eq:seminorms}
p_n \big( {\ts \sum_{w \in W}} h_w N_w \big) :=
\sup_{w \in W} |h_w | (\mc N (w) + 1 )^n .
\end{equation}
The completion $\mc S = \mc S (\mc R ,q)$ of $\mc H (\mc R ,q)$ with respect to the family
of norms $\{ p_n \}_{n \in \mh N}$ is a nuclear Fr\'echet space. It consists of all possible
infinite sums $h = \sum_{w \in W} h_w N_w$ such that $p_n (h) < \infty \; \forall n \in \mh N$.
By \cite[Section 6.2]{Opd-Sp} or \cite[Appendix A]{OpSo2} $\mc S (\mc R,q)$ is a unital
locally convex *-algebra.

A crucial role in the harmonic analysis on affine Hecke algebra is played by a particular
Fourier transform, which is based on the induction data space $\Xi$. Let $\mc V_\Xi$
be the vector bundle over $\Xi$, whose fiber at $(P,\delta,t) \in \Xi$ is the representation
space $\C [W^P] \otimes V_\delta$ of $\pi (P,\delta,t)$. Let $\mr{End} (\mc V_\Xi)$ be the
algebra bundle with fibers $\mr{End}_\C (\C [W^P] \otimes V_\delta)$. Of course these
vector bundles are trivial on every connected component of $\Xi$, but globally not even the
dimensions need be constant. Since $\Xi$ has the structure of a complex algebraic variety,
we can construct the algebra of polynomial sections of $\mr{End} (\mc V_\Xi)$:
\[
\mc O \big( \Xi ; \mr{End} (\mc V_\Xi) \big) := \bigoplus_{P,\delta} \mc O (T^P) \otimes
\mr{End}_\C (\C [W^P] \otimes V_\delta) .
\]
For a submanifold $\Xi' \subset \Xi$ we define the algebra $C^\infty \big( \Xi' ;
\mr{End} (\mc V_\Xi) \big)$ in similar fashion.
The intertwining operators from Theorem \ref{thm:intOp} give rise to an action of the
groupoid $\mc W$ on the algebra of rational sections of $\mr{End} (\mc V_\Xi)$, by
\begin{equation}\label{eq:actSections}
(w \cdot f) (\xi) = \pi (w,w^{-1} \xi ) f (w^{-1} \xi) \pi (w, w^{-1} \xi )^{-1} ,
\end{equation}
whenever $w^{-1} \xi \in \Xi$ is defined. This formula also defines groupoid actions of
$\mc W$ on $C^\infty \big( \Xi' ; \mr{End} (\mc V_\Xi) \big)$, provided that $\Xi'$ is a
$\mc W$-stable submanifold of $\Xi$ on which all the intertwining operators are regular.
Given a suitable collection $\Sigma$ of sections of $(\Xi', \mr{End} (\mc V_\Xi) )$,
we write
\[
\Sigma^{\mc W} = \{ f \in \Sigma : (w \cdot f) (\xi) = f(\xi) \text{ for all } w \in \mc W,\
\xi \in \Xi' \text{ such that } w^{-1} \xi \text{ is defined} \} .
\]
The Fourier transform for $\mc H$ is the algebra homomorphism
\begin{align*}
& \mc F : \mc H \to \mc O \big( \Xi ; \mr{End} (\mc V_\Xi) \big) , \\
& \mc F (h) (\xi) = \pi (\xi) (h) .
\end{align*}
The very definition of intertwining operators shows that the image of $\mc F$ is contained
in the algebra $\mc O \big( \Xi ; \mr{End} (\mc V_\Xi) \big)^{\mc W}$.
By \cite[Theorem 5.3]{DeOp1} the Fourier transform extends to an isomorphism of
Fr\'echet *-algebras
\begin{equation}\label{eq:FourierIso}
\mc F : \mc S (\mc R,q) \to C^\infty \big( \Xi_{un} ; \mr{End} (\mc V_\Xi ) \big)^{\mc W} .
\end{equation}
Let $(P_1,\delta_1), \ldots, (P_N, \delta_N)$ be representatives for the action of $\mc W$ on
pairs $(P,\delta)$. Then the right hand side of \eqref{eq:FourierIso} can be rewritten as
\begin{equation}\label{eq:25}
\bigoplus\nolimits_{i=1}^N \big( C^\infty (T^{P_i}_{un}) \otimes
\mr{End} (\C [W^{P_i}] \otimes V_{\delta_i}) \big)^{\mc W_{P_i ,\delta_i}} ,
\end{equation}
where $\mc W_{P,\delta} = \{ w \in \mc W : w (P) = P, \delta \circ \psi_w^{-1} \cong \delta \}$
is the isotropy group of $(P,\delta)$.

It was shown in \cite[Corollary 5.5]{DeOp1} that \eqref{eq:FourierIso} implies that
\begin{equation}\label{eq:ZS}
\text{the center } \mc Z \text{ of } \mc S \text{ is isomorphic to } C^\infty (\Xi_{un})^{\mc W},
\end{equation}
so the space of central characters of $\mc S$ is $\Xi_{un} / \mc W$.
We let $\mr{Mod}_f (\mc S)$ be the category of finite dimensional $\mc S$-modules
and $\mr{Mod}_{f,\mc W \xi} (\mc S)$ the full subcategory of modules which admit the
central character $\mc W \xi$. The collection $\mr{Irr}_{\mc W \xi} (\mc S)$ of
(equivalence classes of) irreducible objects in $\mr{Mod}_{f,\mc W \xi} (\mc S)$ equals
$\mr{Irr}_{\mc W \xi} (\mc H)$, in the notation of \eqref{eq:packetReps}.
\vspace{2mm}

\section{Formal completion of the Schwartz algebra}
\label{sec:Ext}

We study how the Ext-groups of Fr\'echet $\mc S$-modules behave under formal completion
with respect to a maximal ideal of the center. For compatibility with \cite{OpSo1} we will work in
the category $\mr{Mod}_{bor}(\mc S)$ of complete bornological modules over the nuclear
Fr\'echet algebra $\mathcal{S}$, equipped with its precompact bornology. This is an exact category
with respect to the class of $\mc S$-module extensions that are split as bornological vector spaces.

Two remarks for readers who are not familiar with bornologies. As discussed in the introduction, 
one may think of Fr\'echet $\mc S$-modules and exact sequences which admit a continuous linear 
splitting, that is just as good in the setting of this paper. Moreover, if one is only interested 
in finite dimensional modules, then the subtleties with topological modules are superfluous and 
all the results can also be obtained in an algebraic way. However, the equivalence of the algebraic 
and the bornological approach in a finite dimensional setting is not automatic, it follows from 
the existence of certain nice projective resolutions \cite{OpSo1}.

Although the entire section is formulated in terms of the Schwartz algebra $\mc S$, the essence
is the study of algebras of smooth functions on manifolds. The results and proofs are equally 
valid, and of some independent interest, if one replaces $\mc S$ by $C^\infty (M)$, where $M$
is a smooth manifold such that $C^\infty (M) \cong \mc S (\Z)$ as Fr\'echet spaces. According
to \cite[Satz 31.16]{MeVo} this is the case when $M$ has no boundary or is the closure of an open 
bounded subset of $\R^n$ with $C^1$ boundary. The formal completion of $C^\infty (M)$ with respect
to a maximal ideal is a power series ring, which provides a good geometric interpretation
for the material in this section.

In \eqref{eq:ZS} $\mathcal{Z} \cong C^\infty (\Xi_{un})^{\mathcal{W}}$ is equipped with the
Fr\'echet topology from $C^\infty (\Xi_{un})$. We fix $\xi = (P,\delta,t) \in \Xi_{un}$.
Let $m_{\mathcal{W}\xi}^\infty$ denote the closed ideal of $\mathcal{Z}$
of functions that are flat at $\mathcal{W}\xi$, that is, their Taylor series around any 
point of $\mc W \xi$ is zero. We define a Fr\'echet module
$\widehat{\mc{Z}}_{\mc{W}\xi}$ over $\mc{Z}$ by the exact sequence
\begin{equation}\label{eq:flatexZ}
0 \to m_{\mathcal{W}\xi}^{\infty} \to \mathcal{Z} \to \widehat{\mc{Z}}_{\mc{W}\xi} \to 0 .
\end{equation}
It follows easily from Borel's lemma that the Taylor series
map $\tau_\xi$ defines a continuous surjection
$\tau_\xi:\mathcal{Z}\to \mathcal{F}^{\mathcal{W}_\xi}_\xi$
where $\mathcal{F}^{\mathcal{W}_\xi}_\xi$
denotes the unital Fr\'echet algebra of $\mathcal{W}_\xi$-invariant formal power series
at $\xi$. According to \cite[Remarque IV.3.9]{Tou} $\tau_\xi$ is not linearly split.
The kernel of $\tau_\xi$ is isomorphic to $m_{\mathcal{W}\xi}^{\infty}$, hence
we may identify $\widehat{\mc{Z}}_{\mc{W}\xi}$ and $\mathcal{F}^{\mathcal{W}_\xi}_\xi$.
Notice that this algebra is Noetherian, like any power series ring with
finitely many variables and coefficients in a field.
By Whitney's spectral theorem \cite[V.1.6]{Tou}
\begin{equation}\label{eq:m2xi}
\overline{(m^\infty_{\mc W \xi})^2} = \overline{m^\infty_{\mc W \xi}} = m^\infty_{\mc W \xi} .
\end{equation}
For similar reasons $\mathcal{I}_{\mathcal{W}\xi}:=m_{\mathcal{W}\xi}^{\infty}\mathcal{S}$
is a closed two-sided ideal in $\mathcal{S}$.
Hence by applying the completed projective tensor product functor
$?\widehat\otimes_\mathcal{Z}\mathcal{S}$ to
\eqref{eq:flatexZ} we obtain
\begin{equation}\label{eq:q}
0\to\mathcal{I}_{\mathcal{W}\xi}\to\mathcal{S}\to
\widehat{\mc{Z}}_{\mc{W}\xi}\widehat\otimes_{\mathcal{Z}}\mathcal{S} \to 0 .
\end{equation}
We define the completion of $\mc{S}$ at $\mc{W}\xi$ to be
the Fr\'echet $\widehat{\mc{Z}}_{\mc{W}\xi}$-algebra
\begin{equation}\label{eq:compl}
\widehat{\mathcal{S}}_{\mathcal{W}\xi}:=\widehat{\mc{Z}}_{\mc{W}\xi}
\widehat\otimes_{\mathcal{Z}}\mathcal{S} .
\end{equation}
By \eqref{eq:FourierIso} and \eqref{eq:ZS} $\mc{S}$ is a finitely generated
$\mc{Z}$-module, so $\widehat{\mathcal{\mc S}}_{\mathcal{W}\xi}$ is left and right
Noetherian. In particular the category $\mr{Mod}_{fg}(\widehat{\mc S}_{\mc W \xi})$ of
finitely generated $\widehat{S}_{\mc W \xi}$-modules is closed under extensions.
Any $M \in \mr{Mod}_{fg} (\widehat{\mathcal{S}}_{\mathcal{W}\xi})$
admits a finite presentation
\begin{equation}\label{eq:pres}
\widehat{\mathcal{S}}_{\mathcal{W}\xi}^{n_2}\to
\widehat{\mathcal{S}}_{\mathcal{W}\xi}^{n_1}\to M \to 0 .
\end{equation}
It follows that $\widehat{\mathcal{S}}_{\mathcal{W}\xi}$
is a Fr\'echet space of finite type, i.e. that its topology is defined by
an increasing sequence of seminorms all of whose cokernels have finite codimension.
By \cite{Kopp} all continuous linear maps between such
Fr\'echet spaces have closed images. In particular, the relations module
im$\big( \widehat{\mathcal{S}}_{\mathcal{W}\xi}^{n_2}\to
\widehat{\mathcal{S}}_{\mathcal{W}\xi}^{n_1} \big)$
in \eqref{eq:pres} is closed, showing that $M$ is a Fr\'echet module (i.e. the
canonical topology on $M$ in the sense of \cite{Tou} is separated, hence
defines a Fr\'echet module structure on $M$). It also follows easily that
morphisms between finitely generated $\widehat{\mathcal{S}}_{\mathcal{W}\xi}$-modules
are automatically continuous with respect to the canonical topology. Summarizing:

\begin{lem}\label{lem:fgmod}
$\mr{Mod}_{fg} (\widehat{\mathcal{S}}_{\mathcal{W}\xi})$ is a Serre subcategory
of  $\mr{Mod}_{Fr\acute e} (\widehat{\mathcal{S}}_{\mathcal{W}\xi})$, the
category of Fr\'echet $\widehat{\mathcal{S}}_{\mathcal{W}\xi}$-modules.
\end{lem}

Let $q : \mc S \to \widehat{\mc S}_{\mc W \xi}$ be the quotient map and
\[
q^* : \mr{Mod}_{bor}(\widehat{\mc S}_{\mc W \xi}) \to \mr{Mod}_{bor}(\mc S)
\]
the associated pullback functor. In the opposite direction the completed bornological
tensor product provides a functor
\[
\widehat{\mc S}_{\mc W \xi} \widehat{\otimes}_{\mc S} : \mr{Mod}_{bor}(\mc S) \to
\mr{Mod}_{bor}(\widehat{\mc S}_{\mc W \xi}) .
\]
We remark that for Fr\'echet modules this tensor product agrees with the completed
projective one.

\begin{lem}\label{lem:compl}
Let $V \in \mr{Mod}_{bor}(\mc S)$ and abbreviate $\widehat{V}_{\mc W \xi} :=
\widehat{\mc S}_{\mc W \xi} \widehat{\otimes}_{\mc S} V$.
\enuma{
\item $\widehat{V}_{\mc{W}\xi} \cong \widehat{\mc{Z}}_{\mc{W}\xi}\widehat
\otimes_{\mathcal{Z}} V \cong V / \overline{m^{\infty}_{\mc{W}\xi} V}$.
\item For all $M \in \mr{Mod}_{bor} (\widehat{\mc S}_{\mc W \xi})$ there is a
natural isomorphism $\widehat{\mc S}_{\mc W \xi} \widehat{\otimes}_{\mc S} q^* M \cong M$.
\item $q^*$ is fully faithful and exact.
\item $q^*$ is right adjoint to $\widehat{\mc S}_{\mc W \xi} \widehat{\otimes}_{\mc S}$.
\item $\mr{Mod}^{\mc{W}\xi, tor}_{bor}(\mc{S}) :=
\{M \in \mr{Mod}_{bor}(\mc{S}) \mid m^{\infty}_{\mc{W}\xi} M = 0 \}$ is a Serre
subcategory of $\mr{Mod}_{bor}(\mc{S})$.
\item $q^*$ defines equivalences of exact categories
$\mr{Mod}_{bor}(\widehat{\mc S}_{\mc W \xi}) \to \mr{Mod}^{\mc{W}\xi, tor}_{bor}(\mc{S})$
and $\mr{Mod}_{Fr\acute e}(\widehat{\mc S}_{\mc W \xi}) \to
\mr{Mod}^{\mc{W}\xi, tor}_{Fr\acute e}(\mc{S})$.
}
\end{lem}
\begin{proof}
(a) By the associativity of bornological tensor products
\begin{multline*}
\widehat{V}_{\mc W \xi} = \widehat{\mc S}_{\mc W \xi} \widehat{\otimes}_{\mc S} V =
\widehat{\mc{Z}}_{\mc{W}\xi} \widehat \otimes_{\mathcal{Z}} \mc S \widehat{\otimes}_{\mc S} V
\cong \widehat{\mc{Z}}_{\mc{W}\xi}\widehat\otimes_{\mathcal{Z}} V \\
\cong \widehat{\mc{Z}}_{\mc{W}\xi}\widehat\otimes_{\mathcal{Z}}(V /m^{\infty}_{\mc{W}\xi} V)
\cong \widehat{\mc{Z}}_{\mc{W}\xi} \widehat\otimes_{\widehat{\mc{Z}}_{\mc{W}\xi}}
(V/m^{\infty}_{\mc{W}\xi} V) \cong V/ \overline{m^{\infty}_{\mc{W}\xi} V} .
\end{multline*}
(b) follows immediately from (a) \\
(c) is obvious. \\
(d) Since the kernel of any $\mc S$-module homomorphism $V \to q^* M$ contains
$\overline{m^{\infty}_{\mc{W}\xi} V}$,
\[
\mr{Hom}_{\mc S} (V, q^* M)
\cong \mr{Hom}_{\mc S} (V/ \overline{m^{\infty}_{\mc{W}\xi} V}, q^* M) .
\]
By (a), (b) and (c) the right hand side is isomorphic to
\[
\mr{Hom}_{\mc S} (q^* \widehat{V}_{\mc W \xi}, q^* M) \cong
\mr{Hom}_{\widehat{\mc S}_{\mc W \xi}} (\widehat{V}_{\mc W \xi}, M) .
\]
(e) The only nontrivial part concerns extensions. Let
\[
0 \to V_1 \xrightarrow{i} V_2 \xrightarrow{p} V_3 \to 0
\]
be a short exact sequence in $\mr{Mod}_{bor} (\mc S)$ and assume that
$m^\infty_{\mc W \xi} V_1 = m^\infty_{\mc W \xi} V_3 = 0$. For $m \in m^\infty_{\mc W \xi} V_2$
and $v_2 \in V_2$ we have $p (m v_2) = m p (v_2) \in m^\infty_{\mc W \xi} V_3 = 0$, so
$m^\infty_{\mc W \xi} V_2 \subset i (V_1)$. Moreover by \eqref{eq:m2xi}
\[
m^\infty_{\mc W \xi} V_2 = \overline{(m^\infty_{\mc W \xi})^2} V_2 \subset
\overline{m^\infty_{\mc W \xi} i (m^\infty_{\mc W \xi} V_1)} =
\overline{m^\infty_{\mc W \xi} i (0)} = 0 ,
\]
so $V_2 \in \mr{Tor}^\infty_{\mc{W}\xi}(\mc{S})$. \\
(f) In view of (c) we only have to show that $q^* : \mr{Mod}_{bor}(\widehat{\mc S}_{\mc W \xi})
\to \mr{Mod}^{\mc{W}\xi, tor}_{bor}(\mc{S})$ is essentially bijective. Clearly any $V \in
\mr{Mod}^{\mc{W}\xi, tor}_{bor}(\mc{S})$ can be considered as a
$\widehat{S}_{\mc W \xi}$-module~$V'$, and as such $V = q^* V'$. Conversely, if $V \cong q^* M$
for some $M \in \mr{Mod}^{\mc{W}\xi, tor}_{bor}(\widehat{\mc S}_{\mc W \xi})$, then
\[
M \cong \widehat{\mc S}_{\mc W \xi} \widehat{\otimes}_{\mc S} q^* M \cong
\widehat{\mc S}_{\mc W \xi} \widehat{\otimes}_{\mc S} V = \widehat{V}_{\mc W \xi}
\]
by part (b).
\end{proof}

The spaces $\mr{Ext}^n_{\mc S} (\pi,\pi')$ carry a natural $\mc Z$-module structure.
Let the tempered central characters of $\pi$ and $\pi^\prime$ be $\mc{W}\xi$ and $\mc{W}\xi^\prime$
respectively. Standard arguments show that $\mr{Ext}^n_{\mc S} (\pi,\pi')$ has central character
$\mc{W}\xi$ if $\mc{W}\xi=\mc{W}\xi^\prime$, and $\mr{Ext}^n_{\mc S} (\pi,\pi')=0$ otherwise.
To compute these Ext-groups we shall at some point want to pass to the formal
completion of $\mc{S}$ at $\mc{W}\xi$. For that purpose we would like the functor
$\widehat{\mc S}_{\mc W \xi} \widehat{\otimes}_{\mc S}$ to be
exact, but unfortunately the authors do not know whether it is so on the categories of
Fr\'echet $\mc S$-modules or bornological $\mc S$-modules. Moreover this functor does not always
respect linear splittings of short exact sequences.

Therefore we restrict to smaller module categories.
Let $\mc S (\Z)$ be the nuclear Fr\'echet space of rapidly decreasing functions $\Z \to \C$ and
let $\Omega$ (resp. $\DNO$) be the category of Fr\'echet spaces that are isomorphic to a quotient
(resp. a direct summand) of $\mc S (\Z)$. Our notation is motivated by certain properties
$(\Omega)$ and $(DN)$ of Fr\'echet spaces, which characterize the quotients and the subspaces
of $\mc S (\Z)$ among all nuclear Fr\'echet spaces \cite{Vogt}.

These categories are exact but not abelian, because some morphisms do not have a kernel
or a cokernel. They are suitable to overcome to problems that can
arise from closed subspaces which are not complemented:

\begin{thm}\label{thm:Vogt}
\enuma{
\item The category $\Omega$ is closed under extensions of Fr\'echet spaces.
\item Let $0 \to V_1 \to V_2 \to V_3 \to 0$ be a short exact sequence of Fr\'echet spaces.
It splits whenever $V_1$ is a quotient of $\mc S (\Z)$ and $V_3$ is a subspace of $\mc S (\Z)$.
\item Every object of $\DNO$ is projective in $\Omega$.
}
\end{thm}
\begin{proof}
(a) is \cite[Lemma 1.7]{VoWa}.\\
(b) is \cite[Theorem 5.1]{Vogt}.\\
(c) follows easily from (b), as in \cite[Theorem 1.8]{Vogt}.
\end{proof}

The very definition \ref{eq:seminorms} shows that $\mc S \cong \mc S (\Z)$ as Fr\'echet spaces, 
and hence $\widehat{\mc S}_{\mc W \xi} \in \Omega$.
However, $\widehat{\mc S}_{\mc W \xi}$ is not isomorphic to a subspace of $\mc S (\Z)$,
because it is a finite extension of a Fr\'echet algebra of formal power series.

Let $\mr{Mod}_\Omega (\mc S)$ be the full subcategory of
$\mr{Mod}_{Fr\acute e}(\mc S)$ consisting of modules whose underlying
spaces belong to $\Omega$. We define $\mr{Mod}_{\Omega} (\mc Z) ,\;
\mr{Mod}_\Omega (\widehat{\mc S}_{\mc W \xi})$ and $\mr{Mod}_{\DNO} (\mc S)$ similarly.
In these categories the morphisms are all continuous module homomorphisms.
Apart from $\mr{Mod}_{\DNO} (\mc S)$ they do not have enough projective
objects, so derived functors are not defined for all modules. Nevertheless the Yoneda
Ext-groups, constructed as equivalence classes of higher extensions, are always available.

\begin{lem}\label{lem:catCS}
The categories $\mr{Mod}_{\DNO} (\mc S)$, $\mr{Mod}_\Omega (\mc S)$ and
$\mr{Mod}_{\Omega}(\widehat{\mc S}_{\mc W \xi})$ have the following properties:
\enuma{
\item They are closed under extensions of Fr\'echet modules.
\item Every short exact sequence in $\mr{Mod}_{\DNO}(\mc S)$ admits a continuous linear splitting.
\item They are exact categories in the sense of Quillen if we declare all short exact
sequences to be admissible.
\item For $F \in \DNO$ the $\mc S$-module $\mc S \widehat{\otimes} F$ is projective in
$\mr{Mod}_{\DNO}(\mc S)$, in $\mr{Mod}_\Omega (\mc S)$ and in $\mr{Mod}_{bor}(\mc S)$,
while $\widehat{\mc S}_{\mc W \xi} \widehat{\otimes} F$ is projective in
$\mr{Mod}_{\Omega}(\widehat{\mc S}_{\mc W \xi})$.
\item $\mr{Mod}_{\DNO}(\mc S)$ and $\mr{Mod}_{fg}(\widehat{\mc S}_{\mc W \xi})$
have enough projectives.
\item The following embeddings preserve Yoneda Ext-groups:
$\mr{Mod}_{\DNO}(\mc S) \to \mr{Mod}_{bor}(\mc S)$, $\mr{Mod}_{\DNO}(\mc S) \to
\mr{Mod}_{\Omega}(\mc S)$, $\mr{Mod}_{fg}(\widehat{\mc S}_{\mc W \xi}) \to
\mr{Mod}_{\Omega}(\widehat{\mc S}_{\mc W \xi})$ and $\mr{Mod}_{fg}
(\widehat{\mc S}_{\mc W \xi}) \to \mr{Mod}(\widehat{\mc S}_{\mc W \xi})$.
}
\end{lem}
\begin{proof}
(a) follows from Theorem \ref{thm:Vogt}.\\
(b) follows from Theorem \ref{thm:Vogt}.b.\\
(c) The definition of an exact category stems from \cite[Section 2]{Qui}, while it is worked
out in \cite[Appendix A]{Kel} which axioms are really necessary. All of these are trivially 
satisfied, except for the pullback and pushout properties. Since the verification of these 
properties is the same in all four categories, we only write it down in $\mr{Mod}_\Omega (\mc S)$.

Suppose that $0 \to V_1 \xrightarrow{i} V_2 \xrightarrow{p} V_3 \to 0$ is admissible exact in
$\mr{Mod}_\Omega (\mc S)$ and that $f: V_1 \to M$ is any morphism in
$\mr{Mod}_\Omega (\mc S)$. We have to show that there is a pushout diagram
\[
\begin{array}{ccccl}
V_1 & \xrightarrow{i} & V_2 & \xrightarrow{p} & V_3\\
\downarrow \scriptstyle{f} & & \downarrow  & &\\
M & \xrightarrow{g} & P & \to &  P / g(M)
\end{array}
\]
in $\mr{Mod}_\Omega (\mc S)$, with an admissible exact second row. Since $i$ is injective
the image of $(i,-f) : V_1 \to V_2 \oplus M$ is closed. Hence $P := V_2 \oplus M / \mr{im}(i,-f)$ is
a Fr\'echet $\mc S$-module. The canonical map $g : M \to P$ is injective and
\[
P / g(M) \cong {\ds \frac{V_2 \oplus M}{g(M) + \mr{im}(i,-f)}} \cong V_2 / i (V_1) \cong V_3 ,
\]
which belongs to $\mr{Mod}_\Omega (\mc S)$. By part (a) $P \in \mr{Mod}_\Omega (\mc S)$, so
$0 \to M \xrightarrow{g} P \to P / g(M) \to 0$ is admissible.

Concerning pullbacks, let $h : N \to V_3$ be any morphism in $\mr{Mod}_\Omega (\mc S)$. We
have to show that there is a pullback diagram
\[
\begin{array}{rcccc}
\ker (Q \to N) & \to & Q & \to & N \\
& & \downarrow & & \downarrow \scriptstyle{h}\\
V_1 & \xrightarrow{i} & V_2 & \xrightarrow{p} & V_3\\
\end{array}
\]
in $\mr{Mod}_\Omega (\mc S)$, with an admissible exact first row.
Let $Q$ be the Fr\'echet $\mc S$-module $\{ (v,n) \in V_2 \oplus N \mid p(v) = h(n) \} =
\ker ( p - h : V_2 \oplus N \to V_3)$. Then
\[
\ker (Q \to N) \cong \ker p \cong V_1 \in \mr{Mod}_{\DNO}(\mc S) ,
\]
so $Q \in \mr{Mod}_\Omega (\mc S)$ by part (a). \\
(d) For every $M \in \mr{Mod}_\Omega (\widehat{\mc S}_{\mc W \xi})$ there is an isomorphism
\[
\mr{Hom}_{\mr{Mod}_{\Omega}(\widehat{\mc S}_{\mc W \xi})} (\widehat{\mc S}_{\mc W \xi}
\widehat{\otimes} F,M) \cong \mr{Hom}_{Fr\acute{e}} (F,M) .
\]
By Theorem \ref{thm:Vogt}.c the right hand side is exact as a functor of $M \in \Omega$, so
$\widehat{\mc S}_{\mc W \xi} \widehat{\otimes} V$ is projective in $\mr{Mod}_{\Omega}
(\widehat{\mc S}_{\mc W \xi})$. The same proof applies to $\mc S \widehat{\otimes} F$ as an
object of $\mr{Mod}_{\DNO}(\mc S)$ and of $\mr{Mod}_\Omega (\mc S)$. In $\mr{Mod}_{bor}(\mc S)$
exact sequences are required to have linear splittings, which implies that topologically free modules
are projective. \\
(e) For any $(\pi,V) \in \mr{Mod}_{\DNO}(\mc S)$
we have the canonical $\mc S$-module surjection
\[
\mc S \widehat{\otimes} V \to V : s \otimes v \mapsto \pi (s) v ,
\]
which is linearly split by $v \mapsto 1 \otimes v$. Its kernel $N$ is (as Fr\'echet space)
a direct summand of $\mc S \widehat{\otimes} V$, which in turn is a direct summand of
$\mc S \widehat{\otimes} \mc S(\Z)$.
Thus $N \in \mr{Mod}_{\DNO}(\mc S)$ and we can build a projective resolution
\begin{equation*}
0 \leftarrow V \leftarrow \mc S \widehat{\otimes} V \leftarrow
\mc S \widehat{\otimes} N \leftarrow \cdots
\end{equation*}
Let $M$ be a finitely generated Fr\'echet $\widehat{\mc S}_{\mc W \xi}$-module, say it is
a quotient of $\widehat{\mc S}_{\mc W \xi}^n$. Since $\widehat{\mc S}_{\mc W \xi}$ is Noetherian,
every submodule of a finitely generated $\widehat{\mc S}_{\mc W \xi}$-module is again finitely
generated. This applies in particular to the kernel of the quotient map
$\widehat{\mc S}_{\mc W \xi}^n \to M$, so we can construct a projective resolution
with free $\widehat{\mc S}_{\mc W \xi}$-modules of finite rank.\\
(f) By part (d) and Lemma \ref{lem:fgmod}, the projective resolutions from part (e) remain
projective in $\mr{Mod}_{bor}(\mc S)$ and in $\mr{Mod}_{\Omega}(\mc S)$, respectively
in $\mr{Mod}_{\Omega}(\widehat{\mc S}_{\mc W \xi})$ and in
$\mr{Mod}(\widehat{\mc S}_{\mc W \xi})$. Furthermore part (b) guarantees that
such a resolution of $\mc S$-modules admits a continuous linear splitting, so it is admissible as
a resolution of bornological modules. It is well-known that the Yoneda Ext-groups in an exact
category agree with the derived functors of Hom when both are defined, see for example
\cite[Theorem 6.4]{Mac}.
\end{proof}

\subsection{Exactness of formal completion}
\label{subsec:exact}

\begin{thm}\label{thm:exact}
The functor $\widehat{S}_{\mc W \xi} \widehat{\otimes}_{\mc S} : \mr{Mod}_\Omega (\mc S)
\to \mr{Mod}_\Omega (\widehat{S}_{\mc W \xi})$ is exact.
\end{thm}
\begin{proof}
We have to show that the image of any short exact sequence under this functor
is again a short exact sequence. Since
$\widehat{\mc S}_{\mc W \xi} \widehat{\otimes}_{\mc S} V \cong
\widehat{\mc{Z}}_{\mc{W}\xi} \widehat\otimes_{\mathcal{Z}} V$ for all
$V \in \mr{Mod}_{\Omega}(\mc S) ,\; \widehat{S}_{\mc W \xi} \widehat{\otimes}_{\mc S}$
is exact if and only if
$\mr{Tor}_1^{\mr{Mod}_{\Omega}(\mc Z)}(\widehat{\mc Z}_{\mc W \xi},V) = 0$ for all
$V \in \mr{Mod}_\Omega(\mc S)$. Because $\mc S$ is finitely generated as a $\mc Z$-module
and $V \in \Omega$, there exists a short exact sequence of Fr\'echet $\mc Z$-modules
\begin{equation}\label{eq:resZV}
0 \to R \to \mc Z \widehat{\otimes} \mc S (\Z) \to V \to 0 .
\end{equation}
The analogue of Lemma \ref{lem:catCS}.d for $\mc Z$ shows that
$\mc Z \widehat{\otimes} \mc S (\Z)$ is projective in $\mr{Mod}_{\Omega}(\mc Z)$.
By the definition of the torsion functor and by Lemma \ref{lem:compl}.a
\begin{multline*}
\mr{Tor}_1^{\mr{Mod}_{\Omega}(\mc Z)}(\widehat{\mc Z}_{\mc W \xi},V) =
\ker \big( \widehat{\mc Z}_{\mc W \xi} \widehat{\otimes}_{\mc Z} R \to
\widehat{\mc Z}_{\mc W \xi} \widehat{\otimes} \mc S (\Z) \big) \cong \\
\ker \Big( \frac{R}{\overline{m^\infty_{\mc W \xi} R}} \to
\frac{\mc Z \widehat{\otimes} \mc S (\Z)}{m^\infty_{\mc W \xi}
\widehat{\otimes} \mc S (\Z)} \Big) \cong \frac{(m^\infty_{\mc W \xi}
\widehat{\otimes} \mc S (\Z)) \cap R}{\overline{m^\infty_{\mc W \xi} R}} .
\end{multline*}
Thus we have to show that
\begin{equation}\label{eq:minftyR}
(m^\infty_{\mc W \xi} \widehat{\otimes} \mc S (\Z)) \cap R
\; \subset \; \overline{m^\infty_{\mc W \xi} R},
\end{equation}
which we will do with a variation on \cite[Chapitre 1]{TouMer}. Via the Fourier transform
$\mc S (\Z)$ is isomorphic to $C^\infty (S^1)$, so
\[
\mc Z \widehat{\otimes} \mc S (\Z) \cong C^\infty (\Xi_{un})^{\mc W} \widehat{\otimes}
C^\infty (S^1) \cong C^\infty (\Xi_{un} \times S^1 )^{\mc W} .
\]
Under this isomorphism $m^\infty_{\mc W \xi} \widehat{\otimes} \mc S (\Z)$ corresponds to the
closed ideal $m^\infty_{\mc W \xi \times S^1} \subset C^\infty (\Xi_{un} \times S^1 )^{\mc W}$
of functions that are flat on $\mc W \xi \times S^1$.

Let $\phi \in  (m^\infty_{\mc W \xi} \widehat{\otimes} \mc S (\Z)) \cap R$. Tougeron and Merrien
\cite[p. 183--185]{TouMer} construct a function $\psi \in m^\infty_{\mc W \xi \times S^1}$ which
is strictly positive outside $\mc W \xi \times S^1$ and divides $\phi$ in
$C^\infty (\Xi_{un} \times S^1 )^{\mc W}$. We take a closer look at this construction.
The first thing to note is that \cite{TouMer} works with smooth functions on manifolds, not
on orbifolds like $\Xi_{un}/ \mc W \times S^1$. But this is only a trifle, for we can construct
a suitable $\tilde \psi \in C^\infty (\Xi_{un} \times S^1)$ and average it over $\mc W$.
Secondly we observe that, because $S^1$ is compact, we can take all steps from \cite{TouMer}
with $S^1$-invariant functions on $\Xi_{un} \times S^1$. Thus we obtain a series of
$\mc W \times S^1$-invariant functions $\sum_i \epsilon_i$ which converges to a $\psi$
as above. (Compared to \cite{TouMer} we include a suitable power of 1/2 in $\epsilon_i$.)
Moreover the support of $\epsilon_i$ is disjoint from $\mc W \xi \times S^1$, so
\[
\epsilon_i / \psi \in m^\infty_{\mc W \xi} \subset C^\infty (\Xi_{un})^{\mc W} \subset
C^\infty (\Xi_{un} \times S^1) .
\]
Since $R$ is a $\mc Z$-submodule of $C^\infty (\Xi_{un} \times S^1)^{\mc W} ,\;
\phi \epsilon_i / \psi \in m^\infty_{\mc W \xi} R$ for all $i$. Hence
\[
\phi = \psi \, \phi / \psi = \sum\nolimits_i \phi \epsilon_i / \psi
\in \overline{m^\infty_{\mc W \xi} R} ,
\]
which proves \eqref{eq:minftyR}.
\end{proof}

\begin{cor}\label{cor:extloc}
Let $M,N \in \mr{Mod}_{\Omega}(\widehat{\mc S}_{\mc W \xi})$ and let $n \in \Z_{\geq 0}$.
\enuma{
\item There is an isomorphism of $\widehat{\mc Z}_{\mc W \xi}$-modules
\[
\mr{Ext}^n_{\mr{Mod}_\Omega (\mc S)}(q^* N, q^* M) \cong
\mr{Ext}^n_{\mr{Mod}_\Omega (\widehat{\mc S}_{\mc W \xi})}(N, M) .
\]
\item Let $V \in \mr{Mod}_{\DNO}(\mc S)$ be such that $\widehat{V}_{\mc W \xi}$
is finitely generated over $\widehat{\mc S}_{\mc W \xi}$.
There are isomorphisms of $\widehat{\mc Z}_{\mc W \xi}$-modules
\[
\mr{Ext}^n_{\mr{Mod}_{bor}(\mc H)} (V,q^* M) \cong
\mr{Ext}^n_{\mr{Mod}_{bor}(\mc S)}(V,q^* M) \cong
\mr{Ext}^n_{\widehat{\mc S}_{\mc W \xi}} (\widehat{V}_{\mc W \xi}, M) .
\]
\item In case $V$ has finite dimension the groups from (b) are also isomorphic to
$\mr{Ext}^n_{\mc H} (V,q^* M)$ and $\mr{Ext}^n_{\mc S}(V,q^* M)$, that is, to the
Ext-groups defined in the category of all (algebraic) $\mc H$- or $\mc S$-modules.
}
\end{cor}
\begin{proof}
The $\widehat{\mc Z}_{\mc W \xi}$-module structure is defined via
\[
\widehat{\mc Z}_{\mc W \xi} \to \mr{End}_{\mc H}(q^* M) = \mr{End}_{\mc S}(q^* M) =
\mr{End}_{\widehat{\mc S}_{\mc W \xi}}(M)
\]
and right multiplication of $\mr{Ext}^n_A (?,M)$ with $\mr{Ext}^0_A (M,M)$. The asserted
isomorphisms preserve this additional structure because in the below proof $M$ is kept constant.\\
(a) Consider the endofunctor $E$ of $\mr{Mod}_\Omega (\mc S)$ defined by $E(V) =
q^* (\widehat{S}_{\mc W \xi} \widehat{\otimes}_{\mc S} V)$. By Lemma \ref{lem:compl}.b it is
idempotent and by Theorem \ref{thm:exact} it is exact. Given any higher extension
\[
V_* :=  \big( V_0 = q^* M \to V_1 \to \cdots \to V_n \to V_{n+1} = q^* N \big),
\]
$E (V_*)$ is again an $n$-fold extension of $q^* N$ by $q^* M$ and it comes with a homomorphism
$V_* \to E (V_*)$. Hence $V_*$ and $E (V_*)$ determine the same equivalence class in
$\mr{Ext}^n_{\mr{Mod}_\Omega (\mc S)} (q^* N, q^* M)$ and
$q^* : \mr{Mod}_\Omega (\widehat{\mc S}_{\mc W \xi}) \to \mr{Mod}_\Omega (\mc S)$
induces the required isomorphism. \\
(b) The first isomorphism is a special case of \cite[Corollary 3.7.a]{OpSo1}. 

Let $P_*$ be a projective
resolution of $V$ in $\mr{Mod}_{\DNO}(\mc S)$. We may assume that $P_n = \mc S \widehat{\otimes} F_n$
for some $F_n \in \DNO$. By Theorem \ref{thm:exact} $\widehat{S}_{\mc W \xi} \widehat{\otimes}_{\mc S} P_*
= \widehat{S}_{\mc W \xi} \widehat{\otimes} F_*$ is a projective resolution of $\widehat{V}_{\mc W \xi}$
in $\mr{Mod}_\Omega (\widehat{S}_{\mc W \xi})$. By Lemma \ref{lem:compl}.d and Lemma \ref{lem:catCS}.d
\begin{align*}
\mr{Ext}^n_{\mr{Mod}_{bor}(\mc S)}(V,q^* M) & 
\cong H^n \big( \mr{Hom}_{\mc S}(\mc S \widehat{\otimes} F_*,q^* M) \big) \\
& \cong  H^n \big( \mr{Hom}_{\widehat{\mc S}_{\mc W \xi}}
(\widehat{S}_{\mc W \xi} \widehat{\otimes} F_*,M) \big) \\
& \cong \mr{Ext}^n_{\mr{Mod}_\Omega (\widehat{\mc S}_{\mc W \xi})} (\widehat{V}_{\mc W \xi}, M) .
\end{align*}
Now let $\widehat{P}_*$ be a projective resolution of $\widehat{V}_{\mc W \xi}$ in
$\mr{Mod}_{fg} (\widehat{\mc S}_{\mc W \xi})$. The proof of Lemma \ref{lem:catCS}.f shows that
\[
\mr{Ext}^n_{\mr{Mod}_\Omega (\widehat{\mc S}_{\mc W \xi})} (\widehat{V}_{\mc W \xi}, M)
\cong H^n \big( \mr{Hom}_{\widehat{\mc S}_{\mc W \xi}}(\widehat{P}_*,M) \big) \cong
\mr{Ext}^n_{\widehat{\mc S}_{\mc W \xi}} (\widehat{V}_{\mc W \xi}, M) .
\]
(c) For finite dimensional $V$ it is explained in the proof of \cite[Corollary 3.7.b]{OpSo1}
that it does not matter whether one works with algebraic or with bornological modules
to define Ext-groups over $\mc H$ or $\mc S$.
\end{proof}
\vspace{1mm}

\section{Algebras with finite group actions}
\label{sec:alginv}

This section is of a rather general nature and does not depend on the rest of the paper.
We start with a folklore result, which we will use several times.
Let $(\pi,V)$ be a finite dimensional complex representation of a finite group $G$.
Then $\mathrm{End}_\C (V)$ becomes a $G$-representation by defining
\[
g \cdot \phi := \pi(g) \circ \phi \circ \pi (g)^{-1} \qquad \phi \in \mathrm{End}_\C (V), g \in G .
\]
\begin{lem}\label{lem:CrossedProduct}
Let $\C [G]$ be the right regular $G$-representation, that is, $\rho (g) h = h g^{-1}$.
Let $B$ be any complex $G$-algebra and endow $B \otimes \mathrm{End}_\C ( \C [G] )$
with the diagonal $G$-action. Then
\[
\big( B \otimes \mr{End}_\C (\C [G]) \big)^G \cong B \rtimes G .
\]
\end{lem}
\begin{proof}
Denote the action of $g \in G$ on $B$ by $\alpha_g$.
For a decomposable tensor $f \otimes g \in B \rtimes G$ we define
$L(f \otimes g) \in B \otimes \mr{End}_\C (\C [G])$ by
\[
L(f \otimes g) (h) = \alpha_{h^{-1} g^{-1}} (f) \otimes gh \qquad h \in G .
\]
It is easily verified that $L(f \otimes g)$ is $G$-invariant and that $L$
extends to an algebra homomorphism
\[
L : B \rtimes G \to \big( B \otimes \mr{End}_\C (\C [G]) \big)^G .
\]
A straightforward calculation shows that $L$ is invertible with inverse
\[
L^{-1}(b) = \sum\nolimits_{g \in G} b(g^{-1})_e \otimes g ,
\]
where $b(g^{-1}) = \sum_{h \in G} b (g^{-1})_h \otimes h \in B \otimes \C [G]$.
\end{proof}

Let $M$ be a smooth manifold, let $T_m (M)$ be the tangent space of
$M$ at $m$ and $T^*_m (M)$ the cotangent space. We suppose that the finite group $G$
acts on $M$ by diffeomorphisms.
We denote the isotropy group of $m \in M$ by $G_m$. For any finite dimensional $G_m$-representation
$\sigma$  we define the $C^\infty (M) \rtimes G$-representation
\[
\mathrm{Ind}_m \sigma = \mathrm{Ind}_{C^\infty (M) \rtimes G_m}^{C^\infty (M) \rtimes G} \sigma_m ,
\]
where $\sigma_m$ is $\sigma$ regarded as a $C^\infty (M) \rtimes G_m$-representation with
$C^\infty (M)$-character~$m$.

We let $V$ be a finite dimensional $G$-representation and we endow the algebra
$C^\infty (M) \otimes \mr{End}_\C (V)$ with the diagonal $G$-action. Algebras of the form
\begin{equation}\label{eq:3.1}
\big( C^\infty (M) \otimes \mr{End}_\C (V) \big)^G
\end{equation}
were studied among others in \cite{Sol-Chern} and \cite[Section 2.5]{Sol-Thesis}.
Let $\Fs_m$ be the Fr\'echet algebra of formal power series at $m \in M$, with complex coefficients.
Then $\mc F_{Gm} := \bigoplus_{m' \in Gm} \Fs_{m'}$ is naturally a $G$-algebra, so we can
again construct
\[
\big( \Fs_{Gm} \otimes \mr{End}_\C (V) \big)^G .
\]
In the proof of Theorem \ref{thm:Ext} we will use some results for such algebras,
which we prove here.

\begin{thm}\label{thm:Moritaeq}
Let $V'$ be another finite dimensional $G$-representation and assume that, for every irreducible
$G$-representation $W, \; \mr{Hom}_G (W,V) = 0$ if and only if $\mr{Hom}_G (W,V') = 0$.
\enuma{
\item The algebras $\big( C^\infty (M) \otimes \mr{End}_\C (V) \big)^G$ and
$\big( C^\infty (M) \otimes \mr{End}_\C (V') \big)^G$ are Morita-equivalent, both in
the algebraic and in the bornological sense.
\item Let $\sigma$ be a finite dimensional representation of $G_m$. Under the isomorphism from 
Lemma \ref{lem:CrossedProduct} the $C^\infty (M) \rtimes G$-representation $\mr{Ind}_m \sigma$
corresponds to a $\big( C^\infty (M) \otimes \mr{End}_\C (\C [G]) \big)^G$-representation
$\pi (m,\sigma)$ with $C^\infty (M)^G$-character $G m$, such that
\begin{multline*}
\ \qquad \mr{Hom}_{\big( C^\infty (M) \otimes \mr{End}_\C (\C [G]) \big)^G}
\big( \C_m \otimes \C [G], \pi (m,\sigma) \big) \: \cong \: \sigma \\
\cong \: \mr{Hom}_{C^\infty (M) \rtimes G} \big( \mr{Ind}_m \C [G_m], \mr{Ind}_m \sigma \big) .
\end{multline*}
These are isomorphisms of $G_m$-representations with respect to the right regular action 
of $G_m$ on $\C [G]$ and $\C [G_m]$.
\item Parts (a) and (b) remain valid if we replace $C^\infty (M)$ by $\Fs_{Gm}$.
\item The analogous statements hold if $M$ is a nonsingular complex affine
$G$-variety, and we replace $C^\infty (M)$ by the ring $\mc{O}(M)$
of regular functions on $M$ everywhere.
}
\end{thm}
\noindent \emph{Remark.} 
The condition on $V$ and $V'$ is called quasi-equivalence in \cite{LePl},
where it is related to stable equivalence of certain $C^*$-algebras.
\begin{proof}
(a) Consider the bimodules
\begin{equation}\label{eq:bimodules}
\begin{array}{lll}
D_1 & = & C^\infty (M) \otimes \mr{Hom}_\C (V,V'), \\
D_2 & = & C^\infty (M) \otimes \mr{Hom}_\C (V',V).
\end{array}
\end{equation}
It is clear that
\begin{equation}\label{eq:13}
D_1 \otimes_{C^\infty (M) \otimes \mr{End}_\C (V)} D_2 \cong C^\infty (M) \otimes \mr{End}_\C (V') ,
\end{equation}
and similarly in the reversed order. We will show that \eqref{eq:13} remains valid if we take
$G$-invariants everywhere. In other words, we claim that the bimodules $D_1^G$ and $D_2^G$
implement the desired Morita equivalence. First we check that $D_1^G$ is projective as a
right $\big( C^\infty (M) \otimes \mr{End}_\C (V) \big)^G$-module. Choose a $G$-representation
$V_3$ such that $V' \oplus V_3 \cong V \otimes \C^d$ for some $d \in \N$ and write
$D_3 = C^\infty (M) \otimes \mr{Hom}_\C (V,V_3)$. Then
\[
D_1^G \oplus D_3^G \cong \big( C^\infty (M) \otimes \mr{Hom}_\C (V,V \otimes \C^d) \big)^G
\cong \big( C^\infty (M) \otimes \mr{End}_\C (V) \big)^G \otimes \C^d ,
\]
so $D_1^G$ is a finitely generated projective
$\big( C^\infty (M) \otimes \mr{End}_\C (V) \big)^G$-module, in both the algebraic and the
bornological sense. Similarly $D_2^G$ is a finitely generated projective
$\big( C^\infty (M) \otimes \mr{End}_\C (V') \big)^G$-module, so all the below algebraic tensor
products are also bornological tensor products.

Multiplication yields a natural map
\begin{equation}\label{eq:14}
D_1^G  \otimes_{\big( C^\infty (M) \otimes \mr{End}_\C (V) \big)^G} D_2^G  \to
\big( C^\infty (M) \otimes \mr{End}_\C (V') \big)^G .
\end{equation}
Take any $m \in M$ and let $I_{Gm} \subset C^\infty (M)^G$ be the maximal ideal
$\{ f \in C^\infty (M)^G \mid f(m) = 0 \}$. Dividing out \eqref{eq:14} by $I_{Gm}$, we obtain
\begin{equation}\label{eq:15}
\mr{Hom}_{G_m} (V,V') \otimes_{\mr{End}_{G_m} (V')} \mr{Hom}_{G_m} (V',V) \to
\mr{End}_{G_m} (V) .
\end{equation}
Decompose the $G_m$-representations $V$ and $V'$ as
\[
V \cong \bigoplus_{\rho \in \mr{Irr}(G_m)} \rho \otimes \C^{m_\rho} ,\qquad
V' \cong \bigoplus_{\rho \in \mr{Irr}(G_m)} \rho \otimes \C^{m'_\rho}
\]
The assumption of the theorem implies that $m_\rho = 0$ if and only if $m'_\rho = 0$. As
\[
\mr{End}_{G_m} (V') \cong
\bigoplus_{\rho \in \mr{Irr}(G_m)} \mr{End}_{G_m} (\rho) \otimes \mr{End}_\C (\C^{m'_\rho}) ,
\]
the left hand side of \eqref{eq:15} is isomorphic to
\begin{multline*}
\hspace{-4mm} \bigoplus_{\rho \in \mr{Irr}(G_m)} \hspace{-4mm} \mr{End}_{G_m} (\rho)
\otimes \mr{Hom}_\C (\C^{m_\rho}, \C^{m'_\rho}) \: \otimes_{\mr{End}_{G_m} (V')}  \!
\bigoplus_{\rho \in \mr{Irr}(G_m)} \hspace{-4mm} \mr{End}_{G_m} (\rho) \otimes
\mr{Hom}_\C (\C^{m'_\rho}, \C^{m_\rho}) \\
\cong \bigoplus_{\rho \in \mr{Irr}(G)} \mr{End}_{G_m} (\rho) \otimes
\mr{End}_\C (\C^{m_\rho}) \; \cong \; \mr{End}_{G_m} (V) .
\end{multline*}
Hence \eqref{eq:15} is a bijection for all $m \in M$, which implies that \eqref{eq:14} is injective.
The image of \eqref{eq:14} is a two-sided ideal of $\big( C^\infty (M) \otimes \mr{End}_\C (V')
\big)^G$, which is dense in the sense that for every $m \in M$ the algebra and its ideal have the
same set of values at $m$. Consequently \eqref{eq:14} is bijective, and an isomorphism of
$\big( C^\infty (M) \otimes \mr{End}_\C (V') \big)^G$-bimodules.\\
(b) Let $J_{Gm}$ and $J'_{Gm}$ be the ideals of $C^\infty (M) \rtimes G$ and $\big( C^\infty (M) 
\otimes \mr{End}_\C (\C [G]) \big)^G$ generated by $I_{Gm}$. As the isomorphism from Lemma 
\ref{lem:CrossedProduct} preserves the central characters, the representations under 
consideration factor through the quotients
\begin{multline} \label{eq:3.2}
\big( C^\infty (M) \rtimes G \big) / J_{Gm} \qquad \cong \qquad \C [Gm] \rtimes G \quad 
\cong \quad \mr{End}_\C (\C [Gm]) \otimes \C [G_m] , \\
\big( C^\infty (M) \otimes \mr{End}_\C (\C [G])) \big)^G / J'_{Gm} \cong 
\big( \bigoplus_{g \in G / G_m} \mr{End}_\C (\C [G]) \big)^G  \cong  \mr{End}_{G_m} (\C [G]) .
\end{multline}
The two rightmost algebras are isomorphic via a bijection $G \to Gm \times G_m$ which is
equivariant with respect to the right $G_m$-action. The representations corresponding to
\eqref{eq:3.2} are
\[
\begin{array}{@{\hspace{1cm}}c@{\hspace{4cm}}c@{\hspace{1cm}}c}
\mr{Ind}_m \sigma & \C [Gm] \otimes \sigma & \C [Gm] \otimes \sigma , \\
\pi (m,\sigma) & \C [G] \otimes_{\C [G_m]} \sigma & \C [G] \otimes_{\C [G_m]} \sigma .  
\end{array}
\]
In the special case where $\sigma$ is the (right) regular representation of $G_m$, we see that
$\pi (n,\C [G_m]) \cong \C_m \otimes \C [G]$. Therefore
\begin{multline*}
\mr{Hom}_{C^\infty (M) \rtimes G} \big( \mr{Ind}_m \C [G_m], \mr{Ind}_m \sigma \big) \cong \\
\mr{Hom}_{\big( C^\infty (M) \otimes \mr{End}_\C (\C [G]) \big)^G} 
\big( \C_m \otimes \C [G], \pi (m,\sigma) \big) \cong \\ 
\mr{Hom}_{\mr{End}_{G_m} (\C [G])} \big( \C [G], \C [G] \otimes_{\C [G_m]} \sigma \big) 
\cong \mr{Hom}_{\C [G_m]} (\C [G_m], \sigma) .
\end{multline*}
Notice that all these steps preserve the right action of $G_m$ on $\C [G]$ or $\C [G_m]$.
The last term is $G_m$-equivariantly isomorphic to $\sigma$ via the evaluation map
$\phi \mapsto \phi (1)$. \\
(c) Let $d \in \N$ and divide \eqref{eq:14} out by $I_{Gm}^d$. It remains bijective,
which implies that
\begin{multline*}
\big( \Fs_{Gm} / I_{Gm}^d \otimes \mr{Hom}_\C (V,V') \big)^{G_m}
\otimes_{(\Fs_{Gm} \otimes \mr{End}_\C (V))^G}
\big( \Fs_{Gm} / I_{Gm}^d \otimes \mr{Hom}_\C (V,V') \big)^{G_m} \cong \\
(\Fs_{Gm} / I_{Gm}^d \otimes \mr{End}_\C (V'))^{G_m} .
\end{multline*}
Hence the bimodules $\big( \Fs_{Gm} \otimes \mr{Hom}_\C (V,V') \big)^G$ and
$\big( \Fs_{Gm} \otimes \mr{Hom}_\C (V,V') \big)^G$ provide the required Morita equivalence.
Knowing this, it is obvious that the proof of part (b) also works with formal power
series instead of smooth functions. \\
(d) All the previous arguments go through without change in the affine algebraic case.
\end{proof}

\subsection{Ext-groups and the Yoneda product}

Let $M$ be a nonsingular complex affine variety. It is well known that (commutative)
regular local rings admit Koszul-resolutions and that these can be used to
determine Ext-spaces of suitable modules. This leads for example to
\[
\mr{Ext}^*_{\mc O (M)} (\C_m, \C_m) \cong \wig^* T_m (M) ,
\]
and with a little more work one can show that the Yoneda product on the left hand side
corresponds to the $\wedge$-product on the right hand side. It turns out that this can be
generalized to crossed products with finite groups. We formulate the next result with
complex affine $G$-varieties, but using \cite[Proposition 6]{Was}
and Subsection \ref{subsec:exact} (for completions in the context of Fr\'echet algebras)
it is allowed to replace $\mc{O}(M)$ by smooth functions on a manifold.

\begin{thm} \label{thm:YonedaExt}
Let $V,V'$ be finite dimensional $G_m$-representations.
\enuma{
\item $\mr{Ext}^n_{\mc{O}(M) \rtimes G} (\mr{Ind}_m V, \mr{Ind}_m V')$ is isomorphic to
\[
\mr{Hom}_{G_m} (V \otimes \wig^n T^*_m (M) , V') \cong
\big( \mr{Hom}_\C (V,V') \otimes \wig^n T_m (M) \big)^{G_m} .
\]
\item The Yoneda product on $\mr{Ext}^*_{\mc{O}(M) \rtimes G}
(\mr{Ind}_m V, \mr{Ind}_m V)$ agrees with the usual product on
$\big( \mr{End}_\C (V) \otimes \wig^* T_m (M) \big)^{G_m}$.
\item For the regular representation $V = \C [G_m]$ we have $\mr{Ind}_m V =
\mr{Ind}_{\mc O (M)}^{\mc O (M) \rtimes G} \C_m$ and
\[
\mr{Ext}^*_{\mc{O}(M) \rtimes G} \big( \mr{Ind}_{\mc O (M)}^{\mc O (M) \rtimes G} \C_m,
\mr{Ind}_{\mc O (M)}^{\mc O (M) \rtimes G} \C_m \big) \cong \wig^* T_m (M) \rtimes G_m
\]
as graded algebras.
\item Let $\Fs_{Gm}$ be the completion of $\mc O (M)$ with respect to the powers
of the ideal $I_{Gm} := \{ f \in \mc O (M) : f(gm) = 0 \, \forall g \in G \}$. Endow it with
the Fr\'echet topology of a complex power series ring. Then (a), (b) and (c) remain
valid if we replace Mod$(\mc O (M) \rtimes G)$ by any of the following categories:
$\mr{Mod}(\Fs_{Gm} \rtimes G),\ \mr{Mod}_{Fr\acute{e}}(\Fs_{Gm} \rtimes G),\
\mr{Mod}_{bor}(\Fs_{Gm} \rtimes G)$.
}
\end{thm}
\begin{proof}
(a) First we copy a part of the proof of \cite[Theorem 3.2]{OpSo1}. Let $V_m$ be $V$
considered as a $C^\infty (M) \rtimes G_m$ module with $C^\infty (M)$-character $m$.
By Frobenius reciprocity and because two $\mc{O}(M) \rtimes G_m$-representations
with different central characters admit only trivial extensions
\begin{equation}\label{eq:3.4}
\begin{aligned}
\mr{Ext}^n_{\mc{O}(M) \rtimes G} (\mr{Ind}_m V , \mr{Ind}_m V' )
& \cong \mr{Ext}^n_{\mc{O}(M) \rtimes G_m} (V_m , \mr{Ind}_m V' ) \\
& \cong \mr{Ext}_{\mc{O}(M) \rtimes G_m}^n (V_m ,V'_m) .
\end{aligned}
\end{equation}
Let $\Fs_m$ be the completion of $\mc{O}(M)$ with respect to the powers of the ideal
\[
I_m := \{f \in \mc{O}(M) : f(m) = 0 \} .
\]
It annihilates $V_m$, so
\[
(\Fs_m \rtimes G_m ) \otimes_{\mc O (M) \rtimes G_m} V_m = V_m
\]
as $\mc{O}(M) \rtimes G_m$-modules. Since $\Fs_m$ is flat over $\mc O (M)$
\cite[Theorem 7.2.b]{Eis}, so is $\Fs_m \rtimes G$ over $\mc O (M) \rtimes G$. Therefore
the functor $\Fs_m \rtimes G \otimes_{\mc O(M) \rtimes G} ?$ induces an isomorphism
\begin{equation}
\mr{Ext}^n_{\mc{O}(M) \rtimes G_m} (V_m , V'_m ) \cong
\mr{Ext}^n_{\Fs_m \rtimes G_m} (V_m , V'_m ) .
\end{equation}
Because the $G_m$-module $I_m^2$ has finite codimension in $\mc{O}(M)$,
there exists a $G_m$-submodule $E_m \subset \mc{O}(M)$ such that
\begin{equation}\label{eq:Em}
\mc{O}(M) = \mh C \oplus E_m \oplus I_m^2 .
\end{equation}
As a $G_m$-module $E_m$ is the cotangent space to $M$ at $m$, more or less by
definition of the latter. Since $\Fs_m$ is a local ring we have $\Fs_m E_m = \Fs_m I_m$,
by Nakayama's Lemma. Any finite dimensional $G_m$-module is projective, so
\[
V \otimes \wig^n E_m \otimes \Fs_m = \mr{Ind}_{G_m}^{\Fs_m \rtimes G_m}
\big( V \otimes \wig^n E_m \big)
\]
is a projective $\Fs_m \rtimes G_m$-module, for all $n \in \mh N$. With these modules
we construct a resolution of $V_m$. Define $\Fs_m \rtimes G_m$-module maps
\[
\begin{array}{lll}
\delta_n : V \otimes \wig^n E_m \otimes \Fs_m &
\to & V \otimes \wig^{n-1} E_m \otimes \Fs_m ,\\
\delta_n \, (v \otimes e_1 \wedge \cdots \wedge e_n \otimes f) & = &
\sum\limits_{i=1}^n (-1)^{i-1} v \otimes e_1 \wedge \cdots \wedge
e_{i-1} \wedge e_{i+1} \wedge \cdots \wedge e_n \otimes e_i f ,\\
\delta_0 : V \otimes \Fs_m & \to & V_m ,\\
\delta_0 \, (v \otimes f) & = & f(m) v .
\end{array}
\]
This makes
\begin{equation}\label{eq:3.9}
\big( V \otimes \wig^* E_m \otimes \Fs_m, \delta_* \big)
\end{equation}
into an augmented differential complex. Notice that in
Mod$(\Fs_m )$ this just the Koszul resolution of
\[
V \otimes \Fs_m \big / I_m \Fs_m = V_m .
\]
Therefore \eqref{eq:3.9} is the required projective resolution of $V_m$.
Since $V'_m$ admits the $\Fs_m$-character $m$ it is annihilated by $E_m$,
and in particular $\phi \circ \delta_{n+1} = 0$ for every $\Fs_m \rtimes G_m$-module
homomorphism $\phi : V \otimes \wig^n E_m \otimes \Fs_m \to V'_m$.
(The authors overlooked this important point while writing \cite{OpSo1}.)
Consequently all the differentials in the complex
\[
\big( \mr{Hom}_{\Fs_m \ltimes G_m} \big( V \otimes \wig^* E_m
\otimes \Fs_m , V'_m \big), \mr{Hom}(\delta_* , \mr{id}_{V'_m}) \big)
\]
vanish and
\begin{equation} \label{eq:ExtVm}
\begin{aligned}
\mr{Ext}_{\mc{O}(M) \rtimes G_m}^n (V_m ,V'_m) & =
\mr{Hom}_{\Fs_m \ltimes G_m} \big( V \otimes \wig^n E_m \otimes \Fs_m , V'_m \big) \\
& \cong \mr{Hom}_{G_m} (V \otimes \wig^n E_m , V') .
\end{aligned}
\end{equation}
Recall that $E_m$ can be identified with $T^*_m (M)$, so the dual space of $\wig^n E_m$ is
$\wig^n T_m (M)$. Thus we can express the right hand side of \eqref{eq:ExtVm} as
\[
\mr{Hom}_{G_m} (V, \wig^n T_m (M) \otimes V') \cong
\big( \mr{Hom}_\C (V,V') \otimes \wig^n T_m (M) \big)^{G_m} .
\]
(b) We recall from \cite[Section III.6]{Mac} how the Yoneda product can be defined using
resolutions of modules. Let $Y$ be a module over any unital ring $R$ and let $P_* \to Y$
be a projective resolution. By definition $\mr{Ext}^n_R (Y,Y) = H^n (\mr{Hom}_R (P_* ,Y))$.
Given any $\phi \in \mr{Hom}_R (P_n,Y)$ with $\phi \circ \delta_{n+1} = 0$,
the projectivity of $P_n$ guarantees that we can extend it to a chain map
$\phi_* \in \mr{Hom}_R (P_* , P_*)$ of degree $-n$:
\[
\xymatrix{
\ar[r] & P_{n+1} \ar[r] & P_n \ar[r] & \cdots \ar[r] & P_1 \ar[r] & P_0 \ar[r] &
\cdots \ar[r] & 0 \ar[r] & 0\\
\ar[r] & P_{n+1} \ar[r] \ar[urrr]_{\phi_1} & P_n \ar[r] \ar[urrr]_{\phi_0} &
\cdots \ar[r] & P_1 \ar[urrr] \ar[r] & P_0 \ar[urrr] \ar[r] & \cdots \ar[r] & 0 \ar[r] & 0 }
\]
By \cite[Theorem III.6.1]{Mac} $\phi_*$ is determined by $\phi$ up to chain homotopy,
so $\mr{Ext}^n_R (Y,Y)$ can be regarded as $\mr{End}(P_*,P_*)$ modulo homotopy. Now
$\mr{End}(P_*,P_*)$ is a graded algebra, and by \cite[Theorem III.6.4 and exercise III.6.2]{Mac}
its multiplication corresponds to the Yoneda product on $\mr{Ext}^*_R (Y,Y)$.

In our setting, by part (a) every
\[
\phi \in \mr{Ext}^n_{\mc{O}(M) \rtimes G} (\mr{Ind}_m V , \mr{Ind}_m V ) \cong
\mr{Ext}^n_{\Fs_m \rtimes G_m} (V_m , V_m )
\]
can be taken of the form
\[
\phi = \sum\nolimits_j \psi_j \otimes \lambda_j \in
\big( \mr{End}_\C (V) \otimes \wig^n T_m (M) \big)^{G_m} .
\]
Above we showed that \eqref{eq:3.9} is a projective resolution of $V_m$, so we want to lift
$\phi$ to $\phi_* \in \mr{End}_{\Fs_M \rtimes G_m} (V \otimes \wig^* E_m \otimes \Fs_m)$.
We claim that this can be done by defining
\[
\phi_{m+n}(v \otimes \omega \otimes f) =
\sum\nolimits_j \psi_j (v) \otimes \imath (\lambda_j) \omega \otimes f ,
\]
where $\imath (\lambda_j) \omega$ means that $\lambda_j$ is inserted in the last
$n$ entries of $\omega \in \wig^{n+m} E_m \cong \wig^{n+m} T^*_m (M)$.
It is clear that $\phi_*$ is $\Fs_m \rtimes G_m$-linear whenever $\phi$ is so.
To check that $\phi_*$ is a chain map, we may just as well work inside
$\mr{End}_{\Fs_m} (V \otimes \wig^* E_m \otimes \Fs_m)$. Then we can reduce the
verification to $\phi$ of the form $\psi \otimes t_1 \wedge \cdots \wedge t_n$ with
$\psi \in \mr{End}_\C (V)$ and $t_i \in T_m (M)$ linearly independent. Furthermore
it suffices to consider $\omega = e_1 \wedge \cdots \wedge e_{n+m}$ for $e_k \in
E_m$ such that $\inp{e_k}{t_i} = 0$ for all $k \leq m$ and all $i$. Now we calculate
\begin{align*}
\delta_m \phi_m & (v \otimes \omega \otimes f) = \delta_m \big( \psi (v) \otimes e_1 \wedge \cdots
\wedge e_m \det \big( \inp{e_{k+m}}{t_i} \big)_{i,k=1}^n \otimes f \big) \\
& = \psi (v) \otimes \sum_{i=1}^m (-1)^i e_1 \wedge \cdots \wedge e_{i-1} \wedge e_{i+1}
\wedge \cdots \wedge e_n \det \big( \inp{e_{k+m}}{t_i} \big)_{i,k=1}^n \otimes e_i f \\
& = \psi (v) \otimes \sum_{i=1}^{m+n} (-1)^{i-1} \imath (t_1 \wedge \cdots \wedge t_n)
e_1 \wedge \cdots \wedge e_{i-1} \wedge e_{i+1}
\wedge \cdots \wedge e_{n+m} \otimes e_i f \\
& = \phi_{m-1} \big( \sum\limits_{i=1}^{n+m} (-1)^{i-1} v \otimes e_1 \wedge \cdots \wedge
e_{i-1} \wedge e_{i+1} \wedge \cdots \wedge e_{n+m} \otimes e_i f \big) \\
& = \phi_{m-1} \delta_{m+n} (v \otimes \omega \otimes f) .
\end{align*}
With these explicit lifts available, we can determine the Yoneda product. Since
\[
(\psi_j \otimes \imath (\lambda_j) \otimes \mr{id}_{\Fs_m}) \circ
(\psi'_k \otimes \imath (\lambda'_k) \otimes \mr{id}_{\Fs_m}) =
\psi_j \psi'_k \otimes \imath (\lambda_j \wedge \lambda'_k) \otimes \mr{id}_{\Fs_m}
\]
we obtain
\[
\phi \circ \phi' = \big( \sum_j \psi_j \otimes \lambda_j \big) \circ \big( \sum_k \psi'_k \otimes
\lambda'_k \big)  = \sum_{j,k} \psi_j \psi'_k \otimes \lambda_j \wedge \lambda'_k .
\]
(c) This follows from (b) and Lemma \ref{lem:CrossedProduct}.\\
(d) By the Chinese remainder theorem $\Fs_{Gm} \cong \bigoplus_{m' \in Gm} \Fs_{m'}$, so
\eqref{eq:3.4} becomes
\[
\mr{Ext}^n_{\Fs_{Gm} \rtimes G} (\mr{Ind}_m V , \mr{Ind}_m V' )
\cong \mr{Ext}_{\Fs_m \rtimes G_m}^n (V_m ,V'_m) .
\]
Thus the proof of (a), (b) and (c) also applies to $\Fs_{Gm} \rtimes G$. The remaining
subtlety is that in $\mr{Mod}_{bor}(\Fs_{m} \rtimes G_m)$ exact sequences are admissible if
and only if they admit a bounded linear splitting. Hence the free modules in this category are of
the form $\Fs_{m} \rtimes G_m \widehat{\otimes}_\C F$, where $F$ is any bornological vector
space. Since $V \otimes \bigwedge^* E_m$ is a finite dimensional $G_m$-representation, it
is a direct summand of $\C [G_m]^d$ for some $d \in \N$. Hence all the modules
$V \otimes \bigwedge^n E_m \otimes \Fs_m$ of the Koszul resolution of $V_m$ are direct
summands of $\C [G_m]^d \otimes \Fs_m = (\Fs \rtimes G_m) \widehat{\otimes}_\C \C^d$,
which means that they are projective in both the algebraic and the bornological sense.

Futhermore we have to check that the projective resolution \eqref{eq:3.9} admits a
bounded linear splitting. Since we dealing with Fr\'echet spaces, bounded is the same
as continuous. Via \eqref{eq:Em} we can identify $\Fs_m$ with the formal completion
of the symmetric algebra of $E_m$. This enables us to define the De Rham differential
$D : \Fs_m \to \bigwedge^1 E_m \otimes \Fs_m$, which is easily seen to be continuous
with respect to the Fr\'echet topology. We define
$D_n : V \otimes \wig^n E_m \otimes \Fs_m \to V \otimes \wig^{n+1} E_m \otimes \Fs_m$ by
\[
D_n (v \otimes e_1 \wedge \cdots \wedge e_n \otimes f) =
(-1)^n v \otimes e_1 \wedge \cdots \wedge e_n \wedge Df ,
\]
and on the augmentation $V_m$ we take $D_{-1} (v) = v \otimes 1$.
With a straightforward calculation one checks that
\[
(D_{n-1} \delta_n + \delta_{n+1} D_n) ((v \otimes e_1 \wedge \cdots \wedge e_n \otimes f)
= (n+g) v \otimes e_1 \wedge \cdots \wedge e_n \otimes f
\]
for $f \in \Fs_m$ homogeneous of degree $g$, except for $n=g=0$, in which case
\[
(D_{-1} \delta_0 + \delta_1 D_0)(v \otimes 1) = v \otimes 1 .
\]
Since our resolution is the direct product of differential complexes indexed by the
degree $n+g$, this shows that $D_* \delta_* + \delta_* D_*$ is invertible. Hence
$D$ is the desired continuous contraction.
\end{proof}
\vspace{1mm}

\section{Analytic R-groups}
\label{sec:Rgroup}

We recall the definition of the analytic R-group from \cite{DeOp2}.
Let $\delta$ be a discrete series representation of $\mc H_P$ with central character
$W(R_P) r_\delta \in T_P / W(R_P)$. Recall that a \emph{parabolic subsystem}
of $R_0$ is a subset $R_Q \subset R_0$ satisfying $R_Q = R_0 \cap \mathbb{R}R_Q$.
We let $Q$ denote the basis of $R_Q$ inside the
positive subset $R_{Q,+} := R_Q\cap R_{0,+}$. We call $R_Q$ \emph{standard} if
$Q\subset F_0$, in which case these notions agree with \eqref{eq:parabolic}.
The set of roots $R_0 \setminus R_P$ is a disjoint union
of subsets of the form $R_Q \setminus R_P$ where $R_Q\subset R_0$ runs over
the collection $\mc{P}^P_{min}$ of minimal parabolic subsystems
properly containing $R_P$. We define $\alpha_Q\in R_{0,+}$ by $Q=\{P,\alpha_Q\}$ and
$\alpha^P_Q=\alpha_Q \big|_{\mf a^P}\in \mf{a}^{P,*}$ for $R_Q\in\mc{P}^P_{min}$.
By the integrality properties of the root system $R_0$ we see that
\begin{equation}\label{eq:mult}
\{\alpha \big|_{\mf a^P}: \alpha\in R_Q\}\subset \mathbb{Z}\alpha_Q^P .
\end{equation}
Clearly $\alpha_Q^P$ is a character of $T^P$ whose kernel contains the
codimension one subtorus $T^Q\subset T^P$. For each $R_Q\in\mc{P}^P_{min}$
we define a $W (R_P)$-invariant rational function on $T$ by
\begin{equation}
c_Q^P(t):=\prod_{\alpha\in R_{Q,+} \setminus R_{P,+}}c_\alpha (t) ,
\end{equation}
where $c_\alpha : T \to \C$ denotes the rank one $c$-function associated to
$\alpha\in R_0$ and the parameter function $q$ of $\mc{H}$,
see e.g. \cite[Appendix 9]{DeOp1}. This $c_\alpha$ plays an important 
role in the representation theory of $\mc H$, as it governs the behaviour of the
intertwining operator $I_{s_\alpha}$ associated to the reflection $s_\alpha$. 
Roughly speaking, singularities of $I_{s_\alpha}$ correspond to the zeros of 
$c_\alpha$, while $I_{s_\alpha}$ tends to be a scalar operator on principal
series representations parametrized by poles of $c_\alpha$.

For any $\alpha\in R_Q\backslash R_P$ the function $c_\alpha$ is a
nonconstant rational function on any coset of the form $r T^P$. In particular
$c_Q^P$ is regular on a nonempty Zariski-open subset of such a coset.
By the $W (R_P)$-invariance we see that for $t\in T^P$ and $r \in W (R_P ) r_\delta$,
the value of this function at $r t$ is independent of the choice
of $r \in W (R_P) r_\delta$. The resulting rational function $t \mapsto c_Q^P (r t)$
on $T^P$ is clearly constant along the cosets of
the codimension one subtorus $T^Q\subset T^P$.
We define the set of \emph{mirrors} $\mc{M}^{P,\delta}_Q$ associated
to $R_Q \in\mc{P}^P_{min}$ to be the set of connected components of the
intersection of the set of poles of this rational function with the unitary part
$T^P_{un}$ of $T^P$. We put
\begin{equation}
\mc{M}^{P,\delta}=\bigsqcup_{R_Q\in\mc{P}^P_{min}} \mc{M}^{P,\delta}_Q
\end{equation}
for the set of all mirrors in $T^P_{un}$ associated to the pair $(P,\delta)$.
Given $M\in\mc{M}^{P,\delta}$ we denote by $R_{Q^M}\subset R_0$ the
unique element of $\mc{P}^P_{min}$ such that $M\in \mc{M}_{Q^M}^{P,\delta}$.
Thus any mirror $M\in\mc{M}^{P,\delta}$ is a connected component of a
hypersurface of $T^P_{un}$ of the form $\alpha^P_{Q^M}=\textup{constant}$.
Observe that for a fixed pair $(P,\delta)$ the set $\mc{M}^{P,\delta}$ is finite
(and possibly empty).

For every mirror $M\in\mc{M}^{P,\delta}$ there exists,
by \cite[Theorem 4.3.i]{DeOp2}, a unique $s_M \in \mc{W}_{P,\delta}$ (i.e.
$s_M\in \mc{W}_{P,P}$ and $\delta\simeq\delta\circ\psi_{s_M}^{-1}$)
which acts on $T^P_{un}$ as the reflection in $M$.
For $\xi = (P,\delta,t) \in \Xi_{un}$ let $\mc M_\xi$ be the collection of mirrors
$M\in\mc{M}^{P,\delta}$ containing $t$. We define
\begin{equation}
R_\xi =  \{\pm \alpha^P_{Q^M} : M \in \mc M_\xi \} \quad \text{and} \quad
R_\xi^+  =  \{\alpha^P_{Q^M} : M \in \mc M_\xi \} .
\end{equation}
Then it follows from \cite[Proposition 4.5]{DeOp2} and (\ref{eq:mult}) that $R_\xi$ is
a root system in $\mf{a}^{P,*}$ and that $R_\xi^+\subset R_\xi$ is a positive subset.
Its Weyl group $W(R_{\xi})$ is generated by the reflections $s_M$ with $M \in \mc M_\xi$,
so it can be realized as a subgroup of $\mc W_\xi = \{ w \in \mc W : w (\xi) = \xi \}$.
The R-group of $\xi \in \Xi_{un}$ is defined as
\begin{equation}\label{eq:24}
\mf R_\xi = \{ w \in \mc W_\xi : w (R_{\xi}^+ ) = R_{\xi}^+ \} ,
\end{equation}
and by \cite[Proposition 4.7]{DeOp2} it is a complement to $W (R_{\xi})$ in $\mc W_\xi$:
\begin{equation}\label{eq:3}
\mc W_\xi = \mf R_\xi \ltimes W (R_{\xi}) .
\end{equation}
With these notions one can state the analogue of the Knapp--Stein linear independence theorem
\cite[Theorem 13.4]{KnSt} for affine Hecke algebras. For reductive $p$-adic groups the result
is proven in \cite{Sil}, see also \cite[Section 2]{Art}.

\begin{thm}\label{thm:KnappStein}
Let $\xi \in \Xi_{un}$.
\enuma{
\item For $w \in \mc W_\xi$ the intertwiner $\pi (w,\xi)$
is scalar if and only if $w \in W(R_{\xi})$.
\item There exists a 2-cocycle $\kappa_\xi$ (depending on the normalization of the
intertwining operators $\pi (w,\xi)$ for $w\in\mc{W}_{\xi}$) of $\mf R_\xi$
such that $\mr{End}_{\mc H} (\pi (\xi))$
is isomorphic to the twisted group algebra $\C [\mf R_\xi ,\kappa_\xi]$.
\item Given the normalization of the intertwining operators, there is a unique bijection
\[
\mr{Irr}(\C [\mf R_\xi ,\kappa_\xi]) \to \mr{Irr}_{\mc W \xi} ( \mc S ) :
\rho \mapsto \pi (\xi,\rho) ,
\]
such that $\pi (\xi) \cong \bigoplus_\rho \pi (\xi,\rho) \otimes \rho$ as
$\mc H \otimes \C [\mf R_\xi ,\kappa_\xi]$-modules.
\item Let $\mr{Mod}_{f,un,\mc W \xi}(\mc S)$ be the category of finite dimensional unitary
$\mc S$-representations that admit the central character $\mc W \xi$.
The contravariant functor
\[
\begin{array}{lccc}
\Fxi : & \mr{Mod}_{f, un, \mc W \xi} (\mc S) & \to & \mr{Mod}_f (\C [\mf R_\xi ,\kappa_\xi]) ,\\
& \pi & \mapsto & \mr{Hom}_{\mc H}(\pi, \pi (\xi))
\end{array}
\]
is a duality of categories, with quasi-inverse
$\rho \mapsto \pi (\xi,\rho) := \mr{Hom}_{\C [\mf R_\xi ,\kappa_\xi]} (\rho ,\pi (\xi))$.
}
\end{thm}
\begin{proof}
Part (a) is \cite[Theorem 5.4]{DeOp2}, parts (b) and (c) are \cite[Theorem 5.5]{DeOp2} and
part (d) is \cite[Theorem 5.13]{Opd-ICM}.
\end{proof}

Recall that $\C [\mf R_\xi ,\kappa_\xi]$ has a canonical basis $\{T_r \mid r \in \mf R_\xi \}$ 
and that its multiplication is given by
\begin{equation}\label{eq:5.1}
T_r T_{r'} = \kappa_\xi (r,r') T_{r r'} \qquad r,r' \in \mf R_\xi 
\end{equation}
Via a trace map $\Fxi (\pi)$ is naturally identified with the dual space of
\begin{equation}\label{eq:5.2}
\Phi_\xi (\pi) := \mr{Hom}_{\mc H}(\pi (\xi),\pi) = \pi (\xi)^* \otimes_{\mc H} \pi .
\end{equation}
Given that $\Fxi (\pi)$ is a $\C [\mf R_\xi ,\kappa_\xi]$-representation, \eqref{eq:5.1} 
shows that $\Phi_\xi (\pi)$ is a representation of $\C [\mf R_\xi ,\kappa^{-1}_\xi]$. Thus
Theorem \ref{thm:KnappStein}.d is equivalent to saying that the covariant functor
\[
\Phi_\xi : \mr{Mod}_{f, un, \mc W \xi} (\mc S) \to \mr{Mod}_f (\C [\mf R_\xi ,\kappa^{-1}_\xi])
\]
is an equivalence of categories. Its quasi-inverse is 
\begin{equation}\label{eq:5.3}
\rho \mapsto \pi (\xi) \otimes_{\C [\mf R_\xi ,\kappa^{-1}_\xi]} \rho .
\end{equation}
It turns out that Theorem \ref{thm:KnappStein} is compatible with parabolic induction, 
in the following sense. For $P \subset Q \subset F_0$ we have $\xi = (P,\delta,t) \in 
\Xi_{un}^Q \subset \Xi_{un}$. The R-group of $\mc H^Q = \mc H (\mc R^Q ,q)$ at $\xi$ is
\begin{equation}\label{eq:RxiQ}
\mf R_\xi^Q = \{ r \in \mf R_\xi : r \text{ fixes } t T^Q_{un} \text{ pointwise} \} .
\end{equation}
The cocycle $\kappa_\xi^Q : \mf R_\xi^Q \times \mf R_\xi^Q \to \C^\times$
is the restriction of $\kappa_\xi$. By \cite[Proposition 6.1]{Opd-ICM}
\begin{equation}\label{eq:indPhi}
\mr{Ind}_{\C \big[ \mf R_\xi^Q,\kappa_\xi^Q \big]}^{\C [\mf R_\xi,\kappa_\xi]} \circ 
\Phi_\xi^{*Q} = \Fxi \circ \mr{Ind}_{\mc H^Q}^{\mc H}
\end{equation}
as functors $\mr{Mod}_{f, un, \mc W^Q \xi} (\mc S (\mc R^Q,q)) \to
\mr{Mod}_f (\C [\mf R_\xi ,\kappa_\xi])$. This follows from the same argument as for R-groups
in the setting of reductive groups, which can be found in \cite[Section 2]{Art}.

Although it is not needed for our main results, it would be interesting to generalize Theorem
\ref{thm:KnappStein} to non-tempered induction data $\xi \in \Xi \setminus \Xi_{un}$.
We will do that in a weaker sense, only for positive induction data. Sections 5 and 6 do not
depend on the remainder of this section.

One approach would be to define the R-group of $\xi = (P,\delta,t)$
to be that of $\xi_{un} = (P,\delta ,t \,|t|^{-1})$. This works out well in most cases, but it is
unsatisfactory if the isotropy group $\mc W_{\xi_{un}}$ is strictly larger than $\mc W_\xi$.
The problem is that $t \, |t|^{-1}$ is not always a generic point of $(T^P)^{\mc W_\xi}$.
To overcome this, we define $\mc M_\xi$ to be the collection of mirrors $M \in \mc M^{P,\delta}$
that contain not only $t \, |t|^{-1}$, but the entire connected component of $t \, |t|^{-1}$ in
$(T^P_{un})^{\mc W_\xi}$. Notice that all components of $(T^P)^{\mc W_\xi}$ are cosets of a
complex subtorus of $T^P$. Since $t \, |t|^r \in (T^P)^{\mc W_\xi}$ for all $r \in \R ,\; \mc M_\xi$
is precisely the collection of mirrors whose complexification contains $t$.

Now we define $R_{\xi} ,R^+_{\xi}$ and $\mf R_\xi$ as in \eqref{eq:24} and \eqref{eq:3}.
We note that $\mf R_\xi$ may differ from $\mf R_{\xi_{un}}$, but that $R_\xi = R_{P,\delta,t'}$
for almost all $t' \in (T^P_{un})^{\mc W_\xi}$. For positive induction data there is an analogue
of Theorem \ref{thm:KnappStein}:

\begin{thm} \label{thm:KnappSteinNontemp}
Let $\xi = (P,\delta,t) \in \Xi^+$.
\enuma{
\item For $w \in \mc W_\xi$ the intertwiner $\pi (w,\xi)$ is scalar if and only if $w \in W(R_{\xi})$.
\item There exists a 2-cocycle $\kappa_\xi$ of $\mf R_\xi$ such that
$\mr{End}_{\mc H} (\pi (\xi))$ is isomorphic to the twisted group algebra $\C [\mf R_\xi ,\kappa_\xi]$.
\item There exists a unique bijection
\[
\mr{Irr}(\C [\mf R_\xi ,\kappa_\xi]) \to \mr{Irr}_{\mc W \xi} ( \mc H ) :
\rho \mapsto \pi (\xi,\rho) ,
\]
and there exist indecomposable direct summands $\pi_\rho$ of $\pi (\xi)$,
such that $\pi (\xi,\rho)$ is a quotient of $\pi_\rho$ and
$\pi (\xi) \cong \bigoplus_\rho \pi_\rho \otimes \rho$ as
$\mc H \otimes \C [\mf R_\xi ,\kappa_\xi]$-modules.
}
\end{thm}
\noindent \emph{Remark.} 
Part (a) also holds for general $\xi \in \Xi$. Part (d) of Theorem
\ref{thm:KnappStein} does not admit a nice generalization to $\xi \in \Xi^+$.
One problem is that the category $\mr{Mod}_f (\C [\mf R_\xi ,\kappa_\xi])$ is semisimple,
while $\pi (\xi)$ is not always completely reducible. One can try to consider the category of
$\mc H$-representations all whose irreducible subquotients lie in $\mr{Irr}_{\mc W \xi} (\mc H)$,
but then one does not get nice formulas for the functors $\Phi_\xi$ and $\Fxi$.
\begin{proof}
(a) and (b) Since $\xi$ is positive, the operators $\pi (w,\xi)$ with $w \in \mc W_\xi$ span
$\mr{End}_{\mc H} (\pi (\xi))$ \cite[Theorem 3.3.1]{Sol-Irr}.
Let $v \in W(R_{\xi})$. The intertwiner $\pi (v,P,\delta ,t')$ is rational in $t'$ and
by Theorem \ref{thm:KnappStein}.a it is scalar for all $t'$ in a Zariski-dense subset of
$(T^P_{un})^{\mc W_\xi}$. Hence $\pi (v,P,\delta ,t')$ is scalar for all
$t' \in (T^P)^{\mc W_\xi}$, which together with \eqref{eq:3} and \eqref{eq:multIntOp} implies that
\[
\mr{End}_{\mc H} (\pi (\xi)) = \mr{span} \{ \pi (w,\xi) : w \in \mf R_\xi \} .
\]
All the intertwiners $\pi (w,P,\delta,t')$ depend continuously on $t'$, so the type of
$\pi (P,\delta,t')$ as a projective $\mf R_\xi$-representation is constant on connected
components of $(T^P)^{\mc W_\xi}$. Again by Theorem \ref{thm:KnappStein}.a
$\{ \pi (w,\xi) : w \in \mf R_\xi \}$ is linearly independent for generic
$t' \in (T^P_{un})^{\mc W_\xi}$, so it is in fact linearly independent for all
$t' \in (T^P)^{\mc W_\xi}$.
Now the multiplication rules for intertwining operators \eqref{eq:multIntOp} show that
$\mr{End}_{\mc H} (\pi (\xi))$ is isomorphic to a twisted group algebra of $\mf R_\xi$.\\
(c) Let $A = \mr{End}_{\C [\mf R_\xi,\kappa_\xi]} (\pi(\xi))$ be the bicommutant of
$\pi (\xi, \mc H)$ in $\mr{End}_\C (\pi (\xi))$. There is a canonical bijection
\[
\mr{Irr}(\C [\mf R_\xi ,\kappa_\xi]) \to \mr{Irr} (A) : \rho \mapsto \pi_\rho
\]
such that $\pi (\xi) \cong \bigoplus_\rho \pi_\rho \otimes \rho$ as
$A \otimes \C [\mf R_\xi ,\kappa_\xi]$-modules. By construction the irreducible $A$-subrepresentations
of $\pi (\xi)$ are precisely the indecomposable $\mc H$-subrepresentations of $\pi (\xi)$.
By \cite[Proposition 3.1.4]{Sol-Irr} every $\pi_\rho$ has a unique irreducible quotient
$\mc H$-representation, say $\pi (\xi,\rho)$, and $\pi (\xi,\rho) \cong \pi (\xi,\rho')$ if and
only if $\pi_\rho \cong \pi_{\rho'}$ as $\mc H$-representations. Thus
$\rho \mapsto \pi (\xi,\rho)$ sets up a bijection between Irr$(\C [\mf R_\xi ,\kappa_\xi])$ and the
equivalence classes of irreducible quotients of $\pi (\xi)$. By Theorem \ref{thm:parametrizationOfDual}
the latter collection is none other than $\mr{Irr}_{\mc W \xi} (\mc H)$.
\end{proof}
\vspace{1mm}

\section{Extensions of irreducible tempered modules}

In this section we will determine the structure of $\mc S$ in an infinitesimal neighborhood of a
central character $\mc W \xi$, and we will compute the Ext-groups for irreducible
representations of this algebra.

As $\Xi_{un}$ has the structure of a smooth
manifold (with components of different dimensions) we can consider its tangent space
$T_\xi (\Xi_{un})$ at $\xi$. The isotropy group $\mc W_\xi$ acts linearly on $T_\xi (\Xi_{un})$
and for $w \in \mc W_\xi$ we denote the determinant of the corresponding linear map by
$\det (w )_{T_\xi (\Xi_{un})} $. Since the connected component of $\xi$ in $\Xi_{un}$ is
diffeomorphic to $T^P_{un}$, we have
\[
T_\xi (\Xi_{un}) \cong \text{Lie}(T^P_{un}) \cong i \mf a^P \cong i (\mf a / \mf a_P) .
\]
Recall the canonical $\mc W_\xi$-representation on $T_\xi (\Xi_{un})$, and the
decomposition $\mc{W}_\xi=\mc{W}(R_\xi)\rtimes \mf{R}_\xi$.
By Chevalley's Theorem \cite{Che,Hum} the algebra
$\mathbb{R}[T_\xi (\Xi_{un})]^{W (R_{\xi})}$ of real polynomial invariants on
$T_\xi (\Xi_{un})$ for the action of the real reflection group
$W(R_\xi)$ is itself a free real polynomial
algebra on $\dim T^P_{un}$ generators. Let $m_\xi$ be its maximal ideal
of $W(R_\xi)$-invariant polynomials vanishing at $\xi$. Define
\begin{equation}
E_\xi = m_\xi / m_\xi^2 .
\end{equation}
Then $E_\xi$ carries itself a real representation
of $\mf{R}_\xi$. Since the $\mf{R}_\xi$-invariant ideal
$m_\xi^2$ has finite codimension in $\mathbb{R}[T_\xi (\Xi_{un})]^{W (R_{\xi})}$
we can choose a decomposition
\begin{equation}\label{eq:Exidecomp}
\mathbb{R}[T_\xi (\Xi_{un})]^{W (R_{\xi})}=\mathbb{R}\oplus E_\xi\oplus m_\xi^2
\end{equation}
in $\mf{R}_\xi$-stable subspaces, where $\mathbb{R}$ denotes the
subspace of constants. We will identify $E_\xi$ with this
$\mf{R}_\xi$-subrepresentation of $\mathbb{R}[T_\xi (\Xi_{un})]^{W (R_{\xi})}$.
Clearly $E_\xi$ generates the ideal $m_\xi$. Let $S (E_\xi) = \mc O (E_\xi^* \otimes_\R \C)$
be the algebra of complex valued polynomial functions on $E^*_\xi$ and let
$\widehat{S (E_\xi)}$ be its formal completion at $0 \in E^*_\xi$.
Notice that \eqref{eq:Exidecomp} induces $\mf R_\xi$-equivariant isomorphisms
\begin{equation}
\mc{O}(T_\xi(\Xi_{un}))^{W (R_\xi)} \cong S (E_\xi) \quad \text{and} \quad
\mc{F}_{\xi}^{W(R_\xi)} \cong \widehat{S (E_\xi)} .
\end{equation}
Let $\kappa_\xi : \mf R_\xi \times \mf R_\xi \to \C^\times$ be the cocycle from Theorem
\ref{thm:KnappStein}.b and let $\mf R_\xi^*$ be a Schur extension of $\mf R_\xi$.
This object is also known as a representation group \cite[Section 53]{CuRe} and is 
characterized by the property that equivalence classes of irreducible projective 
$\mf R_\xi$-representations are in natural bijection with equivalence classes of
(linear) $\mf R_\xi^*$-representations. 
There is a unique central idempotent $p_\xi \in \C [\mf R_\xi^*]$ such that
\[
\C [\mf R_\xi ,\kappa^{-1}_\xi] \cong p_\xi \C [\mf R_\xi^*]
\]
as algebras and as $\C [\mf R^*_\xi]$-bimodules. Then $p^*_\xi \C[\mf R_\xi^*] =
\C [\mf R_\xi, \kappa_\xi]$. With these notations we can formulate a
strengthening of Theorem \ref{thm:KnappStein}:

\begin{thm}\label{thm:equivCat}
The Fr\'echet algebras $\widehat{\mc S}_{\mc W \xi}$ and $p_\xi \big(\widehat{S (E_\xi)}
\rtimes \mf R^*_\xi \big)$ are Morita equivalent. There are equivalences of exact categories
\[
\mr{Mod}^{\mc W \xi,tor}_{bor}(\mc S) \cong
\mr{Mod}_{bor} \big( \widehat{\mc S}_{\mc W \xi} \big)
\cong \mr{Mod}_{bor} \big( p_\xi (\widehat{S (E_\xi)} \rtimes \mf R^*_\xi) \big) ,
\]
and similarly for Fr\'echet modules. 
\end{thm}
\begin{proof}
Write $\xi = (P,\delta,t)$. By \eqref{eq:FourierIso} and \eqref{eq:25} we may replace $\mc S$ by
\begin{equation}\label{eq:2}
\big( C^\infty (T^P_{un}) \otimes \mr{End}_\C (\C [W^P] \otimes V_\delta ) \big)^{\mc W_{P,\delta}} .
\end{equation}
Let $U_\xi \subset T^P_{un}$ be a nonempty open ball around $t$ such that
\[
\overline{U_\xi} \cap w \overline{U_\xi} = \left\{ \begin{array}{lll}
\overline{U_\xi} & \text{if} & w \in \mc W_\xi \\
\emptyset & \text{if} & w \in \mc W_{P,\delta} \setminus \mc W_\xi .
\end{array} \right.
\]
The localization of \eqref{eq:2} at $U_\xi$ is
\begin{equation}
\begin{split}\label{eq:6}
C^\infty (U_\xi)^{\mc W_\xi} \otimes_{C^\infty (T^P_{un})^{\mc W_{P,\delta}}}
\big( C^\infty (T^P_{un}) \otimes \mr{End}_\C (\C [W^P] \otimes V_\delta ) \big)^{\mc W_{P,\delta}} \\
\cong \big( C^\infty (U_\xi) \otimes \mr{End}_\C (\C [W^P] \otimes V_\delta ) \big)^{\mc W_\xi} .
\end{split}
\end{equation}
The action of $\mc W_\xi$ on $C^\infty (U_\xi) \otimes \mr{End}_\C (\C [W^P] \otimes V_\delta )$
is still defined by \eqref{eq:actSections}, so for $w \in \mc W_\xi$ and $t' \in U_\xi$
\begin{equation}\label{eq:7}
(w \cdot f)(t') = \pi (w,P,\delta, w^{-1}t') f (w^{-1} t') \pi (w,P,\delta, w^{-1}t')^{-1} .
\end{equation}
Let $V_\xi$ be the vector space $\C [W^P] \otimes V_\delta$ endowed with the
$\mc H$-representation $\pi (\xi)$. It is also a projective $\mc W_\xi$-representation, so we can
define a $\mc W_\xi$-action on $C^\infty (U_\xi) \otimes \mr{End}_\C (V_\xi)$ by
\begin{equation}\label{eq:8}
(w \cdot f)(t') = \pi (w,\xi) f (w^{-1} t') \pi (w,\xi)^{-1} .
\end{equation}
The difference between \eqref{eq:7} and \eqref{eq:8} is that $\pi (w,P,\delta, w^{-1}t')$
depends on $t'$, while $\pi (w,\xi)$ does not. Since $U_\xi$ is $\mc W_\xi$-equivariantly
contractible to $t$, we are in the right position to apply \cite[Lemma 7]{Sol-Chern}.
Its proof, and in particular \cite[(20)]{Sol-Chern} shows that the algebra \eqref{eq:6},
with the $\mc W_\xi$-action \eqref{eq:7}, is isomorphic to
\begin{equation}\label{eq:9}
\big( C^\infty (U_\xi) \otimes \mr{End}_\C (V_\xi ) \big)^{\mc W_\xi},
\end{equation}
with respect to the $\mc W_\xi$-action \eqref{eq:8}.
By Theorem \ref{thm:KnappStein}.a the elements of $W (R_{\xi})$ act trivially on
$\mr{End}_\C (V_\xi)$, but nontrivially on $C^\infty (U_\xi)$. Moreover by \eqref{eq:3} 
$W (R_{\xi})$ is normal in $\mc W_\xi$ and $\mc W_\xi / W (R_{\xi}) \cong \mf R_\xi$,
so we can rewrite \eqref{eq:9} as
\begin{equation}\label{eq:10}
\big( C^\infty (U_\xi )^{W (R_{\xi})} \otimes \mr{End}_\C (V_\xi ) \big)^{\mf R_\xi} .
\end{equation}
The formal completion of \eqref{eq:10} at $\xi$ is
\begin{equation}\label{eq:11}
\mc{F}_\xi^{\mc{W}_\xi}\widehat\otimes_{C^\infty (U_\xi )^{\mc{W}_\xi}}
\big( C^\infty (U_\xi )^{W (R_{\xi})} \otimes \mr{End}_\C (V_\xi ) \big)^{\mf R_\xi} .
\end{equation}
By Lemma \ref{lem:compl}.a the latter is isomorphic to
\begin{equation}\label{eq:12}
\Big( (C^\infty (U_\xi )^{W (R_{\xi})} \big/ (\overline{m^{\mc{W}_\xi,\infty}_\xi
C^\infty (U_\xi )^{W (R_{\xi} )}}   ) \otimes \mr{End}_\C (V_\xi ) \Big)^{\mf R_\xi} ,
\end{equation}
where $m^{\mc{W}_\xi,\infty}_\xi\subset C^\infty (U_\xi )^{\mc{W}_\xi}$
denotes the ideal of flat functions at $\xi$. We claim that
\begin{equation}\label{eq:mWxi}
\overline{m^{\mc{W}_\xi,\infty}_\xi C^\infty (U_\xi )^{W (R_{\xi})}} =
m^{W(R_\xi),\infty}_\xi ,
\end{equation}
the ideal of flat functions in $C^\infty (U_\xi )^{W (R_{\xi} )}$. By Whitney's
spectral theorem the right hand side of \eqref{eq:mWxi} is closed, so it contains
the left hand side. Let us prove the opposite inclusion.
By the finiteness of $\mf{R}_\xi$ the ring of $W(R_\xi)$-invariant
polynomials $\mc{O}(T_\xi (\Xi_{un}))^{W (R_{\xi})}$ is finitely generated
as a module over $\mc{O}(T_\xi (\Xi_{un}))^{\mc{W}_\xi}$.
Since $\mc{O}(T_\xi (\Xi_{un}))^{\mc{W}_\xi}$
is Noetherian ring, there exists a finite presentation
\[
\big( \mc{O}(T_\xi (\Xi_{un}))^{\mc{W}_\xi} \big)^{n_2} \xrightarrow{\; A \;}
\big( \mc{O}(T_\xi (\Xi_{un}))^{\mc{W}_\xi} \big)^{n_1} \xrightarrow{\; P \;}
\mc{O}(T_\xi (\Xi_{un}))^{W (R_{\xi})} \to 0.
\]
Let $P_i\in\mc{O}(T_\xi (\Xi_{un}))^{W (R_{\xi})}$ be the image of
the $i$-th basis vector under $P$. Then $A$ is an $n_1 \times n_2$-matrix
with entries in $\mc{O}(T_\xi (\Xi_{un}))^{\mc{W}_\xi}$, expressing
the relations between the generators $P_i$ over $\mc{O}(T_\xi (\Xi_{un}))^{\mc{W}_\xi}$.
Because the formal completion $\mc{F}_\xi^{\mc{W}_\xi}$
of $\mc{O}(T_\xi (\Xi_{un}))^{\mc{W}_\xi}$ at $\xi$ is
flat over $\mc{O}(T_\xi (\Xi_{un}))^{\mc{W}_\xi}$, the relations module of the
generators $P_i$ is also given by the matrix $A$ when working over the formal
power series ring $\mc{F}_\xi^{\mc{W}_\xi}$. In addition one easily checks that
\begin{equation}\label{eq:formal}
\mc{F}_\xi^{\mc{W}_\xi}\otimes_{\mc{O}(T_\xi (\Xi_{un}))^{\mc{W}_\xi}}
\mc{O}(T_\xi (\Xi_{un}))^{W (R_{\xi})} =\mc{F}_\xi^{W (R_{\xi})} .
\end{equation}
Now let $f\in m^{W(R_\xi),\infty}_\xi$. By an argument of Po\'enaru
\cite[p.106]{Poe} the $P_i$ also generate $C^\infty (U_\xi )^{W (R_{\xi})}$ as a
module over $C^\infty (U_\xi )^{\mc{W}_\xi}$, so we may write
$f = \sum^{n_1}_{i=1} \phi_i P_i$ with $\phi = (\phi_i )_{i=1}^{n_1} \in
\big( C^\infty (U_\xi )^{\mc{W}_\xi} \big)^{n_1}$.
By applying the Taylor series homomorphism $\tau_\xi$ at $\xi$
we get $0 = \sum_{i=1}^{n_1} \tau_\xi (\phi_i) P_i$.
It follows from \eqref{eq:formal} and the remarks just above
that there exists a $\widehat\psi\in(\mc{F}_{\xi}^{\mc{W}_\xi})^{n_2}$
such that $\tau_\xi (\phi)=A(\widehat \psi )$. By Borel's Theorem there exist
$\psi\in (C^\infty (U_\xi )^{\mc{W}_\xi})^{n_2}$ such that $\tau_\xi (\psi) = \widehat \psi$.
Writing $\phi^0 := A(\psi) \in (C^\infty (U_\xi )^{\mc{W}_\xi})^{n_1}$, we obtain
\begin{align*}
& \phi_i - \phi_i^0\in
m^{W(R_\xi),\infty}_\xi\cap C^\infty (U_\xi )^{\mc{W}_\xi}=m^{\mc{W}_\xi,\infty}_\xi , \\
& f = \sum_{i=1}^{n_1} \phi_i P_i = \sum_{i=1}^{n_1} (\phi_i - \phi_i^0 ) P_i \in
m^{\mc{W}_\xi ,\infty}_\xi C^\infty (U_\xi )^{W (R_{\xi})} .
\end{align*}
This proves \eqref{eq:mWxi}, from which we conclude that \eqref{eq:12} is isomorphic to
\begin{equation}\label{eq:12.5}
\big( \mc{F}_{\xi}^{W(R_\xi)} \otimes \mr{End}_\C (V_\xi ) \big)^{\mf R_\xi} .
\end{equation}
For compatibility with Section \ref{sec:alginv} we put
$M = E_\xi^* \otimes_\R \C,\; m= 0$ and $G = G_m = \mf{R}_\xi^*$. By Theorem \ref{thm:KnappStein}.c
every irreducible representation of $\C [\mf R_\xi,\kappa_\xi] = p_\xi^* \C [\mf R_\xi^*]$ 
appears at least once in $V_\xi$, so $V_\xi$ contains precisely the same irreducible 
$\mf R_\xi^*$-representations as the right regular representation of $\mf R_\xi^*$ on
$\C [\mf R_\xi,\kappa^{-1}_\xi] = p_\xi \C [\mf R_\xi^*]$. Now parts (a) and (d) of 
Theorem \ref{thm:Moritaeq} say that \eqref{eq:12.5} is Morita equivalent to
\begin{equation}\label{eq:16}
\big( \widehat{S (E_\xi)} \otimes \mr{End}_\C (p_\xi \C [\mf R^*_\xi] ) \big)^{\mf R^*_\xi} 
\end{equation}
Because $p_\xi$ is central and by Lemma \ref{lem:CrossedProduct}, \eqref{eq:16} is isomorphic to
\begin{equation}\label{eq:18}
\widehat{S (E_\xi)} \rtimes p_\xi \C [\mf R^*_\xi] = 
p_\xi \big( \widehat{S (E_\xi)} \rtimes \mf R^*_\xi \big) .
\end{equation}
This establishes the required Morita equivalence. Since the bimodules in the proof of Theorem
\ref{thm:Moritaeq}.a are bornologically projective, this Morita equivalence is an exact functor.
That and parts (c) and (f) of Lemma \ref{lem:compl} yield the equivalences of exact categories.
\end{proof}

The above proof and Theorem \ref{thm:Moritaeq} show that the Morita bimodules are
\begin{equation}\label{eq:5.4}
\Big( \widehat{S (E_\xi)} \otimes \mr{Hom}_\C \big( \pi (\xi), p_\xi \C[\mf R_\xi^*] \big) 
\Big)^{\mf R_\xi^*} \; \text{and} \; \Big( \widehat{S (E_\xi)} \otimes 
\mr{Hom}_\C \big( p_\xi \C[\mf R_\xi^*], \pi (\xi) \big) \Big)^{\mf R_\xi^*} ,
\end{equation}
where $\C [\mf R_\xi^*]$ is considered as the right regular representation of $\mf R_\xi^*$.
The equivalence of categories maps $\pi \in \mr{Mod}_{\mc W \xi}(\mc S)$ to the 
$\widehat{S (E_\xi)} \rtimes \mf R_\xi^*$-module
\[
\mr{Hom}_\C \big( \pi (\xi), p_\xi \C[\mf R_\xi^*] \big)^{\mf R_\xi^*} 
\underset{\mr{End}_\C (\pi (\xi))^{\mf R_\xi^*}}{\otimes} \pi =
p_\xi \mr{Hom}_{\C [\mf R_\xi^*]} \big( \pi (\xi), \C[\mf R_\xi^*] \big)
\underset{\mr{End}_{\C [\mf R_\xi^*]} (\pi (\xi))}{\otimes} \pi .
\]                                                                                         
By construction $\mr{End}_{\C [\mf R_\xi^*]} (\pi (\xi)) = \pi (\xi)(\mc H)$, while 
$\mr{Hom}_{\C [\mf R_\xi^*]} \big( \pi (\xi), \C[\mf R_\xi^*] \big)$ is isomorphic to
the dual space of $\pi (\xi)$ via the map $\C [\mf R_\xi^*] \to \C$ defined
by evaluation at 1. Hence the above module simplifies to
\[
p_\xi \pi (\xi)^* \otimes_{\mc H} \pi = \pi (\xi)^* \otimes_{\mc H} \pi = 
\mr{Hom}_{\mc H}(\pi (\xi),\pi) .
\]
Conversely, consider a $\widehat{S (E_\xi)} \rtimes p_\xi \C [\mf R_\xi^*]$-module 
$\rho$ with central character $0 \in E_\xi^*$. The Morita equivalence sends it to
\begin{multline*}
\mr{Hom}_\C \big( p_\xi \C[\mf R_\xi^*], \pi (\xi) \big)^{\mf R_\xi^*} 
\underset{\mr{End}_\C (p_\xi \C [\mf R_\xi^*])^{\mf R_\xi^*}}{\otimes} \rho =
p_\xi^* \mr{Hom}_{\C [\mf R_\xi^*]} \big( p_\xi \C[\mf R_\xi^*], \pi (\xi) \big)
\underset{p_\xi \C [\mf R_\xi^*]}{\otimes} \rho \\
= p_\xi^* \pi (\xi) \underset{p_\xi \C [\mf R_\xi^*]} \rho =
\pi (\xi) \underset{\C [\mf R_\xi ,\kappa^{-1}_\xi]}{\otimes} \rho .
\end{multline*}
Together with \eqref{eq:5.2} and \eqref{eq:5.3} this confirms that Theorem \ref{thm:equivCat} 
generalizes the covariant version of Theorem \ref{thm:KnappStein}.d.

If $\pi' \in \mr{Mod}_{f, un, \mc W \xi} (\mc S)$, then by \eqref{eq:5.1} 
\begin{equation}\label{eq:1}
\Fxi (\pi) \otimes \Phi_\xi (\pi')
\text{ is a representation of } \C [\mf R_\xi,1],
\end{equation} 
that is, a (linear) representation of the group $\mf R_\xi$.

\begin{thm}\label{thm:Ext}
Let $\pi, \pi'$ be finite dimensional unitary $\mc{S}$-modules with central character
$\mc W \xi$. Let $\Fxi (\pi) = \mr{Hom}_{\mc H}(\pi ,\pi (\xi)) \in 
\mr{Mod}(\mathbb{C}[\mf{R}_\xi,\kappa_\xi])$ and \\
$\Phi_\xi (\pi') = \mr{Hom}_{\mc H}(\pi (\xi),\pi') \in 
\mr{Mod}(\mathbb{C}[\mf{R}_\xi,\kappa^{-1}_\xi])$ be as in Theorem \ref{thm:KnappStein}. Then
\begin{equation}
\mr{Ext}^n_\mathcal{H}(\pi,\pi') \cong
\big( \Fxi (\pi) \otimes_\C \Phi_\xi (\pi') \otimes_\R \wig^n E^*_\xi \big)^{\mathfrak{R}_\xi} .
\end{equation}
For $\pi = \pi'$ the Yoneda-product on $\mr{Ext}^*_{\mc H} (\pi,\pi)$ corresponds to the
usual product on $\big( \mr{End}_\C (\Phi_\xi (\pi)) \otimes_\R \wig^* E^*_\xi \big)^{\mf R_\xi}$.
\end{thm}
\begin{proof}
Our strategy is to consider $\pi$ and $\pi'$ as modules over various algebras $A$, such that
the extension groups $\mr{Ext}_A^n (\pi, \pi')$ for different $A$'s are all isomorphic.
By \cite[Corollary 3.7.b]{OpSo1}
\begin{equation}\label{eq:ExtHS}
\mr{Ext}^n_{\mc H} (\pi,\pi') \cong \mr{Ext}^n_{\mc S} (\pi,\pi') \qquad n \in \Z_{\geq 0}.
\end{equation}
Since $\pi$ and $\pi'$ have finite dimension, it does not matter whether we calculate 
their Ext-groups in the category of algebraic $\mc S$-modules or in any of the categories 
of topological $\mc S$-modules studied in Section \ref{sec:Ext}.

From \eqref{eq:ExtHS} and exactness of localization we see that the first steps of the proof
of Theorem \ref{thm:Moritaeq}, up to \eqref{eq:10}, preserve the Ext-groups.
For the computation of $\textup{Ext}_\mc{S}^n (\pi,\pi')$, Corollary \ref{cor:extloc} allows us
to replace $\mc S$ by its formal completion $\widehat{\mc S}_{\mc W \xi} \cong \Fs_\xi^{\mc W_\xi}
\widehat{\otimes}_{\mc Z} \mc S$ at the central character $\mc{W}\xi$.
This brings us to \eqref{eq:11}, which is Morita equivalent to \eqref{eq:18}. Under these
transformations $\pi$ corresponds to $\Phi_\xi (\pi)$, regarded as a representation of 
\eqref{eq:18} or of $\widehat{S (E_\xi)} \rtimes \mf R_\xi^*$-representation via the 
evaluation at 0. Because \eqref{eq:18} is a direct summand of $\widehat{S (E_\xi)} \rtimes 
\mf R^*_\xi$, we can consider the Ext-groups of $\Phi_\xi (\pi)$ and $\Phi_\xi (\pi')$ just as
well with respect to $\widehat{S (E_\xi)} \rtimes \mf R^*_\xi$. Therefore
\begin{equation}\label{eq:21}
\mr{Ext}_{\mc H}^n (\pi,\pi') \cong
\mr{Ext}^n_{\widehat{S (E_\xi)} \rtimes \mf R^*_\xi} (\Phi_\xi (\pi),\Phi_\xi (\pi')) ,
\end{equation}
and we can apply Theorem \ref{thm:YonedaExt}. Here $\Fs_{Gm} = \Fs_m =
\widehat{S (E_\xi)}$ and $\wig^n_\C T_m (M) \cong \wig^n_\R E_\xi^* \otimes_\R \C$,
so by  Theorem \ref{thm:YonedaExt}.a
\begin{equation}\label{eq:26}
\mr{Ext}_{\mc H}^n (\pi,\pi') \cong
\big( \Fxi (\pi) \otimes_\C \Phi_\xi (\pi') \otimes_\R \wig^n E^*_\xi \big)^{\mathfrak{R}^*_\xi} .
\end{equation}
Moreover $\Fxi (\pi) \otimes_\C \Phi_\xi (\pi')$ is a $\mf R_\xi$-representation, so the entire 
action factors through $\mf R^*_\xi \to \mf R_\xi$. Thus we may just as well use $\mf R_\xi$ 
to determine the invariants in \eqref{eq:26}. The Yoneda product is described by
Theorem \ref{thm:YonedaExt}.b.
\end{proof}
\vspace{1mm}

\section{The Euler--Poincar\'e pairing}

Let $\pi, \pi' \in \mr{Mod}_f (\mc H)$. Their Euler--Poincar\'e pairing is defined as
\begin{equation}\label{eq:4}
EP_{\mc H} (\pi,\pi') = \sum_{n \geq 0} (-1)^n \dim_\C \mr{Ext}_{\mc H}^n (\pi,\pi') .
\end{equation}
Since $\mc H$ is Noetherian and has finite cohomological dimension \cite[Proposition 2.4]{OpSo1},
this pairing is well-defined.
By the main result of \cite{OpSo1}, for $\pi, \pi' \in \mr{Mod}_f (\mc S)$
\begin{equation} \label{eq:5}
EP_{\mc S} (\pi,\pi') := \sum_{n \geq 0} (-1)^n \dim_\C \mr{Ext}_{\mc S}^n (\pi,\pi')
\quad \text{equals} \quad EP_{\mc H}(\pi,\pi') .
\end{equation}
It is clear from the definition \eqref{eq:4} that $EP_{\mc H}(\pi,\pi') = 0$ if $\pi$ and $\pi'$
admit different $Z(\mc H)$-characters. With \eqref{eq:5} we can strengthen this to
$EP_{\mc H}(\pi,\pi') = 0$ whenever $\pi$ and $\pi'$ admit different $Z(\mc S)$-characters.
This is really stronger, because $Z(\mc S)$ is larger than the closure of $Z(\mc H)$ in $\mc S$.
By \eqref{eq:25} any discrete series representation $\delta$ is projective in $\mr{Mod}_f (\mc S)$, so 
\[
EP_{\mc S}(\delta,\pi') = \dim_\C \mr{Hom}_{\mc S}(\delta,\pi') 
\text{ for all } \pi' \in \mr{Mod}_f (\mc S) .
\]
Together with \eqref{eq:5} this shows that $\delta$ also behaves like a projective 
$\mc H$-representation for the Euler--Poincar\'e pairing (over $\mc H)$ of tempered modules, 
something that is completely unclear from \eqref{eq:4}.

For $\ep \in \mh R$ let $q^\ep$ be the parameter function $q^\ep (w) = q(w)^\ep$.
For every $\ep$ we have the affine Hecke algebra $\mc H (\mc R ,q^\ep )$ and its Schwartz
completion $\mc S (\mc R ,q^\ep )$. We note that $\mc H (\mc R ,q^0 ) = \mh C [W]$
is the group algebra of $W$ and that $\mc S (\mc R ,q^0 ) = \mc S (W)$ is the
Schwartz algebra of rapidly decreasing functions on $W$.
The intuitive idea is that these algebras depend continuously on $\ep$.
We will use this in the form of the following rather technical result.

\begin{thm}\label{thm:scalingReps}
\textup{\cite[Corollary 4.2.2]{Sol-Irr}} \\
For $\ep \in [-1,1]$ there exists a family of additive functors
\begin{align*}
& \tilde \sigma_\ep : \mr{Mod}_f (\mc H (\mc R ,q)) \to
\mr{Mod}_f (\mc H (\mc R ,q^\ep )) , \\
& \tilde \sigma_\ep (\pi ,V) = (\pi_\ep ,V) .
\end{align*}
with the properties
\begin{itemize}
\item[(1)] the map
\[
[-1,1] \to \mr{End}\, V : \ep \mapsto \pi_\ep (N_w)
\]
is analytic for any $w \in W$,
\item[(2)] $\tilde \sigma_\ep$ is a bijection if $\ep \neq 0$,
\item[(3)] $\tilde \sigma_\ep$ preserves unitarity,
\item[(4)] $\tilde \sigma_\ep$ preserves temperedness if $\ep \geq 0$,
\item[(5)] $\tilde \sigma_\ep$ preserves the discrete series if $\ep > 0$.
\end{itemize}
\end{thm}

\begin{thm}\label{thm:EP}
\enuma{
\item For all $\pi,\pi' \in \mr{Mod}_f (\mc H)$ and $\ep \in [0,1] $:
\[
EP_{\mc H (\mc R, q^\ep)} (\tilde \sigma_\ep (\pi), \tilde \sigma_\ep (\pi')) = EP_{\mc H} (\pi,\pi') .
\]
\item The pairing $EP_{\mc H}$ is symmetric and positive semidefinite.
\item Suppose $P \subset F_0 ,\; \R P \neq \mf a^*$ and $V \in \mr{Mod}_f (\mc H^P)$.
Then $\mr{Ind}_{\mc H^P}^{\mc H}(V)$ lies in the radical of $EP_{\mc H}$.
}
\end{thm}
\begin{proof}
See Proposition 3.4 and Theorem 3.5 of \cite{OpSo1}. We note that the symmetry
of $EP_{\mc H}$ is not automatic, it is proved via part (a) for $\ep = 0$ and a more detailed
study of $EP_{W} = EP_{\mc H (\mc R ,q^0)}$ \cite[Theorem 3.2]{OpSo1}.
\end{proof}

The pairing $EP_{\mc H}$ extends naturally to a Hermitian form on $G_\C (\mc H)$, say complex
linear in the second argument. By Theorem \ref{thm:EP}.c it factors through the quotient
\[
\mr{Ell} (\mc H) := G_\C (\mc H) \Big/ \sum_{P \subset F_0, \R P \neq \mf a^*}
\mr{Ind}_{\mc H^P}^{\mc H} \big( G_\C (\mc H^P) \big)
\]
Following \cite{Opd-ICM} we call Ell$(\mc H)$ the space of elliptic characters of $\mc H$. Notice
that Ell$(\mc H) = 0$ and $EP_{\mc H} = 0$ if the root datum $\mc R$ is not semisimple.

\begin{lem}\label{lem:Elltemp}
The composite map $G_\C (\mc S) \to G_\C (\mc H) \to \mr{Ell} (\mc H)$ is surjective.
\end{lem}
\begin{proof}
We may and will assume that $\mc R$ is semisimple. We have to show that every irreducible
$\mc H$-representation $\pi$ can be written as a linear combination of tempered representations
and of representations induced from proper parabolic subalgebras.
Recall \cite[Section 2.2]{Sol-Irr} that a Langlands datum is a triple $(P,\sigma,t)$ such that
\begin{itemize}
\item $P \subset F_0$ and $\sigma$ is an irreducible tempered $\mc H_P$-representation;
\item $t \in T^P$ and $|t| \in T^{P++}$.
\end{itemize}
The Langlands classification \cite[Theorem 2.2.4]{Sol-Irr} says that the $\mc H$-representation
$\pi (P,\sigma,t)$ has a unique irreducible quotient $L(P,\sigma,t)$ and that there is
(up to equivalence) a unique Langlands datum such that $L(P,\sigma,t) \cong \pi$.

If $P = F_0$, then $T^P = \{1\}$ because $\mc R$ is semisimple.
So $\mc H_P = \mc H^P = \mc H$ and $\pi \cong \sigma$, which by definition is tempered.

Therefore we may suppose that $P \neq F_0$. Since $\pi (P,\sigma,t)$ is induced from $\mc H^P$,
we have $\pi = L(P,\sigma,t) - \pi (P,\sigma,t)$ in Ell$(\mc H)$. By \cite[Lemma 2.2.6.b]{Sol-Irr}
\[
\pi (P,\sigma,t) - L (P,\sigma,t) \in G_\C (\mc H)
\]
is a sum of representations $L(Q,\tau,s)$ with $P \subset Q$ and
$\norm{cc_P (\sigma)} < \norm{cc_Q (\tau)}$, where $cc_P (\sigma)$ is as in
\eqref{eq:ccdelta}. Given the central character of $\pi$ (an element of $T/W_0$),
there are only finitely many possibilities for $cc_P (\sigma)$. Hence we can deal with
the representations $L(Q,\tau,s)$ via an inductive argument.
\end{proof}

From \cite[Lemma 4.2.3.a]{Sol-Irr} we know that $\tilde \sigma_0 : \mr{Mod}_f (\mc H) \to
\mr{Mod}_f (W)$ commutes with parabolic induction, so it induces a linear map
\begin{equation}
\sigma_{\mr{Ell}} : \mr{Ell}(\mc H) \to \mr{Ell}(W) = \mr{Ell} \big( \mc H (\mc R ,q^0) \big) .
\end{equation}

\begin{thm}\label{thm:sigmaEll}
The pairings $EP_{\mc H}$ and $EP_{W}$ induce Hermitian inner products on respectively
Ell$(\mc H)$ and Ell$(W)$, and the map $\sigma_{\mr{Ell}}$ is an isometric bijection.
\end{thm}
\begin{proof}
By Lemma \ref{lem:Elltemp} and \cite[(3.37)]{Sol-Irr}, $\sigma_{\mr{Ell}}$ is a linear bijection.
According to \cite[Theorem 3.2.b]{OpSo1} $EP_{W}$ induces a Hermitian inner product on
Ell$(W)$, and by Theorem \ref{thm:EP} $\sigma_{\mr{Ell}}$ is an isometry. Therefore the
sesquilinear form on Ell$(\mc H)$ induced by $EP_{\mc H}$ is also a Hermitian inner product.
\end{proof}

\subsection{Arthur's formula}

In this subsection we prove Arthur's formula for the Euler--Poincar\'e pairing \eqref{eq:4}
for tempered representations of affine Hecke algebras \eqref{eq:ArKaScSt}.

Recall the setup of Theorem \ref{thm:Ext}.
The expression $\det (1-w )_{T_\xi (\Xi_{un})} = \det (1-w )_{\mf a^P}$
is analogous to the ``Weyl measure'' in equation \eqref{eq:elliptic} and to $d(r)$
in equation \eqref{eq:Arthur}. Notice that
\[
\det (1-w )_{T_\xi (\Xi_{un})} \geq 0
\]
because the tangent space of $\Xi_{un}$ at $\xi$ is a real representation of the finite 
group $\mc W_\xi$. Clearly $\det (1-w )_{T_\xi (\Xi_{un})} \neq 0$ if and only if $w$ acts 
without fixed points on $T_\xi (\Xi_{un}) \setminus \{0\}$, in which case we say that $w$ 
is elliptic in $\mc W_\xi$.

It is an elementary result in homological algebra that, for the purpose of computing
Euler--Poincar\'e pairings, one may replace any module by its semisimplification.
Hence in the next theorem it suffices to compute $EP_\mc{H}(\pi,\pi')$ for
$\pi,\pi'\in \mr{Mod}_{f,\mc W \xi}(\mc S)$ completely reducible. Recall that irreducible
tempered modules are unitarizable \cite[Corollary 3.23]{DeOp1}. In particular Theorem
\ref{thm:KnappStein}.d applies to $\pi$ and $\pi'$.

\begin{thm}\label{thm:ArthurFormula}
Let $\pi, \pi' \in \mr{Mod}_{f,un,\mc W \xi}(\mc S)$, as in Theorem \ref{thm:KnappStein}.d.
Denote by $\Fxi (\pi) = \mr{Hom}_{\mc H}(\pi ,\pi (\xi))$ and $\Fxi (\pi') =
\mr{Hom}_{\mc H}(\pi', \pi (\xi))$ the corresponding modules of
$\mathbb{C}[\mf{R}_\xi,\kappa_\xi]$. Then
\[
EP_{\mc H} (\pi,\pi') = | \mf R_\xi |^{-1} \sum_{r \in \mf R_\xi}
\det (1-r )_{T_\xi (\Xi_{un})} \mr{tr}_{\Fxi (\pi)}(r) \overline{\mr{tr}_{\Fxi (\pi')}(r)}  .
\]
\end{thm}
\begin{proof}
By \eqref{eq:26}
\begin{align}\label{eq:20}
EP_{\mc H}(\pi,\pi') & = \sum_{n \geq 0} (-1)^n \dim_\C
\big( \Fxi (\pi) \otimes_\C \Phi_\xi (\pi') \otimes_\R \wig^n E^*_\xi \big)^{\mathfrak{R}^*_\xi} \\
\nonumber & = \sum_{n \geq 0} (-1)^n |\mf R^*_\xi |^{-1} \sum_{r \in \mf R^*_\xi}
\mr{tr}_{\Fxi (\pi)}(r) \mr{tr}_{\Phi_\xi (\pi')}(r) \mr{tr}_{\wig^n E_\xi^*}(r) \\
\nonumber & = |\mf R^*_\xi |^{-1} \sum_{r \in \mf R^*_\xi}\mr{tr}_{\Fxi (\pi)}(r)
\overline{\mr{tr}_{\Fxi (\pi')}(r)} \sum_{n \geq 0} (-1)^n \mr{tr}_{\wig^n E_\xi^*}(r) \\
\nonumber & = |\mf R^*_\xi |^{-1} \sum_{r \in \mf R^*_\xi}
\mr{tr}_{\Fxi (\pi)}(r) \overline{\mr{tr}_{\Fxi (\pi')}(r)} \det (1-r )_{E_\xi^*} .
\end{align}
Notice that this formula does not use the entire action of $\mf R^*_\xi$ on
${E_\xi^*}$, only $\det (1-r )_{{E_\xi^*}}$. This determinant is zero
whenever $r \in \mf R^*_\xi$ fixes a nonzero vector in ${E_\xi^*}$. Suppose that
$R_{\xi}$ is nonempty. Then $R_{\xi}^\vee$ is a nonempty root system in $\mf a^P$ and
$\R R_{\xi}^\vee$ can be identified with a subspace of ${E_\xi^*}$.
Every $r \in \mf R^*_\xi$ fixes
$\sum_{\alpha \in R^+_{\xi}} \alpha^\vee \in \mf a^P \setminus \{0\}$, so $r$ fixes
nonzero vectors of $T_t (U_\xi)$ and ${E_\xi^*}$. We conclude that
\[
\det (1-r )_{{E_\xi^*}} = 0 = \det (1-r )_{T_t (U_\xi)}
\]
whenever $R_{\xi}$ is nonempty. Therefore we may always replace
${E_\xi^*}$ by $T_t (U_\xi) = T_\xi (\Xi_{un})$ in \eqref{eq:20}.
So by \eqref{eq:21} and \eqref{eq:20}
\begin{equation}\label{eq:22}
EP_{\mc H} (\pi, \pi') = | \mf R^*_\xi |^{-1} \sum_{r \in \mf R^*_\xi}
\det (1-r )_{T_\xi (\Xi_{un})} \mr{tr}_{\Fxi (\pi)}(r) \overline{\mr{tr}_{\Fxi (\pi')}(r)} .
\end{equation}
Finally we want to reduce from $\mf R^*_\xi$ to $\mf R_\xi$. The action of $\mf R^*_\xi$ on 
$T_\xi (\Xi_{un})$ is defined via the quotient map $\mf R^*_\xi \to \mf R_\xi$, so that is no 
problem. By \eqref{eq:1} the $\mf R^*_\xi$-representation $\Fxi (\pi) \otimes \Phi_\xi (\pi')$ 
is actually a representation of $\mf R_\xi$, with trace
\[
\mr{tr}_{\Fxi (\pi) \otimes \Fxi (\pi')^*} (r) =
\mr{tr}_{\Fxi (\pi)}(r) \overline{\mr{tr}_{\Fxi (\pi')}(r)}.
\]
Hence any two elements of $\mf R^*_\xi$ with the same image in $\mf R_\xi$ give the same
contribution to \eqref{eq:22}.
\end{proof}

The next result follows from Theorem \ref{thm:sigmaEll} but it is also interesting to derive
it from Theorem \ref{thm:ArthurFormula}, since that proof can be generalized to reductive groups.

\begin{cor}\label{cor:radEP}
Let $\xi = (P,\delta,t) \in \Xi_{un}$ and let $\chi \in G_\C \big( \mr{Mod}_{f,un,\mc W \xi}(\mc S) \big)
\subset G_\C (\mc H)$. Then $EP_{\mc H} (\chi,\chi) = 0$ if and only if $\chi \in
\sum_{P \subset Q \subset F_0, \R Q \neq \mf a^*} \mr{Ind}_{\mc H^Q}^{\mc H} (G_\C (\mc H^Q))$.
\end{cor}
\begin{proof}
The formula \eqref{eq:22} says (by definition) that $EP_{\mc H}(\chi,\chi')$ is the elliptic
pairing $e_{\mf R_\xi^*}(\Phi_\xi (\chi'),\Phi_\xi (\chi))$ of trace functions on $\mf R_\xi^*$
with respect to the $\mf R_\xi^*$-representation $T_\xi (\Xi_{un})$, in the sense of Reeder
\cite{Ree}. By \cite[Proposition 2.2.2]{Ree} the radical of the Hermitian form $e_{\mf R_\xi^*}$
is $\sum_\Gamma \mr{Ind}_{\Gamma}^{\mf R_\xi^*} (G_\C (\Gamma))$, where the sum runs
over all subgroups $\Gamma \subset R_\xi^*$ for which $T_\xi (\Xi_{un})^\Gamma \neq 0$.
Clearly it suffices to consider $\Gamma$'s that contain the central subgroup $Z_\xi =
\ker (\mf R_\xi^* \to \mf R_\xi)$. From \eqref{eq:22} we see that $\mf R_\xi^*$-representations
with different $Z_\xi$-character are orthogonal for $e_{\mf R_\xi^*}$, so the above remains
valid if we restrict to $p G_\C (\mf R_\xi^*) = G_\C (\C [\mf R_\xi,\kappa_\xi])$. In particular
\begin{equation}\label{eq:ell=0}
e_{\mf R_\xi^*} (\Phi_\xi (\chi),\Phi_\xi (\chi)) = 0 \; \Longleftrightarrow \;
\Phi_\xi (\chi) \in \sum\nolimits_{\Gamma^*} \mr{Ind}_{p \C [\Gamma^*]}^{p \C [\mf R_\xi^*]}
\big( G_\C (p \C [\Gamma^*]) \big) ,
\end{equation}
where $Z_\xi \subset \Gamma^* \subset \mf R_\xi^*$ and $T_\xi (\Xi_{un})^{\Gamma^*} \neq 0$.
Since $\mf R_\xi$ is built from elements of the Weyl groupoid $\mc W ,\;
T_\xi (\Xi_{un})^{\Gamma^*}$ is $\mc W_\xi$-conjugate to $\mf a^Q$ for some set of simple roots
$Q \supset P$. In view of \eqref{eq:3} and \eqref{eq:RxiQ} this means that $\Gamma^* / Z_\xi$
is conjugate to a subgroup of $\mf R_\xi^Q$. Thus the right hand side of \eqref{eq:ell=0} becomes
\[
\Phi_\xi (\chi) \in \sum_{P \subset Q \subset F_0, \mf a^Q \neq 0}
\mr{Ind}_{\C \big[ \mf R_\xi^Q,\kappa_\xi^Q \big]}^{\C [\mf R_\xi,\kappa_\xi]}
\big( G_\C (\C [\mf R_\xi^Q,\kappa_\xi^Q]) \big) .
\]
Theorem \ref{thm:KnappStein}.d implies that
\[
\Phi_\xi^Q : G_\C \big( \mr{Mod}_{f,un,\mc W^Q \xi}(\mc S (\mc R^Q,q)) \big) \to
G_\C (\C [\mf R_\xi^Q,\kappa_\xi^Q])
\]
is bijective, which together with \eqref{eq:indPhi} allows us to 
rephrase \eqref{eq:ell=0} in $G_\C (\mc H)$ as
\[
EP_{\mc H} (\chi,\chi) = 0 \; \Longleftrightarrow \;
\chi \in \sum_{P \subset Q \subset F_0, \mf a^Q \neq 0} \mr{Ind}_{\mc H^Q}^{\mc H}
\big( G_\C \big( \mr{Mod}_{f,un,\mc W^Q \xi}(\mc S (\mc R^Q,q)) \big) \big) .
\]
Finally we note that the condition $\mf a^Q \neq 0$ is equivalent to $\R Q \neq \mf a^*$.
\end{proof}
\vspace{1mm}

\section{The case of reductive $p$-adic groups}
\label{sec:padic}

Here we discuss how the proofs of our main results can be adjusted so that they apply
to tempered representations of reductive groups over local non-archimedean fields.

Let $L$ be a reductive $p$-adic group, let $\mc H (L)$ be its Hecke algebra and let Mod$(\mc H (L))$
be the category of smooth $L$-representations. Let $K \subset L$ be any compact open subgroup and
consider the subalgebra $\mc H (L,K)$ of $K$-biinvariant functions in $\mc H (L)$. According to
\cite[Section 3]{BeDe} there exist arbitrarily
small "good" compact open $K$ such that Mod$(\mc H (L,K))$ is equivalent to the category
consisting of those smooth $L$-representations that are generated by their $K$-invariant vectors.
Here the adjective good means that the latter category is a Serre subcategory of Mod$(\mc H (L))$.
It also is known from \cite{BeDe} that these subcategories exhaust Mod$(\mc H(L))$.

The tempered smooth $L$-representations are precisely those smooth representations that extend
in a continuous way to modules over the Harish-Chandra--Schwartz algebra $\mc S (L)$. 
All extensions of admissible $\mc S(L)$-modules can be studied with subalgebras $\mc S (L,K)$ 
of $K$-biinvariant functions, see \cite{ScZi1,OpSo3}. Clearly \eqref{eq:FourierIso} is similar 
to the Plancherel isomorphism for the $\mc S (L)$ \cite{ScZi2,Wal}. One can easily deduce from 
\cite{Wal} that the subalgebra $\mc S (L,K)$ has exactly the same shape as \eqref{eq:25}, see 
\cite[Theorem 10]{Sol-Chern}.

Suppose that $V \in \mr{Mod}(\mc S (L))$ is admissible. By comparing explicit projective 
resolutions of $V$ as an $\mc H (L)$-module and as an $\mc S(L)$-module, it is shown 
in \cite[Proposition 4.3.a]{OpSo3} that
\begin{equation}\label{eq:7.8}
\mr{Ext}_{\mc H (L)}^n (V,V') \cong \mr{Ext}_{\mc S (L)}^n (V,V')
\quad \text{for all} \quad V' \in \mr{Mod}(\mc S (L)) .
\end{equation}
Here we work in the category of all modules over $\mc H (L)$ or $\mc S (L)$, as advocated 
by Schneider and Zink \cite{ScZi1,ScZi2}. Assume that $V$ and $V'$ are generated by their
$K$-invariant vectors for some good compact open subgroup $K$. Then \cite[Proposition 4.3.b]{OpSo3} 
says that \eqref{eq:7.8} is also isomorphic to $\mr{Ext}_{S (L,K)}^n (V^K,{V'}^K)$. In case that 
moreover ${V'}^K$ is a Fr\'echet $\mc S (L,K)$-module, \cite[Proposition 4.3.c]{OpSo3} 
provides a natural isomorphism
\begin{equation}\label{eq:7.9}
\mr{Ext}_{S (L,K)}^n (V^K,{V'}^K) \cong \mr{Ext}_{\mr{Mod}_{Fr}(S (L,K))}^n (V^K,{V'}^K) .
\end{equation}
Here the subscript 'Fr' indicates the category of Fr\'echet modules with exact sequences that
are linearly split. These results can play the role of \cite[Corollary 3.7]{OpSo1}
in the proofs of Corollary \ref{cor:extloc} and Theorem \ref{thm:Ext} for $p$-adic groups.

Alternatively, one can consider $\mc S(L)$ as a bornological algebra. It is natural to endow 
$\mc S (L)$ with the precompact bornology and $\mc H (G)$ with the fine bornology, see 
\cite{Mey-Ho}. For all bornological $\mc S (L)$-modules $V,V'$ and all $n \in \Z_{\geq 0}$
\begin{equation}\label{eq:Meyer}
\mr{Ext}_{\mr{Mod}_{bor}(\mc H (L))}^n (V,V') \cong 
\mr{Ext}_{\mr{Mod}_{bor}(\mc S (L))}^n (V,V')
\end{equation}
by \cite[Theorem 21]{Mey-Ho}. It follows quickly from the definition of 
the fine bornology \cite[pp. 364--365]{Mey-Ho} that
\[
\mr{Ext}_{\mc H (L)}^n (V,V') \cong \mr{Ext}_{\mr{Mod}_{bor}(\mc H (L))}^n (V,V')
\quad \text{if } V \text{ is admissible.}
\]
Let $P$ be a parabolic subgroup and $M$ a Levi factor of $P$.
Let $\sigma$ be an irreducible smooth unitary $M$-representation which is square-integrable
modulo the center of $M$. The $L$-representation $\mathcal{I}_P^L (\sigma)$, the
smooth normalized parabolic induction of $\sigma$, plays the role of $\pi (\xi)$.
Choose good compact open subgroups $K_i \subset L$
such that $\mathcal{I}_P^L (\sigma)$ is generated by its $K_i$-invariant vectors. We may
assume that the $K_i$ decrease when $i \in \N$ increases and that $\bigcap_i K_i = \{1\}$,
so $\mc S (L) = \bigcup_i \mc S (L,K_i)$. Harish-Chandra's Plancherel isomorphism for
$L$ shows that $\mc S (L)$ contains a direct summand
\begin{equation}\label{eq:SLsigma}
\mc S (L)^\sigma = \bigcup\nolimits_i \mc S (L,K_i)^\sigma
\end{equation}
which governs all tempered $L$-representations in the block determined by $\sigma$.
Moreover the algebras $\mc S (L,K_i )^\sigma$ are nuclear Fr\'echet of the form
\eqref{eq:3.1}, and they are all Morita equivalent. The algebra $\mc S(L)^\sigma$ is
Morita equivalent to $\mc S (L,K_i )^\sigma$ via the bimodules $\mc S(L)^\sigma e_{K_i}$
and $e_{K_i} \mc S(L)^\sigma$, where $e_{K_i} \in \mc H (L)$ is the idempotent
associated to $K_i$. However, $\mc S(L)^\sigma$ is not a Fr\'echet algebra, only an
inductive limit of Fr\'echet algebras. 

All the algebras in \eqref{eq:SLsigma} have the same center, which by \cite{Wal} is isomorphic
to $C^\infty (T)^{\mc W}$ for a suitable compact torus $T$ and a finite group $\mc W$.
The representation $\sigma$ determines a point of $T$ and an orbit $\mc W \sigma
\subset T$. Let $m^\infty_{\mc W \sigma} \subset C^\infty (T)^{\mc W}$ be the ideal
of functions that are flat at $\mc W \sigma$. We define
\begin{align*}
& \widehat{\mc S (L)}_{\mc W \sigma} =
\mc S (L)^\sigma / \overline{m^\infty_{\mc W \sigma} \mc S (L)^\sigma} , \\
& \mr{Mod}_{bor}^{\mc W \sigma, tor}(\mc S (L)) =
\{ V \in \mr{Mod}_{bor}(\mc S (L)) : m^\infty_{\mc W \xi} V = 0 \} .
\end{align*}
All the results from Sections \ref{sec:Ext} and \ref{sec:alginv} also hold for these objects,
with some obvious changes of notation. We have to be careful only when we want to
apply Theorem \ref{thm:Moritaeq} to $\mc S (L)^\sigma$. The complication is that over there
we work with a finite dimensional representation $V$ of some finite
group $G$, which in the case under consideration is a central extension $\mf R_\sigma^*$
of the R-group $\mf R_\sigma$ from \cite{Sil}. That is enough for $\mc S (L,K_i )^\sigma$,
but for $\mc S (L)^\sigma$ we are naturally lead to the infinite dimensional module
$V = \mc I_P^L (\sigma)$. Then we must replace the bimodules \eqref{eq:bimodules} by
\begin{equation}\label{eq:moduleslim}
\begin{array}{lll}
D_1 & = & \lim\limits_{i \to \infty} C^\infty (T) \otimes \mr{Hom}_\C
\big(\mc I_P^L (\sigma)_{K_i},V' \big), \\
D_2 & = & \lim\limits_{i \to \infty} C^\infty (T) \otimes \mr{Hom}_\C
\big( V',\mc I_P^L (\sigma)^{K_i} \big),
\end{array}
\end{equation}
where the subscript $K_i$ means coinvariants and the superscript $K_i$ means invariants.
In view of \eqref{eq:SLsigma}, the proof of Theorem \ref{thm:Moritaeq} goes through.

The proof of Theorem \ref{thm:Ext} relies on two deep results: the Knapp--Stein linear
independence theorem (Theorem \ref{thm:KnappStein}) and the Plancherel isomorphism for
$\mc S$ \eqref{eq:FourierIso}. These also hold for the algebras $\mc S(L), \mc S (L)^\sigma$
and $\mc S (L,K_i)^\sigma$, as we indicated above, so our proof remains valid.
We obtain equivalences of exact categories
\begin{equation}\label{eq:equivCat}
\mr{Mod}^{\mc W \sigma,tor}_{bor}\big( \mc S (L) \big) \cong
\mr{Mod}_{bor} \big( \widehat{\mc S (L)}_{\mc W \sigma} \big)
\cong \mr{Mod}_{bor} \big( p_\sigma (\widehat{S (E_\sigma)} \rtimes \mf R^*_\sigma) \big) .
\end{equation}
We note however that these equivalences do not preserve Fr\'echet modules, essentially because
$\mc S (L)$ is too large to admit enough such modules. The first part of \eqref{eq:equivCat} does
not change the modules, the second part comes from a Morita equivalence of bornological algebras.
To find the Morita bimodules, we start with \eqref{eq:moduleslim}. From the proof of
Theorem \ref{thm:Moritaeq} we see that we must take $\mf R_\sigma^*$-invariants and that we
have to replace $C^\infty (T)$ by a formal power series ring, which is none other than
$\widehat{S (E_\sigma)}$. So the functor from left to right in \eqref{eq:moduleslim} is given
by the bornological tensor product with
\begin{equation}
D = \lim_{i \to \infty} \Big( \widehat{S (E_\sigma)} \otimes \mr{Hom}_\C
\big( \mc I_P^L (\sigma)_{K_i},p_\sigma \C [\mf R_\sigma^*] \big) \Big)^{\mf R_\sigma^*}
\end{equation}
over $\widehat{\mc S (L)}_{\mc W \sigma}$.
For the opposite direction we can tensor with the bimodule
\begin{equation}
D^\vee = \lim_{i \to \infty} \Big( \widehat{S (E_\sigma)} \otimes \mr{Hom}_\C
\big( p_\sigma \C [\mf R_\sigma^*], \mc I_P^L (\sigma)^{K_i} \big) \Big)^{\mf R_\sigma^*}
\end{equation}
over the algebra
\[
p_\sigma \big( \widehat{S (E_\sigma)} \rtimes \mf R^*_\sigma \big) \cong \big( \widehat{S (E_\sigma)} 
\otimes \mr{End}_\C (p_\sigma \C [\mf R_\sigma^*]) \big)^{\mf R_\sigma^*} = \big( \widehat{S (E_\sigma)}
\otimes \mr{End}_\C (\C [\mf R_\sigma ,\kappa^{-1}_\sigma]) \big)^{\mf R_\sigma^*} .
\]
To compute the Ext-groups we need the fundamental result \eqref{eq:Meyer}. Using that the
proof of Theorem \ref{thm:Ext} goes through, thus establishing Theorem \ref{thm:2}.

Now it is clear that our proof of \ref{thm:ArthurFormula} is also valid for $L$.
To formulate the result, consider the real Lie algebra $\mr{Hom}(X^* (M), \R)$
of the center of $M$. The group $\mf R_\sigma$ acts on $\mr{Hom} (X^* (M), \R)$
and we denote by $d(r)$ the determinant of the linear transformation $1 - r$.
Let $\pi$ be a finite length unitary tempered $L$-representation all whose irreducible constituents
occur in $\mathcal{I}_P^L(\sigma)$ and let $\rho = \textup{Hom}_L (\pi,\mathcal{I}_P^L(\sigma))$
be the projective $\mf R_\sigma$-representation associated to it via the Knapp--Stein Theorem. Then
\begin{equation}\label{eq:7.1}
EP_L (\pi,\pi^\prime) = |\mf R_\sigma|^{-1} \sum_{r \in \mf R_\sigma} |d(r)| \,
\mr{tr}_\rho (r) \, \overline{\mr{tr}_{\rho'} (r)} .
\end{equation}
To relate this to Kazhdan's elliptic pairing we need some additional properties
of the Euler--Poincar\'e pairing in Mod$(\mc H (L))$. Recall that $G_\C (L)$ is the Grothendieck
group of the category of admissible $L$-representations, tensored with $\C$. Since the elliptic
pairing $e_L (\pi,\pi')$ depends only on the traces of $\pi$ and $\pi'$, it factors via the canonical
map from admissible representations to the Grothendieck group. By standard homological algebra
$EP_L$ has the same property. We extend pairings to $G_\C (L)$ by making them conjugate linear
the first argument and linear in the second.

\begin{lem}\label{lem:7.1}
Suppose that the center $Z(L)$ of $L$ is compact. Then $G_\C (L)$ is spanned by
the union of all irreducible tempered $L$-representations and all
representations parabolically induced from proper Levi subgroups of $L$.
\end{lem}
\begin{proof}
This follows from the Langlands classification, for which we refer to \cite[Section XI.2]{BoWa}
and \cite{Kon}. The argument is analogous to the proof of Lemma \ref{lem:Elltemp}.
\end{proof}

The next result is known from \cite[Lemma III.4.18]{ScSt},  but the authors found it useful to
have a simpler proof that does not depend on the projective resolutions constructed in \cite{ScSt}.

\begin{prop}\label{prop:7.2}
\enuma{
\item Let $P$ be a parabolic subgroup of $L$ with Levi factor $M$, such that $Z(M)$ is not
compact. Then $\mc I_P^L (G_\C (M))$ lies in the radical of $EP_L$.
\item $EP_L$ is a Hermitian form on $G_\C (L)$.
\item $EP_L$ is positive semidefinite and its radical is $\sum_{P,M} \mc I_P^L (G_\C (M))$,
where the sum runs over $P$ and $M$ as in part (a).
}
\end{prop}
\begin{proof}
(a) Since we do not yet know that $EP_L$ is Hermitian, we have to deal with both its left
and its right radical. According to \cite[Claim 4.3]{Bez}, an elementary argument which
Bezrukavnikov ascribes to Bernstein,
\begin{equation}\label{eq:7.2}
EP_L = 0 \text{ if } Z(L) \text{ is not compact.}
\end{equation}
Let $V,W$ be smooth $L$-representations and $V',W'$ admissible smooth $M$-representations.
By Frobenius reciprocity
\begin{equation}\label{eq:7.3}
\mr{Ext}^n_L (V, \mc I_P^L (W')) \cong \mr{Ext}^n_M (r_P^L (V),W') ,
\end{equation}
where $r_P^L$ denotes Jacquet's restriction functor. As $M$ is a proper Levi subgroup of $L$,
its center is not compact, so \eqref{eq:7.2} and \eqref{eq:7.3} show that $\mc I_P^L (W')$
is in the right radical of $EP_L$. Now we could use Bernstein's second adjointness theorem to
reach to same conclusion for the left radical, but we prefer to do without that deep result.

Let $\tilde V$ denote the contragredient representation of $V$, that is, the smooth part of
the algebraic dual space of $V$. The functor $V \mapsto \tilde V$ is exact and for admissible
representations $\tilde{\tilde V} \cong V$. Suppose that $V,W$ are admissible and that
\begin{equation}\label{eq:7.4}
0 \to W \to V_n \to \cdots \to V_1 \to V \to 0
\end{equation}
is an $n$-fold extension in Mod$(\mc H (L))$. The contragredience functor
yields $n$-fold extensions
\begin{align}
\label{eq:7.5} & 0 \to \tilde V \to \tilde V_1 \to \cdots \to \tilde V_n \to \tilde W \to 0 , \\
\label{eq:7.6} & 0 \to W \to \tilde{\tilde V}_n \to \cdots \to \tilde{\tilde V}_1 \to V \to 0 .
\end{align}
The existence of a natural inclusion $V_j \to \tilde{\tilde V}_j$ means that \eqref{eq:7.6} is equivalent to
\eqref{eq:7.4} in the sense of Yoneda extensions. Hence the functors $V_j \mapsto \tilde{\tilde V}_j$
and $V_j \mapsto \tilde V_j$ induce isomorphisms between the corresponding Yoneda Ext-groups.
Since Mod$(\mc H (L))$ has enough projectives, we may also phrase this with the derived functors
of $\mr{Hom}_L$:
\begin{equation}\label{eq:7.7}
\mr{Ext}^n_L (V,W) \cong \mr{Ext}^n_L (\tilde W, \tilde V) .
\end{equation}
By \eqref{eq:7.7}, \cite[Proposition 3.1.2]{Cas} and \eqref{eq:7.3}
\begin{multline*}
\mr{Ext}^n_L (\mc I_P^L (V'),W) \cong
\mr{Ext}^n_L \big( \tilde W,  \widetilde{\mc I_P^L (V')} \big) \cong
\mr{Ext}^n_L \big( \tilde W, \mc I_P^L (\tilde{V'}) \big) \cong
\mr{Ext}^n_M \big( r_P^L (\tilde W), \tilde{V'} \big) .
\end{multline*}
Now \eqref{eq:7.2} shows that $\mc I_P^L (V')$ is in the left radical of $EP_L$.\\
(b) By part (a) and Lemma \ref{lem:7.1} it suffices to check that $EP_L$ is symmetric for
tempered admissible $L$-representations. Since this pairing factors via the
Grothendieck group we may moreover replace every representation by its semisimplification.
Such tempered representations $V,W$ are unitary by \cite[Prop III.4.1]{Wal}, which implies
that $V$ (resp. $W$) is isomorphic to the contragredient of the conjugate representation
$\overline V$ (resp. $\overline W$). Using \eqref{eq:7.7} we conclude that
\[
EP_L (V,W) = EP_L (\overline V, \overline W)  =
EP_L \big( \tilde{\overline W},\tilde{\overline V} \big) = EP_L (W,V) .
\]
(c) By \eqref{eq:Meyer} tempered representations with different $Z(\mc S(L))$-characters are
orthogonal for $EP_L$. Thus we only have to check positive definiteness for the Grothendieck group
of finite length unitary $\mc S (L)$-representations with one fixed $Z(\mc S(L))$-character, modulo
the span of representations that are parabolically induced from the indicated Levi subgroups. In this
setting the proof of Corollary \ref{cor:radEP} applies, all the required properties of
R-groups are provided by \cite[Section 2]{Art}. That determines the radical, while \eqref{eq:7.1}
shows that $EP_L$ is positive semidefinite.
\end{proof}

Recall the elliptic pairing $e_L$ on $G_\C (L)$ from \cite{Kaz} and \eqref{eq:elliptic}.

\begin{thm}\label{thm:7.3}
Suppose that the local non-archimedean field underlying $L$ has characteristic 0. Then
$EP_L (\pi,\pi') = e_L (\pi,\pi')$ for all admissible $L$-representations.
\end{thm}
\begin{proof}
According to \cite[Theorem A]{Kaz} and Proposition \ref{prop:7.2}.c the Hermitian forms $e_L$
and $EP_L$ have the same radical. Reasoning as in the proof of Proposition \ref{prop:7.2}.b, it suffices
to check the equality $EP_L (\pi,\pi') = e_L (\pi,\pi')$ for irreducible tempered $L$-representations
$\pi,\pi'$. This follows from \eqref{eq:7.1} and \cite[Corollary 6.3]{Art}.
\end{proof}


\vspace{1mm}


\begin{thebibliography}{99}

\bibitem[Art]{Art} J. Arthur,
``On elliptic tempered characters",
Acta. Math. \textbf{171} (1993), 73--130

\bibitem[BeDe]{BeDe} J.N. Bernstein, P. Deligne,
``Le ``centre" de Bernstein'',
pp. 1--32 in: \emph{Repr\'esentations des groupes r\'eductifs sur un corps local},
Travaux en cours, Hermann, 1984

\bibitem[Bez]{Bez} R. Bezrukavnikov,
``Homological properties of representations of
$p$-adic groups related to geometry of the group at infinity'',
PhD. Thesis, Tel-Aviv University, 1998, arXiv:math.RT/0406223


\bibitem[BoWa]{BoWa} A. Borel, N.R. Wallach,
\emph{Continuous cohomology, discrete subgroups, and
representations of reductive groups}, Annals of Mathematics
Studies \textbf{94}, Princeton University Press, Princeton NJ, 1980

\bibitem[BuKu]{BK} C.J. Bushnell, P.C. Kutzko,
``Types in reductive $p$-adic groups: the Hecke algebra of a cover'',
Proc. Amer. Math. Soc. {\bf 129.2} (2001), 601--607.

\bibitem[Cas]{Cas} W. Casselman,
``Introduction to the theory of admissible representations
of $p$-adic reductive groups'',
preprint, 1995

\bibitem[Che]{Che} C.C. Chevalley,
``Invariants of finite groups generated by reflections",
American J.  Math. \textbf{77.4} (1955), 778--782

\bibitem[CuRe]{CuRe} C.W. Curtis, I. Reiner,
\emph{Representation theory of finite groups and associative algebras},
Pure and Applied Mathematics \textbf{11},
John Wiley \& Sons, 1962

\bibitem[DeOp1]{DeOp1} P. Delorme, E.M. Opdam,
``The Schwartz algebra of an affine Hecke algebra'',
J. reine angew. Math. \textbf{625} (2008), 59--114

\bibitem[DeOp2]{DeOp2} P. Delorme, E.M. Opdam,
``Analytic R-groups of affine Hecke algebras'', arXiv:0909.1227,
2009, to appear in J. reine angew. Math.

\bibitem[Eis]{Eis} D. Eisenbud,
\emph{Commutative algebra with a view to algebraic geometry},
Graduate Texts in Mathematics \textbf{150},
Springer Verlag, 1995

\bibitem[HC]{HC} Harish-Chandra,
``Harmonic analysis on real reductive groups III. The Maass--Selberg
relations and the Plancherel formula",
Ann. of Math. \textbf{104} (1976), 117--201.

\bibitem[Hei]{Hei} V. Heiermann,
``Op\'erateurs d'entrelacement et alg\'ebres de Hecke avec param\`etres d'un groupe
r\'eductif $p$-adique - le cas des groupes classiques'',
Selecta Math. \textbf{17} (2011), doi: 10.1007/s00029-011-0056-0

\bibitem[Hum]{Hum} J.E. Humphreys,
\emph{Reflection groups and Coxeter groups},
Cambridge Studies in Advanced Mathematics \textbf{29},
Cambridge University Press, 1990

\bibitem[Kaz]{Kaz} D. Kazhdan,
``Cuspidal geometry of $p$-adic groups'',
J. Analyse Math. \textbf{47} (1986), 1--36

\bibitem[Kel]{Kel} B. Keller,
``Chain complexes and stable categories",
Manuscripta Math. \textbf{67.4} (1990), 379--417

\bibitem[KnSt]{KnSt} A.W. Knapp, E.M. Stein,
``Intertwining operators for semisimple groups II",
Invent. Math. \textbf{60.1} (1980), 9--84

\bibitem[Kon]{Kon} T. Konno,
``A note on the Langlands classification and irreducibility
of induced representations of $p$-adic groups'',
Kyushu J. Math. \textbf{57} (2003), 383--409

\bibitem[Kopp]{Kopp} M.K. Kopp,
``Fr\'echet algebras of finite type'',
Arch. Math. \textbf{83.3} (2004), 217--228

\bibitem[LePl]{LePl} C.W. Leung, R.J. Plymen,
``Arithmetic aspects of operator algebras'',
Compos. Math. \textbf{77.3} (1991), 293--311

\bibitem[Lus1]{Lus-Gr} G. Lusztig,
``Affine Hecke algebras and their graded version'',
J. Amer. Math. Soc \textbf{2.3} (1989), 599--635

\bibitem[Lus2]{Lus-Unip} G. Lusztig,
``Classification of unipotent representations of simple $p$-adic groups'',
Internat. Math. Res. Notices \textbf{11} (1995), 517--589.

\bibitem[Mac]{Mac} S. Mac Lane,
\emph{Homology},
Grundlehren der mathematischen Wissenschaften \textbf{114},
Springer-Verlag, 1975

\bibitem[MeVo]{MeVo} R. Meise, D. Vogt,
\emph{Einf\"uhrung in die Funktionalanalysis},
Vieweg Studium: Aufbaukurs Mathematik \textbf{62},
Friedr. Vieweg \& Sohn, 1992

\bibitem[Mey1]{Mey-Emb} R. Meyer,
``Embeddings of derived categories of bornological modules'',
arXiv:math.FA/0410596, 2004

\bibitem[Mey2]{Mey-Ho} R. Meyer,
``Homological algebra for Schwartz algebras of reductive $p$-adic groups'',
pp. 263--300 in: \emph{Noncommutative geometry and number theory},
Aspects of Mathematics \textbf{E37}, Vieweg Verlag, 2006

\bibitem[Mor]{Mor} L. Morris,
``Level zero $G$-types'',
Compositio Math. \textbf{118} (1999), 135--157

\bibitem[Oort]{Oort} F. Oort,
``Yoneda extensions in abelian categories",
Math. Ann. \textbf{153} (1964), 227--235

\bibitem[Opd1]{Opd-Sp} E.M. Opdam,
``On the spectral decomposition of affine Hecke algebras'',
J. Inst. Math. Jussieu \textbf{3.4} (2004), 531--648

\bibitem[Opd2]{Opd-ICM} E.M. Opdam,
``Hecke algebras and harmonic analysis",
pp. 1227--1259 in: \emph{Proceedings ICM Madrid 2006 Vol. II},
European Mathematical Society Publishing House, 2006

\bibitem[OpSo1]{OpSo1} E.M. Opdam, M. Solleveld,
``Homological algebra for affine Hecke algebras",
Adv. Math. \textbf{220} (2009), 1549--1601

\bibitem[OpSo2]{OpSo2} E.M. Opdam, M. Solleveld,
``Discrete series characters for affine Hecke algebras and their formal dimensions",
Acta Math. \textbf{105} (2010), 105--187

\bibitem[OpSo3]{OpSo3} E.M. Opdam, M. Solleveld,
``Extension of tempered representations of reductive $p$-adic groups'',
in preparation

\bibitem[Poe]{Poe} V. Po\'enaru,
\emph{Singularit\'es $C^\infty$ en Pr\'esence de Sym\'etrie},
Lecture Notes in Mathematics \textbf{510}, Springer-Verlag, 1976.

\bibitem[Qui]{Qui} D. Quillen,
``Higher algebraic K-theory",
pp. 85--147 in: \emph{Algebraic $K$K-theory, I: Higher $K$K-theories}, 
Lecture Notes in Mathematics \textbf{341}, Springer-Verlag, 1973 

\bibitem[Ree]{Ree} M. Reeder,
``Euler--Poincar\'e pairings and elliptic representations
of Weyl groups and $p$-adic groups'',
Compos. Math. \textbf{129} (2001), 149--181

\bibitem[Roc]{Roc} A. Roche,
``Parabolic induction and the Bernstein decomposition",
Compos. Math. \textbf{134.2} (2002), 113--133.

\bibitem[ScSt]{ScSt} P. Schneider, U. Stuhler,
``Representation theory and sheaves on the Bruhat-Tits building'',
Publ. Math. Inst. Hautes \'Etudes Sci. \textbf{85} (1997), 97--191

\bibitem[ScZi1]{ScZi1} P. Schneider, E.-W. Zink, 
``The algebraic theory of tempered representations of $p$-adic groups. 
Part I: Parabolic induction and restriction'',
J. Inst. Math. Jussieu \textbf{6} (2007), 639--688

\bibitem[ScZi2]{ScZi2} P. Schneider, E.-W. Zink, 
``The algebraic theory of tempered representations of $p$-adic groups. 
Part II: Projective generators'',
Geom. Func. Anal. \textbf{17.6} (2008), 2018--2065

\bibitem[Sil]{Sil} A.J. Silberger,
``The Knapp--Stein dimension theorem for $p$-adic groups'',
Proc. Amer. Math. Soc. \textbf{68.2} (1978), 243--246, and
``The Knapp--Stein dimension theorem for $p$-adic groups. Correction"
Proc. Amer. Math. Soc. \textbf{76.1} (1979), 169--170

\bibitem[Slo]{Slo} K. Slooten,
``Induced discrete series representations for Hecke algebras of types
$B\sp {\rm aff}\sb n$ and $C\sp {\rm aff}\sb n$'',
Internat. Math. Res. Notices \textbf{10} (2008), doi: 10.1093/imrn/rnn023

\bibitem[Sol1]{Sol-Chern} M. Solleveld,
``Some Fr\'echet algebras for which the Chern character
is an isomorphism'', K-Theory \textbf{36} (2005), 275--290

\bibitem[Sol2]{Sol-Thesis} M. Solleveld,
\emph{Periodic cyclic homology of affine Hecke algebras},
PhD Thesis, Universiteit van Amsterdam, 2007, arXiv:0910.1606

\bibitem[Sol3]{Sol-Irr} M. Solleveld
``On the classification of irreducible representations of affine
Hecke algebras with unequal parameters",
arXiv:1008.0177, 2010, to appear in Representation Theory

\bibitem[Tou]{Tou} J-C. Tougeron,
\emph{Id\'eaux de Fonctions Diff\'erentiables},
Springer Verlag, 1972

\bibitem[ToMe]{TouMer} J-C. Tougeron, J. Merrien,
``Id\'eaux de Fonctions Diff\'erentiables II'',
Annales de l'institut Fourier, \textbf{20.1} (1970), 179--233

\bibitem[Vogt]{Vogt} D. Vogt,
``On the functors $\mathrm{Ext}^1 (E,F)$ for Fr\'echet spaces",
Studia Math. \textbf{85} (1987), 163--197

\bibitem[VoWa]{VoWa} D. Vogt, M.J. Wagner,
``Charakterisierung der Quotientenr\"aume von $s$ und eine Vermutung von Martineau",
Studia Math. \textbf{67} (1980), 225--240

\bibitem[Wal]{Wal} J.-L. Waldspurger,
``La formule de Plancherel pour les groupes $p$-adiques (d'apr\`es Harish-Chandra)",
J. Inst. Math. Jussieu \textbf{2.2} (2003), 235--333

\bibitem[Was]{Was} A. Wassermann,
``Cyclic cohomology of algebras of smooth functions on orbifolds'',
pp. 229--244 in: \emph{Operator algebras and applications Vol. I},
London Mathematical Society Lecture Notes \textbf{135},
Cambridge University Press, 1988.

\end{thebibliography}
\end{document}